\documentclass{report}
\usepackage[left=2.5 cm,right=2.5 cm,top=2.5 cm,bottom=2 cm,includeheadfoot]{geometry}

\usepackage{mathrsfs}
\usepackage{amssymb}
\usepackage{amsfonts}
\usepackage{latexsym}
\usepackage{amsthm}
\usepackage{amsmath}
\usepackage[arrow, matrix, curve]{xy}
\usepackage{graphicx}
\usepackage{MnSymbol}
\usepackage{tikz}
\usetikzlibrary{matrix,arrows}
\input xy

\theoremstyle{plain}
\newtheorem{theorem}{\bf Theorem}[section]
\newtheorem{lemma}[theorem]{\bf Lemma}
\newtheorem{corollary}[theorem]{\bf Corollary}
\newtheorem{proposition}[theorem]{\bf Proposition}
\newtheorem{definition}[theorem]{\bf Definition}

\newtheorem{conjecture}[theorem]{\bf Conjecture}

\def\a{\alpha}  \def\cA{{\mathcal A}}       
          
\def\g{\gamma}  \def\cC{{\mathcal C}}       
         
\def\d{\delta}  \def\cE{{\mathcal E}}       
\def\D{\Delta}         
			           
\def\z{\zeta}   \def\cH{{\mathcal H}}       
\def\e{\eta}           
\def\p{\psi}           
        
\def\k{\kappa}         
\def\l{\lambda} \def\cM{{\mathcal M}}       
 \def\cN{{\mathcal N}}       
\def\m{\mu}     \def\cO{{\mathcal O}}       
\def\n{\nu}            
\def\r{\rho}           
\def\s{\sigma}         
                       
\def\t{\tau}           
\def\f{\phi}           
\def\F{\Phi}           
\def\P{\Psi}           
\def\o{\omega}  \def\cX{{\mathcal X}}       
     \def\cY{{\mathcal Y}}

\def\O{\Omega}

\def\Z{{\mathbb Z}}  \def\C{{\mathbb C}}  
 \def\N{{\mathbb N}}  
	\def\Q{{\mathbb Q}}	\def\P{{\mathbb P}}

\newcommand{\Spec}{\mathrm{Spec}}
\newcommand{\Proj}{\mathrm{Proj}}
\newcommand{\Pic}{\mathrm{Pic}}

\def\Sym{\mathrm{Sym}}

\newcommand{\age}{\mathrm{age }}
\newcommand{\Ext}{\mathrm{Ext }}

\newcommand{\fExt}{\mathfrak{Ext}}
\newcommand{\fHom}{\mathfrak{Hom}}
\newcommand{\tor}{\mathrm{tor}}
\newcommand{\Mor}{\mathrm{Mor}}

\newcommand{\tX}{\tilde{X}}

\newcommand{\mg}{\mathcal{M}_g}
\newcommand{\stmg}{\mathcal{M}_g^{\mathrm{st}}}
\newcommand{\Mg}{\overline{\mathcal{M}}_g}
\newcommand{\stMg}{\overline{\mathcal{M}}_g^{\mathrm{st}}}
\newcommand{\mgn}{\mathcal{M}_{g,n}}
\newcommand{\stmgn}{\mathcal{M}_{g,n}^{\mathrm{st}}}
\newcommand{\Mgn}{\overline{\mathcal{M}}_{g,n}}
\newcommand{\stMgn}{\overline{\mathcal{M}}_{g,n}^{\mathrm{st}}}

\newcommand{\stMgnm}{\overline{\mathcal{M}}_{g,n-1}^{\mathrm{st}}}
\newcommand{\stMgnp}{\overline{\mathcal{M}}_{g,n+1}^{\mathrm{st}}}
\newcommand{\mrat}{\mathcal{M}_{0,2g+2}}
\newcommand{\Mrat}{\overline{\mathcal{M}}_{0,2g+2}}
\newcommand{\M}{\overline{\mathcal{M}}}

\newcommand{\Mgg}{\overline{\mathcal{M}}_{g,g}}
\newcommand{\Ngn}{\overline{\mathcal{N}}_{g,n}}
\newcommand{\ngn}{\mathcal{N}_{g,n}}
\newcommand{\Mgm}{\overline{\mathcal{M}}_{g,m}}

\newcommand{\Mgnn}{\overline{\mathcal{M}}_{g,2n}}
\newcommand{\Mgun}{\overline{\mathcal{M}}_{g+n}}

\newcommand{\MG}{\overline{\mathcal{M}}^G}
\newcommand{\MH}{\overline{\mathcal{M}}^H}
\newcommand{\MGgn}{\overline{\mathcal{M}}^G_{g,n}}
\newcommand{\MS}{\overline{\mathcal{M}}^{S_n}}

\newcommand{\hg}{\mathcal{H}_g}

\newcommand{\Hg}{\overline{\mathcal{H}}_g}
\newcommand{\stHg}{\overline{\mathcal{H}}_g^{\mathrm{st}}}
\newcommand{\hgn}{\mathcal{H}_{g,n}}

\newcommand{\Hgn}{\overline{\mathcal{H}}_{g,n}}
\newcommand{\stHgn}{\overline{\mathcal{H}}_{g,n}^{\mathrm{st}}}

\newcommand{\stHgnm}{\overline{\mathcal{H}}_{g,n-1}^{\mathrm{st}}}

\newcommand{\dmg}{\partial\mathcal{M}_{g}}

\newcommand{\dhg}{\partial\mathcal{H}_{g}}

\newcommand{\dirr}{\d_{\mathrm{irr}}} 
\newcommand{\Dirr}{\D_{\mathrm{irr}}} 
\def\i{\iota}
\begin{document}

\begin{titlepage}
\centering
\centering 
\vspace*{\baselineskip} 

\rule{\textwidth}{1.6pt}\vspace*{-\baselineskip}\vspace*{2pt} 
\rule{\textwidth}{0.4pt}\\[\baselineskip] 

{\huge \textbf{Birational Geometry of moduli spaces of pointed curves}}\\[0.2\baselineskip] 

\rule{\textwidth}{0.4pt}\vspace*{-\baselineskip}\vspace{3.2pt} 
\rule{\textwidth}{1.6pt}\\[\baselineskip] 
\vspace{1,5 cm}
{\LARGE -Dissertation-}\\
zur Erlangung des akademischen Grades doctor rerum naturalium \\
im Fach Mathematik, Fachgebiet Algebraische Geometrie\\
\vspace{1,5 cm}
{\Large Humboldt-Universit\"at zu Berlin} \\ \medskip 	
{\Large Mathematisch-Naturwissenschaftliche Fakult\"at } \\ \medskip
{\Large Institut f\"ur Mathematik} \\
\vspace{4cm}

\begin{flushleft}
{\large
eingereicht von: Irene Schwarz \\ \medskip
Doktorvater: Prof. Dr. Gavril Farkas \\ \medskip
Pr\"asidentin der Humboldt-Universit\"at zu Berlin: 
\\Prof. Dr.-Ing. Dr. Sabine Kunst \\ \medskip
Dekan der Mathematisch-Naturwissenschaftlichen Fakult\"at: 
\\Prof. Dr. Elmar Kulke \\ \medskip
Gutachter: Prof. Dr. Gavril Farkas \\
\hspace{2cm} Prof. Dr. Nicola Tarasca \\
\hspace{2cm} Prof. Dr. Alessandro Verra \\ \medskip
eingereicht am 29.04.2020 \\ \medskip
verteidigt am 22.10.2020}

\end{flushleft}

\thispagestyle{empty}
\newpage
\end{titlepage}

\tableofcontents

\chapter*{Introduction}\addcontentsline{toc}{chapter}{Introduction}

The central part of this thesis are the subsequent Chapters 2,3 and 4 which establish
that, for certain values of $g,n$,  the compactified moduli space $\Ngn$ of $n-$nodal curves of geometric genus $g$  is of general type, i.e. maximal Kodaira dimension, and likewise so are quotients by certain groups of the symmetric group $S_n$ of the moduli space $\Mgn$ of genus $g$ curves marked in $n$ distinct points and, perhaps most importantly, the moduli space $\Hgn$ of hyperelliptic genus $g$ curves marked in $n$ distinct points. Each of these results has been submitted for publication as a separate paper (see \cite{sch2,sch3,sch4}. As, however, the preliminary material needed for the proof of each separate result intersects, we have rewritten those parts of our separate papers to obtain, in slightly amplified form, Chapter 1 {\em Preliminaries} of the present thesis.  This first chapter also contains some material from our diploma thesis \cite{sch0}. 

We recall that $\Ngn:= \Mgnn /G$, with $G:=(\Z_2)^n\rtimes S_n$, can either be seen as a quotient of $\Mgnn$ by the subgroup $G$ of $S_{2n}$ or as a subspace of $\Mgun$ contained in its boundary. The first point of view raised certain natural questions on more general quotients which are then treated in Chapter 3. Hyperelliptic curves are a time-honoured subject of complex function theory: In the very early days of the theory of abstract  algebraic curves,  Riemann's count (see e.g. \cite{gh}) established that not every genus $g$ curve is hyperelliptic (i.e. a 2-sheeted branched cover of the rational curve) by counting the number of free parameters for a general genus $g$ curves as $3g -3$ which, for $g>2$, is bigger than $2g-1$, the number of free parameters for (smooth) genus $g$ hyperelliptic curves. Of course, identification of these numbers with the dimension of the corresponding moduli spaces of smooth curves and the realization of their Deligne-Mumford compactifications as  an honest projective variety belongs to a much later age. Still, $\Hgn$ is a most natural subvariety of $\Mgn$, and we consider our result in Chapter 4 below to be the most important one of our thesis.

To put our work in some mathematical and historical perspective, we shall recall a few well known facts. As already mentioned above, in some version, the theory of moduli goes back to the very early work of Riemann, and for $g=1$ (the case of elliptic curves) it is treated in many good classical books. 
The actual construction of the moduli space $\mg$, however, cannot be accomplished by such elementary means and there are a number of precursors and different approaches to the problem. A common feature in all these approaches is to first equip curves with some additional structure which {\em rigidifies} the problem of classification and then, if possible, to realize $\mg$ by quotienting out this additional structure.

An important precursor along this line is the {\em Hurwitz scheme} $\tilde{\cH}_{d,g}$ which parametrizes pairs $(C,\pi)$ of smooth genus $g$ curves $C$ and degree $d$ branched covers $\pi: C \to \P^1$ with only simple branch points. The space  $\tilde{\cH}_{d,g}$ has been shown to be a quasiprojective variety which is irreducible. This has been analyzed in classical works of Klein, Clebsch, L\"uroth and Hurwitz; a classical reference is \cite{c} and a modern one is \cite{mo}, see also the book \cite{hmo}. A useful compactification of the Hurwitz scheme is not trivial, but due to Knudsen and Mumford, see \cite{kmu}, this has been achieved with the space $\cH_{d,g}$  of {\em admissible covers}, see also \cite{hmo}.

For the actual construction of $\mg$, there are three most common approaches which we recall from \cite{hmo}. In the {\em Teichm\"uller approach} one fixes on an $n$-pointed (smooth) genus $g$ curve $(C;x_1, \ldots ,x_n)$ an additional Teichm\"uller structure, which is an isotopy class $[f]$  of an orientation-preserving homeomorphism $f$ of the $n$-pointed curve $C$ to a fixed $n$-pointed curve $(\Sigma,p_1, \ldots p_n)$. The notion of an isomorphism $\phi$ for curves with Teichm\"uller structure is the natural one, i.e.  $[\phi \circ \tilde{f}]= [f]$.  
A Teichm\"uller structure $[f]$ rigidifies the pointed curve in the sense that, equipped with this additional structure, the curve admits only the identity as an automorphism. The Teichm\"uller space 
$\mathcal{T}_{g,n}$ as a set is then the set of all isomorphism classes of $n$-pointed genus $g$ curves with Teichm\"uller structure.  A basic theorem due to Bers, see 
\cite{ber}, then states that  $\mathcal{T}_{g,n}$ is actually homeomorphic to a ball in $\C^{3g-3+n}$ and a complex analytic manifold, see also \cite{a}. Then the mapping class group $\Gamma_{g,n}$ acts on $\mathcal{T}_{g,n}$ as a properly discontinuous group of holomorphic transformations and one obtains  
$$\mgn=\mathcal{T}_{g,n}/\Gamma_{g,n}$$
as an analytic space, see the discussion in \cite{acg}. Teichm\"uller theory provides access to the rational cohomology of $\mg$ and $\mgn$, and gives, via the Weil-Petersson metric, an embedding in a projective variety with some of the nice properties of the Deligne-Mumford compactification. We shall not need this, and for further details we refer to the literature, e.g. \cite{w1,w2}, \cite{hl} and the textbooks  \cite{hmo,acg}.

An alternative approach to moduli, also heavily based on analysis and classical function theory, see e.g. \cite{s1-3}, is the {\em Hodge theory approach}, see \cite{hmo}. Here the additional data associated with a smooth genus $g$ curve $C$ is its polarized Jacobian, i.e. the data of a complex vector space of dimension $g$, namely the dual
of $H^0(C,K_C)$, a lattice $\Lambda $, namely the first homology group $H_1(C,\Z)$ and a skew-symmetric form $Q$, namely the intersection pairing. Choosing a symplectic basis of $H_1(C,\Z)$ (w.r. to the symplectic form $Q$) consisting of $a-$cycles and $b-$cycles, we may choose a  basis 
$\o_1, \ldots,\o_g$ of $H^0(C,K_C)$ whose period matrix w.r. to the $a-$cycles is the identity. Then the period matrix $P$ given by integrating the $\o_i$ around the $b_j$ 
is symmetric with positive definite imaginary part in view of the Riemann bilinear relations. Thus the locus $\mathfrak{c}_g$ of such period matrices $P$ forms a locally closed subset of the Siegel upper half-space $\mathfrak{h}_g$, and restricting the quotient of $\mathfrak{h}_g$ by the symplectic group $Sp(2g,\Z)$ (a coarse moduli space $\mathfrak{A}_g$ for Abelian varieties of dimension $g$) to $\mathfrak{c}_g$ gives the moduli space 
$\mg$ as an analytic space. Describing the locus $\mg$ in $\mathfrak{A}_g$ (or $\mathfrak{c}_g$ in $\mathfrak{h}_g$) is called the Schottky problem and has been treated in \cite{adc} and a number of papers in the the field of integrable systems, see e.g. papers of Mulase, Novikov, Shiota referenced in \cite{hmo}. Our thesis is not related to this approach, and we shall only add a remark on the Satake compactification $\tilde{\mg}$ obtained by taking the closure of $\mg$ in the (natural) Satake compactification $\tilde{\mathfrak{A}_g }$, see \cite{sa}. Unfortunately, $\tilde{\mg}$ is not modular, i.e. its boundary points cannot be fitted into families of genus $g$ curves. This might be seen as a serious deficit of
approaches to the moduli problem which, depending heavily on Analysis, are not properly algebraic in character, and this problem only vanishes in the {\em geometric invariant theory (GIT) approach}. It gives $\mg$  not only as a quasiprojective variety, but also 
 provides an explicit modular projective compactification $\Mg$ which is absolutely central for the content of this thesis. Similarly one obtains $\mgn$ and its compactification $\Mgn$.  For details we refer to the literature, e.g. \cite{hmo, acg, mfk} and the original paper \cite{dm}, and we just recall that one of the main technical difficulties is in controlling the quotient  of some locally closed subset of the Hilbert scheme of smooth  curves in some projective space $\P^N$ by the {\em continuous} group $PGL(N+1,\C)$. It is here that one needs geometric invariant theory, while for the discrete groups $\Gamma_{g,n}$ and $Sp(2g,\Z)$ nothing like this is required.

Having sketched the theory of constructing the moduli space $\mg$ and its compactication $\Mg$, we now turn to its explicit description. This is a very different subject, in particular if curves are in addition marked in $n$ different points. 
We recall that  already in 1974,  D. Mumford described aspects of the theory for $\mgn$ in his famous Michigan lectures {\em Curves and their Jacobians}, see the Appendix of \cite{m} for a slightly expanded version. While carefully recalling the classical theory in $g=1$ (including additional $n$ marked points) and including some more recent results in $g=2$, he strongly  emphasizes that nothing like this is known already for $g \geq 3$ and that the charm of the subject is in studying the moduli space $\mgn$ in the higher genus case  {\em without} such explicit knowledge. In the expanded version he includes (very shortly) his 1982 result with Harris in the seminal paper \cite{hm} which proves that $\Mg$ is of general type for $g > 23$ and of non-negative Kodaira dimension for $g=23$, see also \cite{ha} for the case of even genus. We emphasize that the explicit description of $\mg$ which we shall partly describe below in low genus  is (much) simpler than an explicit description of $\mgn$.

In 1989 Eisenbud and Harris have written  a further review on the subject in \cite{eh5}, now including the paper \cite{eh4} which treats $\Mg$ for $g \geq 23$ by using the work of Eisenbud and Harris on limit linear series and their proof of the Brill-Noether theorem, see \cite{eh,eh1,eh2}. Furthermore, in 1998 the book \cite{hmo}
gave a systematic exposition of moduli problems including a unified and expanded version of many of the subjects mentioned above, including an exposition to the relation with geometric invariant theory, see \cite{mfk}. It is very much relevant  for the present thesis, although it does {\em not} contain full amplifications on the seminal paper \cite{hm} - written by Harris jointly with Mumford - on the Kodaira dimension of $\Mg$; in particular, the crucial question of the singularities of  $\Mg$ treated in \cite{hm} is emphasized, its solution is referred to as {\em Mumford's argument} and left out: {\em We will give no details here and simply refer to Mumford's argument in \cite{hm}}, see \cite{hmo}, p. 332. However, the book contains an exposition of the underlying moduli problem in the most drastic terms suitable for a general audience, and this, see \cite{hmo}, pp. 328-329,  we shall partly recall here updated to the current state of research as an appropriate introduction to the much more technical questions we shall shortly have to address:  

{\em  If someone put a gun to your head and demanded that you show him this "general curve of genus 2" that everyone was proving theorems about - in other words, that you write down the equation of a general curve of genus 2 - you would have no problem: you would whip out your pen, write
$$ y^2=x^6 +a_5x^5+ \cdots  +a_0 $$ 
and say "where the $a_i$ are general complex numbers". Likewise, if the challenge were to write down a general curve of genus 3, you could write the equation of of a plane quartic
$$\sum_{i+j\leq 4} a_{ij}x^i y^j =0$$
and again take the coefficients $a_{ij}$ to be general. For genus 4 and 5, there is a similar solution: in each case, the canonical model of a general curve is a complete intersection, and you can just write down a homogeneous quadric and a cubic in four variables (for genus 4) or three homogeneous quadrics in five variables (for genus 5) and once more let the coefficients vary freely.}

While similar constructions become ever more complicated in higher genus, they remain possible at least until genus $14$. In genus $g=6$, the general curve is no longer a complete intersection and one needs Brill-Noether theory. The case of genus $g=7,8,9,10$ increase in complexity, but are still amenable by rather classical methods.
This was shown first by Severi for genus $g\leq 10$ in \cite{sev1} (see also \cite{ac3} for a modern treatment). The results were then improved to $g\leq 13$ by Sernesi in \cite{ser} and Chang-Ran in \cite{cr1} and finally to $g\leq 14$ by Verra in \cite{ve}, but all of this requires much more subtle arguments depending on the development of modern algebraic geometry. For $g=15$ it might no longer be possible to parametrize every general curve with finitely many complex parameters. However it is proven in \cite{cr2, schr} that $\M_{15}$ is rationally connected, i.e. for any two general curves $C_1$ and $C_2$ we can find a rational curve in $\M_{15}$ connecting the points $[C_1]$ and $[C_2]$. For genus $g=16$ it was long believed that the paper \cite{cr3} proved the uniruledness of $\M_{16}$, i.e. that any general curve was contained in some line on $\M_{16}$. A recently discovered mistake, however, casts doubt on this result and $g=16$ might again be considered as open. 
Thus, while cases $16\leq g\leq 21$ are currently open, it is well established in \cite{hm,ha,eh4}, that $\Mg$ becomes of general type for $g>23$. More recently Farkas has shown in \cite{f4} that this can be improved to $g\geq 22$. In this range it is impossible to write down equations with finitely many complex parameters for a general curve. To quote \cite{hmo} again:

{\em What's going on here? Basically, to say that there exists a family of curves, parametrized by an open subset of an affine space, that includes the general curve of genus $g$, is exactly to say that the moduli space $\Mg$ is {\em unirational}, i.e. there is a dominant rational map from a projective space $\P^N$ to $\Mg$.  In particular, it implies that the Kodaira dimension of $\Mg$ is negative; that is, there are no pluricanonical forms on $\Mg$. Thus, one consequence of Theorem (6.59)
\footnote{This Theorem recalls the statements on the Kodaira dimension for $\Mg$ proved in \cite{hm,eh4} as recalled in the text above.}
is the fact that for $g \geq 23$, such a family cannot exist.}

It is in view of this argument that we shall henceforth, somewhat loosely, refer to 
the statement {\em A moduli space $\M$ is of general type} 
as saying that $\M$ is the opposite of being unirational (or uniruled) and, in particular, as not being explicitly describable. We further recall that \cite{hmo} then continue to conjecture 
that $\Mg$ actually has Kodaira dimension $-\infty$ up to $g \leq 22$ which is the utmost borderline allowed by the results known then. This conjecture, while recently proven wrong by \cite{f4},  expresses a general believe that the change from describable to utterly non-describable actually happens in a rather sharp transition. We
are, however, with present techniques and knowledge rather far from controlling the unknown territory from $16 \leq g \leq 21$. 

At this point we also wish to caution the reader that the result in \cite{f4} is in fact more recent than our own papers \cite{sch2,sch3} on which the chapters 2 and 3 of this thesis are based. As \cite{sch2} is already accepted for puplication and \cite{sch3} at least submitted for publication, we have chosen not to update this thesis either. Thus in Chapter 2 and 3 we write as if the Kodaira dimensions of $\M_{22}$ and $\M_{23}$ were still unknown.  We remark, however, that the new result in \cite{f4} depends crucially on the use of special divisors which are only available for $g \geq 22$. Thus all earlier results for
$g <22$ are not affected, and all our tables in this regime still present the current state of the art.

A sharp transition, as conjectured above, becomes rather more evident when one starts adding $n$ distinct marked points to our genus $g$ curves, i.e. when one considers $\Mgn$, and it seems that it gets  even more pronounced if one considers appropriate subvarieties (or quotients) of $\Mgn$: All our results in Chapter 2,3,4 actually show this phenomenon, and for hyperelliptic curves it is actually near perfect: The Kodaira dimension of $\Hgn$ jumps from being equal to $- \infty$ for
$n \leq 4g+5$ to a non-negative value for $n=4g+6$, and it is definitely maximal for
$n \geq 4g+7$. Unfortunately, we do not yet have a similarly sharp result for $\Mg$.

The general intuition, of course, is that adding marked points makes the moduli space more difficult to describe and thus pushes $\Mgn$ to being of general type by increasing $n$. A. Logan was first to prove such results in \cite{l}, and these were subsequently refined in \cite{f}. 
Adapting Mumford's argument referred to above to the case at hand is, again, a crucial part of the proof. Fortunately, Mumford's argument may here be taken as a blackbox and then be adapted by a relatively short argument which is none too painful, and a similar approach works in our Chapters 2 and  3 below. 
Unfortunately, this is no longer true for Chapter 4, and here we have to go back to the details of Mumford's argument in painful detail, some of which is already included in our Chapter 1 on {\em Preliminaries}. We recall that, since 2011, the book
\cite{acg} presents a lot of results in the theory of algebraic curves, many of which   are relevant to the subject of this thesis, in a unified way  while before they were only accessible through the original research papers. Mumford's argument is not included in this selection, and one has to go back to the original paper.

We shall finally come to a point which caused us some pain and which resulted in a certain error in our first version of the papers forming the content of Chapter 2 - 4
below. Here it is always necessary to compute the canonical class $K_{\M}$ for the moduli space $\M$ at hand. As we shall shortly explain in our {\em Preliminaries}, there is a huge difference between {\em fine} and {\em coarse} moduli spaces; only the 
former carry {\em universal families}, and these are a godsend for actual calculations. Unfortunately, none of the moduli spaces we consider here actually are fine, and one would be stuck if not for the invention (or existence, depending on ones prefered philosophy of Mathematics) of something called {\em stacks}: On these marvellous objects, universal families magically spring into existence and one may calculate as if everything were {\em fine}. Instead of setting us the anyhow hopeless task of explaining this in full technical detail within the framework of a pleasant-to-read introduction, we  quote from {\cite{b}, pp. 33 -34, in rather informal language:

{\em The mathematical definition of the notion of stack is a stroke of genius, or a cheap cop-out, depending on your point of view: one simply declares the problem to be its own solution!

The problem we had set ourselves was to describe all continous families of triangles.
\footnote{In our case, this would be a family of curves, and {\em continuous} would not be sufficient; there are further conditions on the family, but as we are not aiming at precicision on a technical level anyhow, we might safely ignore such fine points for the time being and simply replace  {\em triangle} by {\em curve} (I.S.).}
We saw that this problem would be solved quite nicely by a universal family, if there was one. 
But there isn't one. So, instead of trying to single out one family to rule all others, we consider {\em all} parameter spaces to {\em be} the moduli stack of triangles.

Thus, the notions of {\em moduli problem} and {\em stack} become synonymous.

The challenge is then to develop techniques for dealing with such a stack as a geometric object, as if it were a space.}\\
\ \\
Taking the last point for granted, one then is rewarded with a phantastic set of tools for actual calculations, including the above mentioned universal family. Control of this object is really empowering, and the very  explicit reference to the famous literary work of Tolkien does not seem to be misplaced.

Unfortunately,  stacks are mathematically defined as certain {\em categories}, and for everyone with some rudimentary training in formal Mathematics it is clear that a Pandora's box of set theoretic problems will spring already from this. At the very least, the category of stacks becomes a bicategory, and similarly many useful things, like universal families for stacks and Cartesian squares start carrying a subscript 2 being objects of level 2, somewhat similar to Russel's and Whitehead's theory of types in the {\em Principia Mathematica}; of course, no working mathematician interested in solving concrete geometrical problems wants to actually have close contact with any of this, and thus the subscripts are rapidly suppressed. In this thesis  we have chosen not to part with tradition at this point. We simply assume that all these problems have been satisfactorily solved and refer anyone interested in more to the more specialized literature, see for instance \cite{dm} which started it all by proving the very important irreducibility of $\Mg$ in arbitrary characteristic by using stacks. Over $\C$ the irreducibility of $\mg$ could already be shown by purely analytical methods such as Teichm\"uller theory, and one might also recall the work of Clebsch in \cite{c} on the Hurwitz scheme mentioned above. However, \cite{dm} first introduced the compactification $\Mg$ of $\mg$ (based on the GIT approach)  which is crucial for almost all algebraic considerations of moduli of curves.  Other references for stacks are \cite{ar, agv}, the monograph \cite{lmb}, the introductory article \cite{b} from which we cited above, the long review article \cite{v} and furthermore the collected works of A. Grothendieck, see e.g. \cite{g1,g2}, in {\em Seminaire de Bourbaki} for the related {\em theory of descent} which we shall completely bypass as does \cite{b}.

We furthermore recall that the textbook \cite{acg} follows to some extent the approach of \cite{dm} in its treatment of moduli (complementing it e.g. by the Teichm\"uller point of view) and thus, by necessity, has to introduce a certain part of the theory of stacks in Chapter XII.  In a somewhat informal way, treating stacks very much as a space (and not as a mere category), it is shown here that the moduli stack $\stMgn$ is actually ramified over the associated coarse moduli space $\Mgn$ over certain boundary divisors (due to the appearance of automorphisms on certain curves), and the ensuing ramification divisor  actually provides the correction term to pass from the canonical class on the stack to the canonical class on the coarse moduli space. In our theorems we are concerned with the Kodaira dimension of coarse moduli spaces, and due to some ambiguity in the existing literature we at first were misled to neglect this  conceptually very important distinction between stack and coarse moduli space at some points. As a rule of thumb, calculations should always first be performed on the stack
(treating it as some form of space) and then ramification over the coarse moduli space can be read off from the automorphism group of curves. We hope that by now all our formulae are correct. Fortunately, all these fine points about ramification divisors did not all affect the outcome of our calculation; that might be a saving grace of fine points.

We close this rather chatty introduction by giving the outline of this thesis. In Chapter 1 we recall all the joint preliminary technical material needed for the subsequent Chapters. Concerning the stacky parts, we have chosen to follow the somewhat pragmatic exposition in \cite{acg}. Chapter 2 contains our result on the moduli space $\Ngn= \Mgnn /G$, with $G:=(\Z_2)^n\rtimes S_n$,  of nodal curves while
 Chapter 3 proves related results for quotients of $\Mgn$ by a larger class of subgroups $G$ of $S_n$. Chapter 4 then contains what we consider as the 
most important achievement of this thesis, namely the proof that the moduli space
$\Hgn$ of hyperelliptic curves is of general type for certain values of $g,n$ and 
that the transition from uniruled to general type is as sharp as one could possibly hope for in a first result.

\chapter{Preliminaries}
\setcounter{equation}{0}

In this thesis we understand by a curve a complete reduced algebraic curve over $\C$.  In particular, it may be singular. We feel free to view a curve alternatively as a complex variety
or a scheme (when we refer to constructions which are standard in the more general category of schemes). We remark that as a general rule we feel free to use any definition from \cite{h} without further explanation. More advanced concepts will be explicitly referenced and sometimes amplified.

\section{Moduli spaces}\label{sec moduli}

A moduli space is a way to provide a set of objects with the structure of scheme. For utmost generality, one does not start with  a set of objects but rather with a functor from the category of schemes to the category of sets (called a moduli problem). The abstract definition is (see \cite{hmo}):

\begin{definition}
Let $M$ be a scheme. Then the \emph{ functor of points} of $M$ is the contravariant functor $\Mor_M$ from the category of schemes to the category of sets defined by sending a scheme $B$ and a morphism of schemes $\f: B\to C$ to:
\begin{equation*}
\begin{split}
 &\Mor_M(B) :=\Mor_{\mathrm{sch}}(B,M)=\{f:B\to M|\mbox{ f is a morphism of schemes}\} \\
 &\Mor_M(\f): \Mor_M(C)\to \Mor_M(B), \quad f\mapsto f \circ\f
\end{split}
\end{equation*}
\end{definition}

We remark that for any
morphism $\phi: M\to N$ between schemes there is a canonical  associated natural transformation $\Phi: \Mor_M \to \Mor_N$.

\begin{definition}\label{def fine}
Let $F$:(schemes)$\to$ (sets) be a functor from the category of schemes to the category of sets.
$F$ is \emph{representable} in the category of schemes, if there is a scheme $M$ and a natural isomorphism (of functors from schemes to sets) $\Psi:F \to \Mor_M$ .
If $F$ is representable (by $M$), then $M$ is called a \emph{fine moduli space} for the \emph{moduli problem} $F$.\
\end{definition}   

As  we shall shortly see, existence of a fine moduli space provides a universal family, and this facilitates computations. Unfortunately, fine moduli spaces seldom exist; in particular this is the case for all moduli problems considered in this thesis. In such a case, if a functor $F$ is not representable, there might still be a scheme $M$ satisfying the weaker conditions of a coarse moduli space.

\begin{definition}\label{def coarse}
A scheme $M$ and a natural transformation $\Psi_M : F\to \Mor_M$ are called a \emph{coarse moduli space} for the moduli problem $F$, if
 \begin{enumerate}
 \item The map $\Psi_{\Spec(C)}: F(\Spec(C))\to \Mor_M(\Spec(C))$ is a  bijection of sets.
 \item Given another scheme $N$ and a natural transformation $\Psi_N : F\to \Mor_N$, there is a unique morphism $\phi: M\to N$ such that the associated natural transformation $\Phi: \Mor_M  \to \Mor_N$ satisfies $\Psi_N=\Phi \circ \Psi_M$
 \end{enumerate}
\end{definition}

This definition may not seem very tangible. But in practice one only considers a special kind of moduli problem for which the moduli space becomes more accessible. This moduli problem is constructed as follows. 
In order to define the moduli problem $F$ one requires a set of geometric objects $O$ and a notion of what a family of these objects over a scheme $B$ is. Usually a family over $B$ is understood
\footnote{We use the vague term {\em understood} on purpose: families of $n$-pointed curves are actually defined as tupels consisting of a bundle and $n$ sections into this bundle.}
as a map $\f: \cX\to B$ from some space $\cX$ to $B$ such that the fibres over closed points of $B$ are elements of $O$ and possibly some other regularity conditions are  satisfied. Then one defines the Functor $F$ by $F(B)=\{\mbox{ families of objects of }O\mbox{ over }B \}$ and $F(\f): F(B)\to F(C), f\mapsto \f\circ f$ for schemes $B$ and $C$, a morphism of schemes $\f: B\to C$ and a family $f:\cX\to B$.

With this construction a moduli space is a scheme $M$, whose closed points correspond to the elements of $O$. Furthermore, the structure of $M$ corresponds to the concept of a family (as above) in the sense that a family $\f: \cX\to B$, where the fibres over closed points are elements in $O$, induces a map $\tilde{\f}: B\to M$. Here a closed point $b \in B$ is mapped to the closed point $m \in M$ corresponding to the fibre $\f^{-1}(b) \in O$. 
Thus a fine moduli space gives a bijection  between $F(B)$, the set of families over the base $B$, and $Mor_M(B)$, the set of scheme morphisms from $B$ to $M$. 
For a coarse moduli space the above mentioned map $\f \mapsto \tilde{\f}$ need not be bijective. But the first condition of Definition \ref{def coarse} still requires that the closed points of $M$ correspond to elements of $O$ and the second that $M$ is universal with respect to this property.
In particular, a coarse moduli space is unique up to isomorphism. This justifies the (standard) use of the definite article.

In the construction  above we require  $O$ to be a set. But in many examples one would like to take $O$  as {\em a proper class $\cO$ up to isomorphism}. This requires some set theoretical considerations which we will briefly discuss without going into detail.
Technically one would now need families of isomorphism classes. Instead one usually takes families in $\cO$, where each family is a set but all families will again form a proper class. One then defines isomorphisms between these families that are compatible with the isomorphisms on $\cO$. The families on $O$ are then defined as isomorphism classes of families in $\cO$. Technically one would always have to show that the isomorphism classes (of families in $\cO$) do indeed form a set and the so defined families of isomorphism classes are actually families in the sense above, i.e. bundles $\f: \cX\to B$. This, however, is usually left to the reader and we shall not part with tradition. \\
\ \\
One of the most important properties of fine moduli spaces is the existence of a special family that already describes all possible families.
\begin{lemma}
Let $M$ be a fine moduli space for a moduli problem $F$. Then there exists a unique (up to isomorphism) family $\cC\to M$, called the \emph{universal family}, such that any family $\cX \to S$ in $M$ induces an isomorphism $\cX\simeq S\times_{M} \cC$.
\end{lemma}
For the bijection between families over a base $B$ and morphisms $B\to M$, which we have informally discussed above, the universal family is just the family corresponding to the identity $\mbox{id}_M$. For coarse moduli spaces not every morphism to $M$ corresponds to a family. In particular a coarse moduli spaces possesses a universal family if and only if it is a fine moduli space. 

For all moduli problems considered in this thesis fine moduli spaces do not exist. As already mentioned in our Introduction, our way out is to consider instead the {moduli stacks} associated to the moduli problem $F$ (or which are, in some sense, synonymous with the moduli problem $F$). This is not just an abstract nicety, but it is essential for actual calculations with equalities and inequalities. In this thesis they are crucial for the calculations in Section 4.3. 
On the other hand, for all stacks considered in this thesis (since they are separate Deligne-Mumford of finite type) coarse moduli spaces do exist. In this situation it is quite possible that one scheme $M$ is the coarse moduli space for 2 different moduli spaces and thus 2 different stacks (as both notions are synonymous). This actually happens in the context of this thesis for $M$ being the coarse moduli space $\hg$ of genus $g$ hyperelliptic curves  which is isomorphic as a space to the quotient
$\mrat/S_{2g+2}$, both of which, however, naturally appear in the context of 2 different moduli problems or 2 different stacks. Thus whenever we talk about the associated stack of a moduli space we mean the stack associated to the underlying moduli problem.

Relevant for this thesis are
the moduli spaces  $\mg$, the moduli space of smooth curves of genus g (up to isomorphism), and its Deligne-Mumford compactification $\Mg$, the moduli space of stable curves of genus $g$ (up to isomorphism) 
as a basic input for everything treated in this thesis, and furthermore the moduli spaces $\mgn$ and $\Mgn$ of n-pointed genus $g$ curves and the moduli spaces $\hgn$ and $\Hgn$ of hyperelliptic $n-$pointed genus $g$ curves as well as various quotients of these spaces (or stacks) by subgroups of the symmetric group $S_n$.  We shall quickly explain the moduli problems for these moduli spaces.  We start with the proper class of all smooth (or stable) curves of genus $g$. Since every curve can be embedded into some $\P^r$
it makes sense to speak of the set of isomorphism classes (which is the set $O$ in our abstract discussion above). Next we need the notion of families.

\begin{definition}
\begin{itemize}
\item A \emph{family of smooth curves of genus $g$} is a flat, proper, surjective morphism of schemes (or analytic spaces) $\f : \cX\to B$, such that for every closed point $b\in B$ the fibre $f^{-1} (b)$ is a smooth curve of genus $g$.
\item A \emph{family of stable curves of genus $g$} is a flat, proper, surjective morphism of schemes (or analytic spaces) $\f : \cX\to B$, such that for every closed point $b\in B$ the fibre $f^{-1} (b)$ is a stable curve of genus $g$.
\item  A \emph{family of $n$-pointed smooth (or stable) curves of genus $g$} consists of family $\f : \cX\to B$ of smooth (or stable) curves of genus $g$ and an $n$-tuple $(\s_1,\ldots,\s_n)$ of sections $\s_i: B\to \cX$ such that their graphs are pairwise disjoints.
\end{itemize}
\end{definition}
We also define when such families are isomorphic.
\begin{definition}
Two families of $n$-pointed smooth (or stable) curves $(\f:\cX\to B;\s_1,\ldots,\s_n)$ and $(\p:\cY\to B;\t_1,\ldots,\t_n)$ are said to be \emph{isomorphic}, if there exists an isomorphism of schemes (or analytic spaces) \\
$\F: \cX\to\cY$ such that $\f =\p\circ\F$ and $\s_i=\F^{-1}\circ\t_i$ for all $i$.
\end{definition}

Let us by way of example briefly discuss the set theoretical implications for families of smooth curves. 

Obviously taking all families over a given base $B$ will not yield a set. However any family of smooth curves will be isomorphic to a family $\f:\cX\to B$, where $\cX\subset \mathbb{P}^r\times B$ and $\f$ is the projection to the second component. This means that the set of all isomorphism classes of families will indeed be a set.

Thus we are reduced to showing that the two notions of isomorphism classes, i.e. isomorphism classes of curves and of families, are compatible.
Assume that $\f:\cX\to B$ and $\p:\cY\to B$ are two isomorphic families of smooth curves of genus $g$. Let $b\in B$ be a closed point of $B$ and $\F: \cX\to\cY$ an isomorphism such that $\f =\p\circ\F$, then $\f^{-1}(b)$ and $\p^{-1}(b)$ will be  closed subschemes of $\cX$ and $\cY$. $F\vert_{\f^{-1}(b)}:\f^{-1}(b)\to \p^{-1}(b)$ will be bijective and therefore an isomorphism. This means that $\f^{-1}(b)$ and $\p^{-1}(b)$ define the same isomorphism class of curves of genus $g$.

Thus $\mg,\Mg,\mgn$ and $\Mgn$ are the coarse moduli spaces of the moduli problems 
$$F_{\mg}(B)=\{ \mbox{isomorphism classes of families of smooth curves of genus }g\mbox{ over }B \},$$ 
$$F_{\Mg}(B)=\{ \mbox{isomorphism classes of families of stable curves of genus }g \mbox{ over }B \},$$
$$F_{\mgn}(B)=\{ \mbox{isomorphism classes of families of smooth }n-\mbox{pointed curves of genus }g\mbox{ over }B \}$$ 
and
$$F_{\Mgn}(B)=\{ \mbox{isomorphism classes of families of stable }n-\mbox{pointed curves of genus }g \mbox{ over }B \}.$$

We refer to \cite{acg} for the construction of $\mgn$ and $\Mgn$, the coarse moduli spaces of smooth (resp. stable) n-pointed genus $g$ curves, parametrising tupels $(C,x_1,\ldots, x_n)$ where $C$ is a genus $g$ curve and $x_1,\ldots,x_n\in C$ are $n$ distinct points on $C$. For $n=0$ these moduli spaces are exactly $\mg$ and $\Mg$. Likewise, for the moduli problems and moduli spaces $\hgn$ and $\Hgn$ of hyperelliptic curves we refer to \cite{acg}, but see also our Section 4.2 for a collection of some specific results.


We shall now give a  geometric reason  that the moduli problems of smooth or stable (n-pointed) curves are not representable:  Some curves admit non-trivial automorphisms. (Here non-trivial depends on the moduli problem. For the hyperellptic curves considered in Chapter 4  both the identity and the hyperelliptic involution are considered trivial.) 
These automorphisms both obstruct the existence of fine moduli spaces and (in some cases) cause the coarse moduli spaces to have singularities.

For some calculations it is however very useful to have the properties of fine moduli spaces, in particular the existence of universal families. For this reason we will whenever necessary use stacks (instead of coarse moduli spaces) to perform actual calculations.

Essentially a stack does not only parametrize isomorphism classes of geometric objects but also keeps track of its automorphisms. By forgetting to keep track of these automorphisms our moduli stacks admit a surjective morphism to a scheme (the associated coarse moduli space). This morphism from a stack to its assotiated coarse moduli space is ramified exactly above those points that correspond to geometric objects with non-trivial automorphisms, and the corresponding ramification divisor is a correction term for the canonical class of the stack compared to the canonical class of the space.	Thus, at least in our situation where stacks are locally representable by spaces, the ramification of the stack over the coarse moduli space gives a completely geometric picture of our objects combined with their automorphisms. The corresponding ramification divisor matters.

Instead of trying to address this problem in a self-contained abstract framework (by amplifying the notion of  line bundles on stacks, their morphisms to line bundles on spaces, the question of local representability of stacks by spaces, all necessary to consistently introduce the notion of a ramification divisor on the canonical line bundle of our stacks), we shall, in a purely pragmatic way, come back to this at a later stage, very much similar to the informal way in which the book \cite{acg} discusses the ramification problem for the map
\begin{equation}
\epsilon:  \stMgn \to \Mgn,
\end{equation} 
via the corresponding automorphism groups, see \cite{acg}, p.386.

For completeness sake, we remark that the (bi)category of stacks is small enough that the concept of line bundles over the stack still makes sense. In fact the rational Picard group of any stack considered in this thesis can be identified with the rational Picard group of its associated coarse moduli space. The category of stacks is however large enough that moduli problems for smooth or stable (n-pointed) curves become representable by a smooth stack.

We have chosen to avoid a systematic exposition of the theory of stacks in this thesis, one reason being that we could not find a concise reference covering all the fine points needed in our context. As remarked above, the standard reference \cite{acg} does not do it, and going beyond \cite{acg} in this respect has never been the intended subject of this thesis. We are, however, convinced that,
 once the basic definitions and concepts are understood, at least in our context most computations become straight forward and many theorems for schemes can be transferred to stacks, at least in the pragmatic way sketched above. This might be questionable in a purist approach to Mathematics, but it was our more modest aim to only obtain correct formulae, and we hope to having achieved at least that.
 

\section{The Kodaira dimension}\label{sec Kdim}

In this thesis we study the birational geometry of moduli spaces and in particular their Kodaira dimension. 

Before we define the Kodaira dimension we recall that for any divisor $D$ on a projective variety $X$ with $H^0(X,D)\neq 0$ the complete linear series $|D|$ defines a rational map $\f_D: X \dashedrightarrow  \P H^0(X,D)=\P^r$. With these maps we can define the Iitaka dimension of divisors. As a general reference we refer to \cite{l1,l2}.

\begin{definition}
Let $X$ be a normal irreducible projective variety, $D$ a divisor on $X$ and $\f_m=\f_{mD}$ the rational maps induced by the complete linear systems $|mD|$ whenever
$H^0(X,mD)\neq 0$. 
 We define the \emph{Kodaira-Iitaka dimension} of $D$ as 
\begin{equation}
\k(X,D)=\max_{m\in \N} \dim(\f_m(X))
\end{equation}
if $H^0(X,mD)\neq 0$ for some $m>0$. Otherwise we define $\k(X,D)=-\infty$.\\
When $X$ is not normal we define the Kodaira-Iitaka dimension of a Divisor $D$ on $X$ as the Kodaira-Iitaka dimension of its pullback under the normalisation. 
\end{definition}
In particular any effective divisor $D$ satisfies $H^0(X,D)\neq 0$ and thus $\k(X,D)\geq 0$.

Of special is interest is the Kodaira-Iitaka dimension of the canonical divisor, which gives an important birational invariant of the variety $X$.

\begin{definition}\label{def Kodaira}
Let $X$ be a smooth irreducible projective variety and $K_X$ its canonical divisor. We define the \emph{Kodaira dimension} of $X$ as
\begin{equation}
\k(X):=\k(X,K_X).
\end{equation}
When $X$ is singular we define $\k(X)$ as the Kodaira dimension of a desingularization of $X$ and when $X$ is not projective as the Kodaira dimension of a compactification of $X$.
\end{definition}

Since the singularities and a possible compactification of $X$ only affect sets of codimension at least 1, the Kodaira dimension is well defined. It is a birational invariant. Furthermore,
it is clear from the definition that the Kodaira-Itaka dimension (and therefore the Kodaira dimension of $X$) is always less or equal to the dimension of $X$. A divisor $D$ on $X$ is called \emph{big} if $\k(X,D)=\dim(X)$ is maximal.
We recall further that a divisor $A$ on $X$ is ample if and only if the induced map 
$\f_{mA}:X\to \P H^0(X,mA)$ is an immersion for some $m>0$. Therefore any ample divisor is big. In fact we have the following characterization of big divisors.

\begin{lemma}\label{big}
Let $D$ be a divisor on an irreducible projective variety then the following conditions are equivalent:
\begin{enumerate}
\item $D$ is \emph{big}, i.e. $\k(X,D)=\dim(X)$.
\item For any ample divisor $A$ on $X$ there exists an integer $m>0$ and some effective divisor $E$ such that $mD\sim A+E$, where $\sim$ denotes linear equivalence.
\item For some ample divisor $A$ on $X$ there exists an integer $m>0$ and some effective divisor $E$ such that $mD\sim A+E$.
\end{enumerate}
\end{lemma}
In particular the class of $D$ in the rational Picard group is big if and only if it is the sum of an ample and an effective divisor class.

For completeness sake we remark that, while the Kodaira dimension of a variety $X$ is a birational invariant, the Iitaka dimension of its canonocal class $K_X$ is not, see \cite{l1} for examples of singular varieties with $\k(X,K_X)>\k(X)$. 
We recall that there is an alternative approach to the Kodaira dimension that uses only birational invariants. We shall not prove that both approaches give the same value of $\k(X)$, see \cite{l1}.
\begin{definition}\label{def Kodaira2}
Let $X$ be an algebraic variety, $K$ its canonical class and $m\in\N$. We call
\begin{itemize}
\item $p_m=\dim H^0(X,mK)$ the $m$-th \emph{plurigenus},
\item $R(X,K)=\bigoplus_{m=0}^\infty H^0(X,mK)$ the \emph{canonical ring},
\item $\Proj(R(X,K))$ the \emph{canonical model} and 
\item $\k(X)=\dim \Proj(R(X,K))$ the \emph{Kodaira dimension} of $X$.
\end{itemize}
Here we define $\k(X)=-\infty$ if the plurigenera of $X$ are zero for all $m>0$. 
\end{definition}


Let us consider the following birational properties

\begin{definition} An algebraic variety $X$ is called
\begin{itemize}
\item \emph{rational}, if $X$ is birational to some projective space $\P^n$,
\item \emph{unirational}, if there exists a dominant rational map $\P^n\dashedrightarrow X$ for some $n$,
\item \emph{uniruled}, if there exists a dominant rational map $\P^1\times Y\dashedrightarrow X$ that does not factor through $Y$ for some  variety $Y$.
\end{itemize}
\end{definition}

All these properties imply that the Kodaira dimension is minimal. In fact
\begin{lemma}
Let $X$ be an algebraic variety. Then we have the implications 

$X$ rational $\Rightarrow$ $X$ unirational $\Rightarrow$ $X$ uniruled $\Rightarrow$ $\k(X)=-\infty$.
\end{lemma}

We are, however, mostly interested in the opposite condition of having maximal Kodaira dimension.
\begin{definition}
An algebraic variety $X$ is called \emph{of general type} if the Kodaira dimension is \\
$\k(X)=\dim(X)$.
\end{definition}

It is clear from Definition \ref{def Kodaira} that a smooth projective variety is of general type if and only its canonical class is big. Thus Lemma \ref{big} gives a nessisary and sufficient criterion for a smooth projective variety to be of general type.
This criterion is also valid for a singular variety if the singularities of $X$ are not too bad in the following sense: 
\begin{definition}\label{def adjcond}
We say that the singularities of an algebraic variety $X$ do not impose \emph{adjunction conditions}, if  any pluricanonical form on the regular locus of $X$ can be pulled back to a pluricanonical form on a desingularization of $X$. More precisely if for any resolution of singularities $\r: \tX  \to X$ and any  $m \in \N$ there is an isomorphism
\begin{equation}
\r^*: H^0(X_{\mathrm{reg}}, mK_{X_{\mathrm{reg}}}) \to H^0(\tX,mK_{\tX}).
\end{equation}
Here $X_{\mathrm{reg}}$ denotes the set of regular points of $X$ and $K_{\tX},  K_{(X)_{\mathrm{reg}}}$ denote the canonical classes on $\tX$ and  $X_{\mathrm{reg}}$.
\end{definition}
 In particular, this condition guarantees that the Kodaira dimension $\k(X)$ equals the Kodaira-Iitaka dimension of the divisor class $K_X$. One then has the following criterion which already has been implicitly used in \cite{hm}.

\begin{proposition}\label{criterion gentype}
Let $X$ be an algebraic variety whose singularities do not impose adjunction conditions and let $K$ be the canonical class. Then we have $\k(X,K_X)=\k(X)$.\\
In particular, if $K$ is effective then $\k(X)\geq 0$. Furthermore if $K=A+E$ for some ample rational divisor $A$ and some effective rational divisor $E$ then $X$ is of general type.
\end{proposition}

Finally we recall the followinmg important additivity and subadditivity properties of the Kodaira dimension.
For cartesian products, the Kodaira dimension satisfies the equation
$$\k(X\times Y)=\k(X)+\k(Y).$$
Furthermore,  for more general fibrations the Kodaira dimension is still conjectured to satisfies the following subadditivity condition.
\begin{conjecture}[Iitaka Conjecture]\label{subadditivity}
Let $f : X \to Y$ be a surjective morphism of proper,
smooth varieties over $\C$. Assuming the generic geometric fibre $F$ of $f$ is connected, we have the following inequality for the Kodaira dimension:
\begin{equation}
\k(X)\geq \k(Y)+\k(F)
\end{equation}
\end{conjecture}
This conjecture has been proven for several special cases. In particular when the base $Y$ is of general type (see \cite{ka}) or the general fibre $F$ is of general type (see \cite{ko}).
Thus if both the base $Y$ and the general fibre $F$ are of general type then so is $X$.

\section{Singularities of moduli spaces}\label{sec sing}

In the last section we have given a criterion for a moduli space $M$ to be of general type. This criterion requires us to control the singularities of $M$, more precisely it requires that the singularities do not impose adjunction conditions, see Definition \ref{def adjcond}. We recall that the singularities of moduli spaces occur when the geometric objects parametrized by $M$ have non-trivial automorphisms. Therefore studying the singularities of $M$ boils downs to studying the automorphisms of these objects. 

Harris and Mumford have proven in the seminal paper \cite{hm} that the singularities of $\Mg$ for $g\geq 4$ do not impose adjunction conditions. The proof uses Kodaira-Spencer theory and the Reid-Tai criterion for canonical singularities. It begins with explicit calculations in local coordinates for smooth curves. These calculations show that $\mg$ has only canonical singularities. In a second step Harris and Mumford show that $\Mg$ has non-canonical singularities only in points corresponding to curves with elliptic tails. The final step is to explicitly construct the pullbacks of pluricanonical forms for these curves with elliptic tails.

A. Logan has shown in \cite{l} that the singularities of $\Mgn$ likewise do not impose adjunction conditions, and Farkas and Verra have shown the same for $\Mgn/S_n$, the moduli space of genus $g$ curves with n-marked points. Both proofs heavily rely on the calculations in \cite{hm}. Essentially, they  adjust these calculations to the case at hand. We shall proceed similarly.

In this section we will not reproduce the calculations in \cite{hm}, but we introduce the definitions and theory needed to understand and adjust them. In particular we will introduce the Reid-Tai criterion and some Kodaira-Spencer theory. We start with some definitions.

Let $M$ be any of the  moduli spaces of curves considered in this thesis, i.e. $\mg$, $\Mg$, $\mgn$, $\Mgn$,$\mgn/G$, $\Mgn/G$, $\hg$, $\Hg$, $\hgn$, $\Hgn$. 
Let $d=\mbox{dim}(M)$, let $C$ be a (stable) curve in $M$  and $\pi: \cC\to \Delta^d,$  
where $\Delta^d =\{z \in \C^d| |z|<1 \}$ denotes the open unit ball, be its local universal deformation space
at $[C]$ given by $C$. The automorphism group $\mbox{Aut}(C)$ is a finite group which acts on $\cC$ and $\Delta^d$. Then a neighbourhood of $[C]$ in $M$ is isomorphic to $\Delta^d/\mbox{Aut}(C)$. Furthermore,  the action of a finite group on a smooth space can always be linearised (see e.g. \cite{acg} Chapter XI §6 Lemma 6.12)  and therefore $\mbox{Aut}(C)$ acts as a group of linear maps  on the tangent space $T_{0,\Delta^d}$ of $\Delta^d$ at $0$. This vector space is isomorphic to the space of infinitesimal deformations of $C$ in $M$ which we denote by $T_C M$.
In suitable coordinates the action of an automorphism $\a\in \mbox{Aut}(C)$ of order $m$ on this tangent space, has the form 
\begin{equation}\label{matrix rep}
 \a=\begin{pmatrix}\z^{a_1} & 0 & \cdots & 0 \\0 & \z^{a_2} & 0 & 0 \\\vdots & 0 & \ddots & \vdots \\0 & 0 & \cdots & \z^{a_d} \\\end{pmatrix},
\end{equation}
for $\z$ a primitive $m$-th root of unity and $a_1,\ldots, a_d\in \{0,1,\ldots, m-1\}$.

\begin{definition}\label{def age}
Let $\a\in GL(d)$ be an automorphism of order $m$, conjugate to a matrix of the form \eqref{matrix rep}. 
We define the {\em age of $\a$ with respect to $\z$} as 
$$\age(\a):=\sum_{i=1}^d \dfrac{a_i}{m} .$$
We say that $\a$ is {\em senior}, if $\age(\a)\geq 1$ with respect to every primitive $m$-th root of unity $\z$. Otherwise we call $\a$ {\em junior}.
\end{definition}

The age of an automorphism $\a$ is a useful criterion to decide if $\a$ induces a canonical singularity in the sense of the following definition.

\begin{definition}\label{def cansing}
Let $X$ be a normal variety and assume that it admits a canonical
sheaf $\omega_X$ which is a line bundle.  Then $X$ has (only) {\em canonical singularities} if for a resolution of singularities $\f: \tilde{X} \to X$ one has 
$$\f^* \omega_X \subset \omega_{\tilde{X}}.$$
\end{definition}

It is clear from the definition that canonical singularities do not impose adjunction conditions.
Before we introduce the Reid-Tai criterion for identifying canonical singularities, we briefly discuss when a moduli space $M$ is smooth in a point $[C]$.

\begin{definition}\label{def qref}
An automorphism $\a\in \mbox{GL}(n)$ is called a {\em quasireflection} if it fixes a hyperplane.
\end{definition}
Clearly, an  automorphism $\a\in \mbox{GL}(n)$ of finite order $m$ is a quasireflection if and only if it has exactly one eigenvalue $\z ^{a}$, with $\z$ a primitive $m$-th root of unity and $1\leq a\leq m-1$, and all other eigenvalues equal to one. Therefore the age of $\a$ is  
$\age(\a)=\frac{a}{m}<1$ and $\a$ is junior.
Quasireflections are important, because they do not impose singularities on a moduli space, see e.g. \cite{lu}. More precisely, one has the following result.
\begin{lemma}
A moduli space $M$ of curves as above is smooth in a point $[C]$ if and only if any automorphism $\a \in \mbox{Aut}(C)$ acts on the tangent space $T_C M$ as either the identity or as a quasireflection.
\end{lemma}
With these preparations we recall the classical criterion for the occurrence of canonical singularities, see \cite{r} and \cite{t}.

\begin{theorem}(Reid/Shepherd/Barron/Tai criterion)\label{Raid-Tai}
Let $G \subset \mbox{GL}(\C^n)$ be a finite group without quasireflections. Then $\C^n /G$ has canonical singularities if and only if every non-trival automorphism $\a\in G$ is senior.
\end{theorem}

Furthermore, the following result of Prill, see \cite{pr} and also \cite{lu}, allows  to also consider groups $G$ with quasireflections.

\begin{proposition}
Let $V = \C^n$ and $G \subset \mbox{GL}(V)$ be a finite group. Then the
subgroup $H \subset G$ generated by the quasireflections in $G$ is a normal subgroup
of $G$, there exists an isomorphism $V/H \xrightarrow{\sim} W=\C^n$ and a finite group
$K \subset \mbox{GL}(W)$ containing no quasireflections such that the following diagram
commutes:
\begin{center}
\begin{tikzpicture}
  \matrix (m) [matrix of math nodes,row sep=2em,column sep=4em,minimum width=2em]
  { V	&  V/H 			& W   \\
   V/G & (V/H)/(G/H) 	& W/K)       \\ };

  \path[-stealth]
    (m-1-1) edge node {} (m-1-2)
            edge node {} (m-2-1)
    (m-1-2) edge node [above] {$\simeq$}(m-1-3)
    (m-1-2) edge node {}(m-2-2)
    (m-1-3) edge node {} (m-2-3)
    (m-2-1) edge node [above] {$\simeq$}(m-2-2)
    (m-2-2) edge node [above] {$\simeq$} (m-2-3);

\end{tikzpicture}
\end{center}
\end{proposition}

Now we recall some basic facts of complex deformation theory for analytic spaces going back to the seminal work of Kodaira-Spencer. These play  a crucial role in the proof of Theorem 1 in \cite{hm} which controls the singularities in $\Mg$. 

As a starting point, we recall that clasical deformation theory for the moduli space $\mg$ uses the line bundle $\o_C \otimes \o_C$ of quadratic differentials. Since 
$H^0(\o_C \otimes \o_C)$ is a well known space, it is easy to perform calculations in local coordinates. In order to extend this theory to non-smooth curves, in our case stable curves, the sheaf of holomorphic differentials $\o_C$ 
is replaced by the sheaf $\O^1_C$ of K\"ahler differentials, see \cite{acg} p.95. Then,  see \cite{acg} Chapter XI  3,  equivalence classes of first order deformations of a nodal curve $C$ in the full space $\Mg$ are given by elements in 
\begin{equation}   \label{serredual}
\Ext^1(\O^1_C, \cO_C) \simeq H^0(\O^1_C \otimes \o_C)^*
\end{equation}
where the isomorphism is Serre duality,
see \cite{hm} p.27. 

In particular, $\Ext^1(\O^1_C, \cO_C)$ may be computed from the {\em local-to-global}  spectral sequence of $\Ext$'s (see \cite{acg}, p.179, and \cite{gh2}) giving the short exact sequence of vector spaces
\begin{equation}  \label{sesgext}
0 \to H^1(C, \fHom_{\cO_C}(\O^1_C,\cO_C)) \to \Ext^1_{\cO_C}(\O^1_C, \cO_C) \to 
H^0(C, \fExt^1_{\cO_c}(\O^1_C, \cO_C)) \to 0.
\end{equation}
 
Here, since the sky-scraper sheaf $ \fExt^1_{\cO_c}(\O^1_C, \cO_C))$   is concentrated at the nodes of $C$, we have
\begin{equation}
H^0(C, \fExt^1_{\cO_c}(\O^1_C, \cO_C)) = \oplus_{p \in \mathrm{Sing}(C)} \Ext^1_{\cO_{C,p}}(\O^1_{C,p} \cO_{C,p}).
\end{equation}
Furthermore, we have to describe how deformations of the curve $C$ are related to deformations of its irreducible components. Denoting by $C_{a}$ the normalizations of the components of $C$ and by $p_{b} \in C_{a}$ the points of $C_{a}$ corresponding to double points in $C$, the deformation space $\Ext^1(\O^1_C, \cO_C)$ is given by the exact sequence (see \cite{hm}, p.33)
\begin{equation}   \label{exseqcomp}
0 \to \oplus_{p \in \mathrm{Sing}(C)} (\tor_p) \to \Ext^1(\O^1_C, \cO_C) \to
\oplus_{a} H^0(\cO_{C_{a}}(2K_{a} + \sum_{b} p_{b})) \to 0
\end{equation}

\section{Birational geometry of $\Mgn$}\label{sec geo}

In this section we shall collect some well known result on the geometry of  $\Mgn$ and $\Mg$. Here we consider  $\Mg=\overline{\mathcal{M}}_{g,0}$ by abuse of notation.

The moduli spaces $\mg$ and $\mgn$ are irreducible algebraic varieties of dimension $3g-3$ and $3g-3+n$ respectively. Their compactifications $\Mg$ and $\Mgn$ are projective varieties.

It is known that $\Mgn$ is uniruled or even unirational for some small values of $g$ and $n$. In \cite{be} one finds a recent summary of results in this direction. For $g$ or $n$ large enough $\Mgn$ becomes of general type. 

The following proposition summarizes the known values of $g$ and $n$ for which $\Mgn$ is of general type, 
collecting results from \cite{l, f, fv4, f4}.  This covers all  cases known up to now.

\begin{proposition} \label{proposit}
The moduli space $\Mgn$ is of general type for $g\geq 22$ or for $n\geq n_{\mathrm{min}}(g)$ given in the following table:
\begin{table}[h!]
$$\begin{array}{c|c|c|c|c|c|c|c|  c|c|c|c|c|  c|c|c|c|c|  c}
g &4&5&6&7&8&9&10&11&12&13&14&15&16&17&18&19&20&21\\
\hline
n_{\mathrm{min}}&16&15&16&15&14&13&11&12&11&11&10&10&9&9&9&7&6&4
\end{array}$$
\caption{ } \label{table mg}
\end{table}
\end{proposition}

It is often useful to have a universal family over a moduli space. However $\Mgn$ is only a coarse moduli space and does not posses such a universal family. For this reason it is useful to consider instead the stacks $\stmg,\stMg,\stmgn,\stMgn$ of smooth or stable (n-pointed) curves. These stacks are irreducible of the same dimension as their associated coarse moduli spaces. They are however smooth and have the properties of fine moduli spaces. In particular, they posses a universal family
(see the explicit discussion in \cite{acg}  on p. 310).

\begin{proposition}
The universal family over $\stMgn$ is the forgetful map $\pi: \stMgnp\to\stMgn$ (forgetting one point).
\end{proposition}

Let us very briefly study the difference between a stacks $\M^{\mathrm{st}} $and its associated coarse moduli space $\M$, without going into detail on what a stack actually is.

There are surjective morphisms $\epsilon :\M^{\mathrm{st}}\to \M$ ramified exactly over $\Sigma$, the locus of points in $\M$ corresponding to curves with non-trivial automorphisms. In particular we can find the ramification divisor of $\epsilon$ by studying the codimension $1$ components of $\Sigma$. To this end we recall from \cite{acg}, Chapter XII Proposition 2.5 as

\begin{proposition}  \label{ram}
We consider the locus $\Sigma \subset \Mgn$ of pointed curves with a non-trivial automorphism. Then 
\begin{enumerate}
\item $\Sigma=\emptyset$ if and only if $g=0$;
\item $\Sigma=\Mgn$ if and only if $(g,n)=(1,1)$ or $(g,n)=(2,0)$.
\item In the remaining cases the codimension 1 components of $\Sigma$ are 
\begin{itemize}
\item the closure in $\M_{1,2}$ of smooth curves $C$ with marked points $x_1,x_2$ such that $2(x_1-x_2)\sim 0$;
\item the closure in $\M_{2,1}$ of smooth curves such that the marked point $x$ is a Weierstra{\ss} point;
\item $\overline{\mathcal{H}}_3$ the locus of stable hyperelliptic curves in $\M_3$;
\item  $\D_{1,\emptyset}$ the locus of curves with elliptic tails without marked points on them for any $g\geq 1$ and any $n$.
\end{itemize}
\end{enumerate}
\end{proposition}

In this thesis we are not interested in the special cases of $(g,n)\in\{ (0,n),(1,1),(2,0),(1,2),(2,1),(3,0)\}$. In all other cases the morphism $\epsilon: \stMgn\to\Mgn$ is  ramified in codimension 1 exactly over $\D_{1,\emptyset}$. We shall discuss this boundary divisor in more detail in Section \ref{sec pic} below and we note that a general element of $\D_{1,\emptyset}$ has only one non-trivial automorphism, the inversion of the elliptic tail with respect to the node. Thus the morphism $\epsilon: \stMgn\to\Mgn$  is simply ramified over $\D_{1,\emptyset}$.

\section{Picard group of $\Mgn$}\label{sec pic}

In this section we will study the rational Picard group of $\Mgn$. We introduce the tautological and boundary classes and show that they form a basis of $\Pic(\Mgn)\otimes \Q$. We will also calculate their pullbacks under the forgetful map $\M_{g,n+1}\to\Mgn$ as well as the canonical class $K_{\Mgn}$ of $\Mgn$. As a reference we refer to \cite{acg}, Chapter XII.

Throughout this thesis we will identify the rational Picard group of the moduli spaces and moduli stacks. This is justified by the following result (see \cite{acg}, Chapter XII, Lemma 6.6).

\begin{theorem}  \label{pic}
$\Pic (\Mgn)\otimes \Q \simeq \Pic (\stMgn)\otimes \Q$
\end{theorem}

In particular, we recall the notion of the Hodge class $\l$ on $\Mgn$. Informally, one considers on $\Mgn$ the Hodge bundle $ E \to \Mgn$ with fibre $E_{[C]}$ over $[C]$ taken as $H^0(C,K_C)$. Then $\l$ is thought of as the class of the determinant bundle of $E$, i.e. its top exterior power.  
Of course, this is not really well defined, e.g. in view of the singularities of $\Mgn$. It is thus advisable to pass to the smooth stack $\stMgn$, define $E$ as the direct image sheaf of the relative dualizing sheaf and then identify the class of its top exterior power with a class in the rational Picard group $\Pic (\Mgn)\otimes \Q, $
using the isomorphism in the above Theorem \ref{pic}.

In order to describe the relevant boundary divisors on $\Mgn$, we need to classify different types of simple nodes.

\begin{definition}\label{nodes}
Let $C$ be a nodal curve of arithmetic genus $g$ and $p\in C$ a node. We call $p$ {\em non-separating} or of type {\em $\d_0$}, if the partial normalization (which we informally describe as {\em  unglueing the node}) preserves connectedness.

We say that $p$ is {\em of order $i\geq 1$} or {\em of type $\d_i$}, if the partial normalization of $C$ in $p$ consist of two connected components of arithmetical genus $i$ and $g-i$ respectively.

We also call a pair of nodes $p,q \in C$ {\em of type $\e_i$}, if they are non-separating, but the partial normalization  in both of them decomposes into two connected components of arithmetical genus $i$ and $g-i-1$ respectively.
\end{definition}

Now we recall that $\Delta_0$ (sometimes also called $\Delta_{\mathrm{irr}}$) on 
$\Mg$ is the boundary component consisting of all (classes of) stable curves of arithmetical genus $g$, having at least one non-separating node. Furthermore, $\Delta_i$,  for $1 \leq i \leq \lfloor\frac{g}{2}\rfloor,$ denotes the boundary component of curves possessing a node of order $i$. Similarly, on $\Mgn$, we denote by $\Dirr$ the set of curves possessing a non-separating node and, for any subset $S \subset \{1, \ldots,n\}$, we denote by $\Delta_{i,S}, 0 \leq i \leq \lfloor\frac{g}{2}\rfloor,$ the boundary component consisting of curves possessing a node of order $i$ such that after ungluing the connected component of genus $i$ contains precisely the marked points labelled by $S$. Note that, if $S$ contains at most 1 point, one has $\Delta_{0,S}= \emptyset$ (the existence of infinitely many automorphisms on the projective line technically violates stability). Thus, in that case, we shall henceforth consider $\Delta_{0,S}$ as the zero divisor.

We shall denote by $\delta_i$ the rational divisor classes of $\Delta_i$ in $\Pic (\Mg)$ (and thus in $\Pic (\Mg)\otimes \Q $) and by $\dirr, \delta_{i,S}$ the rational divisor classes of $\Dirr,\Delta_{i,S}$ in $\Pic(\Mgn)$, respectively. Note that $\delta_0$ is also called $\delta_{\mathrm{irr}}$ in the literature, but we shall usually reserve the notation $\delta_{\mathrm{irr}}$ for the pull-back of $\d_0$ under the forgetful map
$\phi:\Mgn \to \Mg$.

Next we recall the notion of (the rational classes of) the point bundles $\psi_i, 1 \leq i \leq n,$ on $\Mgn$. Informally, the line bundle $\psi_i$ (sometimes called the cotangent class corresponding to the label $i$) is given by choosing as fibre of $\psi_i$ over a point $[C;x_1, \ldots, x_n]$ of $ \Mgn$ the cotangent line $T_{x_i}^v(C)$. More precisely, as in our discussion of the Hodge class, all these bundles are first defined on the smooth stack $\stMgn$ and then their classes are identified with rational classes on $\Mgn$ via the isomorphism in Theorem \ref{pic}.  We explicitly warn the reader that the notion of {\em point bundles on $\Mgn$} is purely a figure of speech: There are no honest such vector bundles on $\Mgn$, only the corresponding rational divisor classes exist as well defined objects in $\Pic (\Mgn)\otimes \Q $.

For later use we also set
\begin{equation}  \label{basechangeprep}
\omega_i:= \psi_i - \sum_{S \subset \{1, \ldots,n \}, S \ni i} \d_{0,S},
\end{equation}
and introduce  $\psi=\sum_{i=1}^n \psi_i.$

The rational Picard group of a moduli space forms a vector space. The classes introduced above give a basis for these vector spaces.  We recall from \cite{ac1}.

\begin{theorem}[Harers Theorem] \label{harer} \ \\
\begin{itemize}
\item The rational Picard group $\Pic (\mg)\otimes \Q$ is freely generated by the hodge class.
\item The rational Picard group $\Pic (\Mg)\otimes \Q$ is freely generated by the hodge class and all the boundary classes.
\item The rational Picard group $\Pic (\mgn)\otimes \Q$ is freely generated by the hodge class and the point bundles.
\item The rational Picard group $\Pic (\Mgn)\otimes \Q$ is freely generated by the hodge class, the point bundles and all the boundary classes.
\end{itemize}
\end{theorem}

We remark that for $\mgn$ and $\Mgn$ we can find a different basis by exchanging the point bundles $\psi$ by the $\omega_i$ from equation \ref{basechangeprep}.

Throughout this thesis we will decompose divisor classes on $\Mgn$ (always considered as elements of the rational Picard group $\Pic (\Mgn)\otimes \Q$) into a linear combination of these  special classes. We will also pull back divisors along forgetful maps, forgetting some or all marked points. In order to compute these pull backs, we need the following lemma (see \cite{acg}, Chapter XVII, Lemma 4.28).

\begin{lemma}\label{pullback}
Let $\pi : \Mgn\to\M_{g,n-1}$ be the forgetful map that forgets the $n$-th marked points. Then we have the following pullbacks
\begin{eqnarray}
&\pi^* \l = \l \\
&\pi^* \omega_i = \omega_i \\
&\pi^* \psi_i = \psi_i -\d_{0,\{i,n\}} \\
&\pi^* \dirr = \dirr \\
&\pi^* \d_{i,S} = \d_{i,S}+\d_{i,S\cup\{ n\}}
\end{eqnarray}

\end{lemma}

We will often require the canonical divisor classes of $\Mg$ and $\Mgn$ (see \cite{acg} XIII §8 Theorem (7.15) and Corollary 7.16).
\begin{lemma}
 The canonical divisor (class) of the moduli stack $\stMgn$ is 
\begin{equation}\label{can st}
K_{\stMgn}= 13\l +\psi -2\d.
\end{equation}
and for $g \geq 1, g+n \geq4,$  the canonical divisor class of the coarse moduli space $\Mgn$ is 
\begin{equation}\label{can cm}
K_{\Mgn}= 13\l +\psi -2\d - \d_{1,\emptyset}.
\end{equation}
\end{lemma}
For the canonical class of $\Mg$ we take $n=0$ and simply delete the class $\psi$ in the above formulae. Note that equation \eqref{can cm} is directly obtained from equation \eqref{can st} by using the ramification divisor from Proposition \ref{ram}.

\section{Effective divisors on $\Mgn$}\label{sec effdiv}

By Proposition \ref{criterion gentype} we can prove that a moduli space is of general type by decomposing the canonical class as the sum of an ample and an effective divisor. For this reason it is essential to know effective divisors on the moduli spaces considered. The easiest way to do this is to recall a few well known effective divisors on $\Mg$ or $\Mgn$ and pull them back to other moduli spaces.  
We recall the following standard result.

\begin{proposition}  \label{divisor}
Let $f:X \to Y$ be a morphism of projective schemes, $D \subset Y$ be an effective divisor and assume that $f(X)$ is not contained in $D$. Then $f^*(D)$ is an effective divisor on $X$.
\end{proposition}
\ 

Now let us begin to recall some well known effective divisors. These divisors were usually found by considering a subspace of codimension $1$ in $\mgn$ (for some specific cases of $g$ and $n$). Then one takes the class of the closure of these subspaces in the rational Picard group of $\Mgn$ and decomposes them as a linear combination of the classes that form a basis for the rational Picard group, see Theorem \ref{harer}.
These decompositions can e.g. be calculated with the help of test curves. We, however, will simply cite the decompositions from the literature.

The most natural effective divisors on $\Mg$ are the Brill-Noether divisors parametrizing all curves $C$ possessing a $\mathfrak{g}^r_d$, i.e. a linear series $(L,V)$ of dimension $r$ and degree $d$, for fixed $g,r,d$ with Brill-Noether number $\r(g,r,d)=-1$. These exist as long as $g+1$ is composite.
We recall from \cite{eh4} Theorem 1:

\begin{lemma}\label{BN divisor}
Assume that $g+1$ is not prime and fix some integers $r,s\geq 1$ such that $g+1=(r+1)(s-1).$ Then 
$$\mathfrak{BN}_g:=\{[C]\in \mathcal{M}_g | C \mbox{ carries a } \mathfrak{g} ^r_{rs-1} \}$$
is an effective divisor on $\mathcal{M}_g$. Furthermore the class of its compactification as a $\Q$-divisor is given by
\begin{equation}\label{BN divisor eq}
[\overline{\mathfrak{ BN}}_g]= c \Big( (g+3)\l -\frac{g+1}{6} \d_0 -\sum_{i= 1}^{\lfloor\frac{g}{2}\rfloor} i(g-i)\d_i \Big),
\end{equation}

for some positive rational number $c$.
\end{lemma}

Note that only the constant $c$ depends on the choice of $r$ and $s$. Note also that  $g+1=(r+1)(s-1)$ implies that the Brill-Noether number satisfies $\r(g,r,rs-1)=-1$. Therefore $\mathfrak{BN}_g$ parametrises precisely those (non-general) curves violating the condition of the Brill-Noether Theorem (see \cite{gh}), which states that a general curve carries a $\mathfrak{g}^r_d$, if and only if $\r(g,r,d)\geq 0$.

How efficient an effective divisor on $\Mg$ is for showing that spaces are of general type is related to the slope $s(D)$  of a divisor $D$: If $D$ has a decomposition
$$ D=a \l - \sum_{i} b_i \d_i, \quad a,b_i\geq 0$$
then its {\em slope} is defined as
$$s(D) = \max_i \frac{a}{b_i}.$$
In applications one wants to minimize the slope. 

There was a well known conjecture, called the \emph{Slope Conjecture}, stating that the Brill-Noether divisors have minimal (positive) slope among all effective divisors on $\Mg$. This conjecture has been disproved for certain values of $g$, see e.g. \cite{f}.

The main problem with the Brill-Noether divisors is that they will not exist for $g+1$  prime. In that case we have to use less efficient divisors parametrizing curves that violate the Gieseker-Petri condition. 

We recall that a curve $C$ satisfies the Gieseker-Petri condition, if for every line bundle $L$ the natural map 
$$\mu_L: H^0(C,L)\otimes H^0(C,K_C\otimes L^{-1}) \to H^0(C,K_C)$$ is injective.
We say that a linear series $(L,V)$ on a curve $C$ violates the Gieseker-Petri condition if the map 
$$\mu_{(L,V)}: V \otimes H^0(C,K_C\otimes L^{-1}) \to H^0(C,K_C)$$ is not injective. In particular a curve $C$ violates the Gieseker-Petri condition if and only if some linear series on $C$ does.

We also recall from \cite{eh4} Theorem 2 the following divisor.

\begin{lemma}\label{GP divisor}
Assume that $g=2d-2$ is even, then
$$\mathfrak{GP}_g:=\{ [C]\in\mathcal{M}_g| C \mbox{ carries a } \mathfrak{g}^1_d \mbox{ violating the Gieseker-Petri condition}\}$$ is an effective divisor on $\mathcal{M}_g$.
Furthermore the class of its compactification is given by 
\begin{equation}\label{GP divisor eq}
[\overline{\mathfrak{GP}}_g]= c( a\l- \sum_{i=0}^{\frac{g}{2}} b_i \d_i),
\end{equation}
 for some rational number $c>0$ and
$$ a=6d^2+d-6, \quad b_0= d(d-1), \quad b_1=(2d-3)(3d-2), \quad b_{i+1}>b_i. $$
\end{lemma}


The divisors we will use most often throughout this thesis are divisors of Weierstra\ss -type. We recall from \cite{l}, Section 5,
the divisors $\mathfrak{D}(g;a_1, \ldots,a_n)$ on $\Mgn$ with $a_i \geq 1$ and $\sum a_i=g$.
They are given by the locus of curves $C$ with marked points $x_1, \ldots, x_n$ such that there exists a $\mathfrak{g}^1_g$ on $C$ containing the divisor $\sum_{1 \leq i \leq n}a_i x_i$.  In particular $\mathfrak{D}(g;g)$ is the Weierstra\ss \ divisor which contains all one-pointed curves $(C,x)$ such that $x$ is a Weierstra\ss \ point of $C$.

We introduce these divisors in the basis that uses the classes $\omega_i$ from equation \eqref{basechangeprep} instead of the point bundles $\psi_i$. This makes some calculations easier.

\begin{lemma}\label{Logan divisor}
Assume that $\sum a_i=g$, then
$$\mathfrak{D}(g;a_1, \ldots,a_n):=\{ [C]\in\Mgn| C \mbox{ carries a }\mathfrak{g}^1_g \mbox{ through }\sum_{i=1}^n a_i x_i \}$$ is an effective divisor on $\mgn$.
The class of its compactification is given by 
\begin{equation}\label{Logan divisor eq}
[\overline{\mathfrak{D}}(g;a_1, \ldots,a_n)]= -\l+\sum_{i=1}^n \frac{a_i(a_i+1)}{2}\omega_i-0\cdot \dirr -\sum_{i,S} b_{i,S}\d_{i,S},
\end{equation}  
where $b_{0,\{i,j\}}=a_i a_j$ and $b_{i,S}\geq 0$ for all $i$ and $S$.
Furthermore if $n=g$ and all $a_i=1$ then 
\begin{equation}\label{Logan divisor eq2}
\begin{split}
[\overline{\mathfrak{D}}(g;1, \ldots,1)]
&= -\l+\sum_{i=1}^n \omega_i-0\cdot \dirr -\sum_{|S|=s} \Big( \frac{s(s-1)}{2}\d_{0,S} +\sum_{i=1}^{\lfloor\frac{g}{2}\rfloor}\frac{|s-i|(|s-i|+1)}{2}\d_{i,S} \Big) \\
&= -\l+\sum_{i=1}^n \psi_i-0\cdot \dirr -\sum_{|S|=s} \Big( \frac{s(s+1)}{2}\d_{0,S} +\sum_{i=1}^{\lfloor\frac{g}{2}\rfloor}\frac{|s-i|(|s-i|+1)}{2}\d_{i,S} \Big)
\end{split}
\end{equation}
\end{lemma}

In practice we are rarely interested in any of the divisors $\mathfrak{D}$ themselves. Usually we want to take a positive multiple of a sum of pull-backs of these divisors. They are most useful when they are $S_n$-invariant, i.e. $|S|=|T|\Rightarrow b_{i,S}=b_{i,T}$, and also the $\psi_i$ or $\omega_i$ all have the same coefficient 
$c$, and when the quotients 
$\frac{b_{i,S}}{c}$  are maximal. The symmetry condition can be achieved by taking the sum over all possible permutations. The second condition requires us to minimize the distance between the weights $a_i$. \\
\ \\
We will calculate these sums of pull-backs separately depending on whether $n\leq g$ or $n\geq g$. In the following computations we shall use the shorthand
$$ \psi= \sum_{i=1}^n \psi_i,\qquad \o=\sum_{i=1}^n \o_i \qquad  \d_{i,s}= \sum_{|S|=s} \d_{i,S}. $$

We begin with the case $n\leq g$. In order to minimize the distance between the weights $a_i$ we decompose
$g=kn+r$, with  $r<n$,  and set
\begin{equation}\label{wmdef}
W_n=[\overline{\mathfrak{D}}(g;a_1, \ldots,a_n)], \qquad a_j=k+1 \, \, (1\leq j \leq r), \quad  
a_j=k  \, \, (r+1 \leq j \leq n).
\end{equation}

This gives, in view of \eqref{Logan divisor eq}
\begin{equation}\label{wm}
\begin{split}
W_n 
& = -\l+	\sum_{i=1}^r \frac{(k+1)(k+2)}{2}\omega_i +\sum_{i=r+1}^n \frac{k(k+1)}{2} \omega_i  -0\cdot\dirr 
\\&-\sum_{i,j\leq r} (k+1)^2 \d_{0,\{i,j\} }  -\sum_{i\leq r,j>r} k(k+1) \d_{0,\{i,j\} } -\sum_{i,j>r} k^2 \d_{0,\{i,j\} } \\
&-\mbox{higher order boundary terms},  
\end{split}
\end{equation}
where {\em higher order}  means a positive linear combination of $\d_{i,S}$ where either $i>0$ or $|S|>2$.

From $W_n$ we want to generate  an $S_n$-invariant divisor class $W$ on $\Mgn$, by summing over appropriate permutations. Thus we let $S,T$ be disjoint subsets of $\{1,\ldots,n \}$ with $|S|=r$ and $|T|=n-r$ and let 
\begin{equation}
\phi_{S,T}: \Mgn \to \Mgn
\end{equation}
be a permutation mapping the class 
$[C; x_1, \ldots ,x_n]$ to $[C; y_1, \ldots ,y_{n}],$
where 
$\{x_1, \ldots ,x_n\}=\{y_1, \ldots,y_n\}$ denote the same set of points and the points $x_i$ labelled by $S$ are sent to the points $y_1, \ldots,y_r$ (all with weights $a_i=k+1$) and the points labelled by $T$ are sent to the points $y_{r+1}, \ldots,y_n$ (all with weights equal to $k$). Clearly, for fixed $g$ and $n$, there are precisely  $\binom{n}{r}$ such permutations. With this notation we introduce
\begin{equation}\label{W small}
\begin{split}
W := &\sum_{S,T} \phi_{S,T}^* W_n \\
=&-\binom{n}{r} \l +c\sum_{i=1}^n \omega_i+0\cdot\dirr-\sum_{|S|=s} \tilde{b}_{0,s}\d_{0,S} -\mbox{higher order boundary terms}\\
 =&-\binom{n}{r} \l +c\sum_{i=1}^n \psi_i+0\cdot\dirr-\sum_{|S|=s} b_{0,s}\d_{0,S} -\mbox{higher order boundary terms}
\end{split}
\end{equation}
where  {\em higher order} denotes a positive linear combination of boundary divisors $\d_{i,S}$ with     $i \geq 1$,
\begin{equation}\label{coef W small}
\begin{split}
b_{0,s}&=\tilde{b}_{0,s}+sc \\ 
c &=\binom{n-1}{r-1} \frac{(k+1)(k+2)}{2} +\binom{n-1}{r} \frac{k(k+1)}{2},\\
b_{0,2} &= 2c + \binom{n-2}{r-2} (k+1)^2 +2\binom{n-2}{r-1} k(k+1) +\binom{n-2}{r}k^2. 
\end{split} 
\end{equation}

The first equation of \eqref{coef W small} is obtained by base change, since equation \eqref{basechangeprep} implies
\begin{equation} \label{basechange}
\o=\psi - \sum_S |S| \d_{0,S}.
\end{equation}
The other two equations of \eqref{coef W small} are proven by combinatorial considerations which we leave to the reader. \\
\ \\
%
%
Now let us turn to the case $n\geq g$. 
We begin with the following divisor on $\Mgg$
\begin{equation}\label{Wg}
W_g =[\overline{\mathfrak{D}}(g;1, \ldots,1)]=  -\l + \psi -0\cdot\dirr -3\d_{0,2}-\frac{g(g+1)}{2}\d_{0,g} - \mbox{other terms},
\end{equation}
where "other terms" means a linear combination of the other boundary divisors with non-negative coefficients.

Then we pull this divisor back to $\Mgn$ in such a way that the pull-back becomes $S_n$-invariant again.
For a set $S\subset \{1,\ldots,n\}$ of cardinality $g$ we take $\pi_S: \Mgn \to \Mgg$ as the forgetful map forgetting all points not labelled by $S$. We construct an effective divisor $W$ on $\Mgn$ by summing over the pull-backs of $W_g$ along $\pi_S$
\begin{equation}\label{W large}
W:=\sum_S \pi_S^* W_g = -\binom{n}{g}\l +\binom{n-1}{g-1}\psi -0\cdot \dirr -\sum_{s\geq 2}b_{0,s}\d_{0,s}  -\mbox{higher order terms},
\end{equation}
with $b_{0,2}=2\binom{n-2}{g-1}+3\binom{n-2}{g-2}=2\binom{n-1}{g-1}+\binom{n-2}{g-2}, \quad b_{0,n}=\binom{n}{g}\frac{g(g+1)}{2} $ and $b_{0,s}\geq b_{0,2}$ for all $s>2$. \\
\ \\
%
Finally we introduce divisors parametrizing curves with marked points that fail the so called {\em Minimal Resolution conjecture}.  

The Minimal Resolution Conjecture (MRC) predicts the Betti numbers of a general finite set of points in a projective space. It closely relates to the theory of  syzygies and Kozul complexes which is a wide and interesting field in its own. It does not, however, have any relevance for the rest of this thesis and we will introduce only the bare minimum needed to understand what is parametrized by the following effective divisor. We refer the interested reader to \cite{fmp} for more details on the MRC and to \cite{e} for the underlying commutative algebra. 

Let us consider a subscheme $X\subset \P^n$. We denote by $S=\C[x_1,\ldots, x_{n+1}]$ the homogeneous coordinate ring of $\P^n$ and by $I_X\subset S$ the saturated ideal of $X$. Then $S_X :=S/I_X$, the coordinate ring of $X$, is a finitely generated graded $S$-module. We can study such a module $M$ with the help of a free resolution
$$ F_{\bullet}: \quad \ldots\to F_n\to\ldots \to F_2 \to F_1 \to F_0 \to M\to 0.$$
Here giving a minimal free resolution is equivalent to giving the \emph{Betti numbers} $b_{i,j}(M)$ via the relation
$$F_i= \bigoplus_{j\in \Z} S(-i-j)^{b_{i,j}(M)}.$$
The \emph{Betti diagram} of $M$ has the number $b_{i,j}(M)$ as its $(j,i)$-th entry. For our subscheme $X\subset \P^n$ we define the Betti diagram of $X$ as the Betti diagram of $S_X$. This diagram has only finitely many rows with some entry different from zero. The number of non-trivial rows is called the \emph{regularity} of $X$.

Now let us consider a subset $\Gamma\subset X$ of finitely many points on $X$. Then the Betti diagram of $\Gamma$ adds two non-trivial rows to the Betti diagram of $X$ which we denote by $j=r-1$ and $j=r$. The MRC predicts that for a general set of points $\Gamma$ these entries satisfy the equations
$$b_{i+1,r-1}(\Gamma) \cdot b_{i,r}(\Gamma)=0 \qquad \mbox{for all } i.$$

The MRC has been studied first and foremost in the case $X=\P^n$ where it is known to hold for small values of $n$ or for $m:=|\Gamma|$ large enough. We, however, are more interested in the case of curves. Here the validity of the MRC depends on an actual embedding of a curve $C$ into projective space; the Betti numbers are not birational invariants of the curve $C$. Such an embedding can e.g. be given by a very ample line bundle or certain linear series. Here we only consider canonical curves $C$, i.e. the curve $C\subset \P^{g-1}$ is embedded into projective space via its canonical divisor. In this special case the MRC is known to be true. We recall from \cite{fmp}:

\begin{proposition}[Minimal Resolution Conjecture]\label{MRC}
Let $C\subset \P^{g-1}$ be a canonical curve of arithmetic genus $g$ and regularity $r$ and let $\Gamma\subset C$ be a general set of $m \geq \max\{g, P_C(r)\}$ points where $P_C$ denotes the Hilbert polynomial of $C$. Then 
$$b_{i+1,r+1}(\Gamma) \cdot b_{i,r+2}(\Gamma)=0 \qquad \mbox{for all } i.$$
\end{proposition}

Another approach to the Betti numbers is to express them as certain cohomology classes. For a smooth curve of genus $g$ we denote by $M_C$ the vector bundle which is the kernel of the evaluation map
$$0\to M_C\to H^0(C,\omega_C)\otimes \mathcal{O}_C \to \omega_C \to 0$$
and for a finite subset $\Gamma\subset C$ we denote ideal sheave of $\Gamma$ in $C$ by $\mathcal{I}_{\Gamma/C}$. Then Proposition 1.6 in \cite{fmp} gives the last two rows in the Betti diagram of $\Gamma$ as
$$ b_{i+1,r-1}=h^0(\bigwedge^i M_C\otimes \mathcal{I}_{\Gamma/C}(r) \quad \mbox{and} \quad
b_{i,r}=h^1(\bigwedge^i M_C\otimes \mathcal{I}_{\Gamma/C}(r) \quad \mbox{for all } i.$$

With some simple calculations (see Corollary 1.8 in \cite{fmp}) we can deduce from this that the following effective divisor from Theorem 4.2. in \cite{f} parametrizes curves in $\mgn$ whose marked points violate the Minimal Resolution Conjecture \ref{MRC}. As mentioned before, for the remainder of this thesis, we only need the existence of this effective divisor.

\begin{lemma}\label{MinRes div}
Fix integers $g,r\geq 1$ and $0\leq k\leq g-1$ and set $n=(2r+1)(g-1)-2k$. Then
$$\mathfrak{Mrc}_{g,k}^r:= \{ [C,x_1,\ldots, x_n]\in \mgn| h^1(C, \bigwedge^i M_C\otimes \o_C^{\otimes(r+1)} \otimes\mathcal{O}_C(-x_1-\cdots-x_n))\geq 1\}$$
is a divisor on $\mgn$. Furthermore the class of its compactification is given by
\begin{equation}
[\overline{\mathfrak{Mrc}}_{g,k}^r]= \frac{1}{g-1} \binom{g-1}{k}(-a \l +c \sum_{i=1}^n\psi_i+b_{\mathrm{irr}}\dirr -\sum_{i,s} b_{i,s} \sum_{|S|=s} \d_{i,S} )
\end{equation}
where 
\begin{equation}\label{coef MinResConj}
\begin{split}
&a	=\frac{1}{g-2}\Big( (g-1)(g-2)(6r^2+6r+r)+k(24r+10k+10-10g-12rg)\Big),	\\ 
&c= rg+g-k-r-1,	\\
&b_{\mathrm{irr}}=\frac{1}{g-2}\Big( \binom{r+1}{2} (g-1)(g-2)+k(k+1+2r-rg-g)\big), \\
&b_{0,s}	= \binom{s+1}{2}(g-1)+s(rg-r-k) \qquad\mbox{and} \qquad b_{i,s}\geq b_{0,s}.
\end{split}
\end{equation}
\end{lemma}

\chapter{On the Kodaira dimension of the moduli space of nodal curves}

We show that the compactification of the moduli space of {\em $n-$nodal curves of geometric genus g}, i.e. $\Ngn:= \Mgnn /G$, with $G:=(\Z_2)^n\rtimes S_n$,  is of general type for $g \geq 24$, for all $n \in \N$. While this is a fairly easy result,  it requires completely different techniques to extend it to low genus $5 \leq g \leq 23$. Here we need
that the number of nodes  varies in a band
 $n_{\mathrm{min}}(g) \leq n \leq  n_{\mathrm{max}}(g)$, where $n_{\mathrm{max}}(g)$ is the largest integer smaller than (or in some cases equal to) $\frac{7}{2}(g-1)-3$.
The lower bound  $n_{\mathrm{min}}(g) $ is close to the bound found in \cite{l}, \cite{f} for $\Mgnn$ to be of general type (in many cases it is identical). This  will be tabled in Theorem  \ref{two ngn} which is the main result of this chapter.

\section{Introduction}

The goal of this chapter is to study the moduli space $\cN_{g,n}:= \cM_{g,2n}/G$ of {\em $n-$nodal curves of geometric genus g} and its compactification $\Ngn:= \Mgnn /G$. 
Here $\cM_{g,2n}$ is the moduli space of smooth curves of genus $g$ with $2n$ distinct marked points   
on which the group $G:=(\Z_2)^n\rtimes S_n$, with  $S_n$ being the symmetric group, acts in the following way: We group the labels $\{1,...,2n\}$ in pairs 
$(i,n+i),$ where  $ i\leq n$, corresponding to pairs of points $(x_i,y_i)$ (where $y_i=x_{i+n}$). Each of the $n$ copies of $\Z_2$ acts 
on one of these pairs by switching components, and $S_n$ acts by permutation of the pairs. 
Then $G$ acts on the moduli space $\Mgnn$ by acting on the labels of the $2n$ marked points.
We shall henceforth assume without further comment that this grouping is fixed whenever we consider $2n$ marked points. Clearly, $G$ then acts transitively  on the marked points   and  fixed point free on the moduli space.

A special interest in nodal curves, and thus in the study of the moduli space
$\Ngn$, (with the aim of better understanding the general properties of smooth curves) is a common feature of all deformation type arguments. They go back at least to 
Severi who proposed  in \cite{se} to use  $g-$nodal curves 
to prove the Brill-Noether theorem, see  \cite{bn}. Such deformation type methods have become prominent in modern algebraic geometry, e.g. in the systematic theory of limit linear series (see \cite{eh2}) which progressively simplified the earlier deformation type arguments in the first rigorous proof of the Brill-Noether theorem and the Gieseker-Petri theorem (see \cite{gh}, \cite{g},  \cite{eh}, \cite{eh1},  \cite{sch1}).

Thus it seems natural  
to study the moduli space $\cN_{g,n}$ of nodal curves in its own right.\\

A crucial point in our analysis is the following diagram (which, incidentally, also shows that $\Ngn$ actually parametrizes $n-$nodal curves with geometric genus $g$ and arithmetic genus $g+n$):
 
  $$\begin{tikzpicture} [baseline=(current  bounding  box.center)]
\matrix (m) [matrix of math nodes,row sep=2em,column sep=3em,minimum width=2em]
{\Mgnn& 	\Mgun\\
 	    \Ngn	&	\\};
\path[->>]
   (m-1-1) edge node[left] {$\pi$} (m-2-1);
\path[->]
	(m-1-1) edge node [above] {$\chi$} (m-1-2)
 	(m-2-1) edge node [below]{$\m$} (m-1-2);
         
\end{tikzpicture}$$
which factors $\chi$ through the quotient map $\pi$, $\chi$ being the map which glues $x_i$ to $y_i$ for every pair $(x_i,y_i)$. Thus each pair is transformed into a nodal point of the curve $C$ under consideration, increasing its arithmetic genus by $n$.

While it was known for a long time that the moduli space $\Mg$ of algebraic curves is uniruled for low genus $g$, Eisenbud, Harris and Mumford showed in the 1980's that $\Mg$ is of general type for $g \geq 24$ (see  \cite{hm} and \cite{eh}), the case of $\overline{\mathcal{M}}_{23}$ being somewhat special (see also \cite{f2}).

Since then there has been a lot of research refining this picture of the birational geometry of moduli spaces of curves. The addition of marked points on the algebraic curve leads to the moduli space $\Mgn$ of $n$-pointed curves of genus $g$ whose Kodaira dimension was studied in \cite{l}, with some improvements in \cite{f}, 
employing refined versions of the original techniques. Heuristically, additional marked points push the moduli space in the direction of being of general type. In particular,  $\overline{\mathcal{M}}_{23,n}$ is of general type for all  $n\geq 1$, and even for $4\leq g\leq 22$  the moduli space $\Mgn$ achieves general type if $n\geq n_{\mathrm{min}}(g)$ for $n_{\mathrm{min}}(g)$ sufficiently large.
In another direction, the paper \cite{bfv} studies the Kodaira dimension of the Picard variety $\overline{\mbox{Pic}}^d_g$ parametrizing line bundles of degree $d$ on curves of genus $g$.\\

A further gratifying aspect of $\Ngn=\Mgnn/G$ is its structure as a quotient of the well-understood space $\Mgnn$ by a finite group. Such quotients have been studied before in various contexts, and their birational geometry might be different from what one could naively expect.

For instance, the space $\overline{\cC}_{g,n}:=\Mgn /S_n$, for $g$ large,is  of general type for $n \leq g-1$ but uniruled for $n \geq g+1$; only the transitional case $g=n$ is challenging (see \cite{fv1}). In fact, to see this for large $n$, one observes that the fibre of $\Mgn/S_{n}\to \Mg$ over a smooth curve $[C]\in \Mg$ is birational to the symmetric product $C_{n}$. Since the Riemann-Roch theorem implies that any effective divisor of degree $d> g$ lies in some $\mathfrak{g}^1_d$, the quotient  $\Mgn/S_{n}$ is trivially uniruled for $n> g$. 

This proof, of course, uses the very special structure of the above fibre as a symmetric product $C_n$ (which can be interpreted as a family of divisors on $C$ and thus is accessible via Brill-Noether theory) and not just abstract properties of the group $S_n$. Still, it points to the possibility that taking the quotient with respect to a sufficiently large group might somehow destroy the property of an algebraic variety of being of general type. At least for low genus, this is compatible with the findings of this chapter, see Theorem \ref{two ngn} below. In our case, such a phenomenon might be related to the existence of an upper finite bound $n_{\mathrm{max}}(g)$ for the number of nodes allowed on the geometric genus $g$ curve 
Note that our group {$G\subset S_{2n}$ is a subgroup of $S_{2n}$, implying that $\Mgnn\to \Mgnn/S_{2n}$ factors through $\Ngn$. Thus the space $\Ngn$ should be somewhat intermediary between
$\Mgnn$ and  $\overline{\cC}_{g,2n}$.

Making this precise  is the central result of the present chapter. Most demanding is the result for small $g$, namely:

\begin{theorem}   \label{two ngn}
$\Ngn$ is of general type in the special cases 
$5 \leq g \leq 23,$  and\\
$n_{\mathrm{min}}(g) \leq n \leq  n_{\mathrm{max}}(g)$ 
with  $n_{\mathrm{min}}(g), n_{\mathrm{max}}(g)$  given in the following table:
$$\begin{array}{c|c|c|c|c|c|c|  c|c|c|c|c|  c|c|c|c|c|  c|c|c}
g &5&6&7&8&9&10&11&12&13&14&15&16&17&18&19&20&21&22&23\\
\hline
n_{\mathrm{min}}&9&9&8&8&8&6&6&6&6&5&6&5&5&5&4&4&2&2&1\\
n_{\mathrm{max}}&10&14&18&21&25&28&32&35&38&42&46&49&52&56&60&63&66&70&74
\end{array}$$
\end{theorem}

For large $g$ we will show the following:

\begin{theorem} \label{one ngn}
The moduli space $\Ngn$ is of general type for $g\geq 24$.
\end{theorem}

The crucial reason for Theorem \ref{one ngn} is that $\Mg$ is of general type for $g\geq 24$. The theorem could automatically be improved if there were proof of $\Mg$ being of general type for some $g<24$.
\footnote{This case has actually occurred. In \cite{f4} it is proven that $\Mg$ is of general type for $g=22$ and $g=23$. However this chapter is based on our paper \cite{sch2} which was published before \cite{f4}. Therefore we have chosen not to update our results.}
Then $\Ngn$ also is of general type, by an identical proof.
We emphasize that in this case our result is uniform in  the number $n$ of nodes. There is no finite upper bound $n_{\mathrm{max}}(g)$. The proof (presented in Section \ref{p1}) uses arguments of fairly general nature and does not require elaborate calculations with carefully selected divisors.
This changes drastically in low genus $g \leq23$.
It is clear from the above diagram that $\Ngn$ can only be of general type if $\Mgnn$ is. 
In our opinion it is easier to show that $\Mgnn$ is of general type than to establish this property for $\Ngn$. Thus we have not even attempted to consider those values of $g,n$ where $\Mgnn$ is not known to be of general type.
Thus, for low genus, we have considered individually the cases in which $\Mgnn$ is known to be of general type (see \cite{l} and \cite{f}).
For the convenience of the reader we shall recall:\\

$\Mgnn$ is of general type 
for $5 \leq g \leq 23,$  and
$n \geq n_{\mathrm{min}}(g) , $ 
with  $n_{\mathrm{min}}(g$) 
being given in the following table:
$$\begin{array}{c|c|c|c|c|c|c|  c|c|c|c|c|  c|c|c|c|c|  c|c|c}
g &5&6&7&8&9&10&11&12&13&14&15&16&17&18&19&20&21&22&23\\
\hline
n_{\mathrm{min}}&8&8&8&7&7&6&6&6&6&5&5&5&5&5&4&3&2&2&1

\end{array}$$
\\

In order to obtain the analogue  table for $\Ngn$ in Theorem \ref{two ngn},
our approach  involves deriving sufficient conditions for the property of being of general type in terms of divisors invariant under the group action. Basic building blocks are the Brill-Noether divisors on $\Mg$ and $\Mgun$, appropriately pulled back, and divisors of Weierstrass-type (see Section 4). For $g+1$ or $g+n+1$ prime, the Brill-Noether divisors have to be replaced by less efficient divisors which we take from \cite{eh}. This gives, in any case,  a system of linear inequalities which is solvable for certain values of  $(g,n)$. This solvability is a sufficient condition for $\Ngn$ being of general type. In spirit our analysis is related to \cite{fv1,fv2,fv3}, but  the detailed analysis is  different and requires new ideas.  
The explicit solution of the relevant systems of inequalities is best done using computer algebra, paying special attention to a number of limiting cases where the appropriate divisors have to be chosen with care.

We further remark that, in the  table of Theorem \ref{two ngn}, the upper cut-off at $n_{\mathrm{max}}(g)$ is the largest integer smaller than (or in some cases equal to) $\frac{7}{2}(g-1)-3$.
and that, using this divisor based computational approach to prove general type in high genus $g \geq 24$ - avoiding the completely different arguments in the proof of Theorem \ref{one ngn} - would give a much weaker result than Theorem \ref{one ngn},  since one only gets general type for $n$ bounded by the same $n_{\mathrm{max}}(g)$. This appearance of an upper bound for $n$ in Theorem \ref{two ngn} is the main difference to both our Theorem \ref{one ngn} and the results in the low-genus table for $\Mgn$.   We remark that in order to improve our result on the upper cut-off with the techniques of this chapter, one needs a new type of effective divisor $A$ on $\Mgnn$  of the form (using our notation established in Section \ref{sec pic})
$$ a \lambda - b_{irr} \d_{irr} - \sum_{i,S} b_{i,S} \d_{i,S},$$
with all coefficients nonnegative and $b_{0,S} \neq 0$ for $|S|=2$, or of the form
$$-a \lambda +c \psi - b_{irr} \d_{irr} - \sum_{i,S} b_{i,S} \d_{i,S},$$
with all coefficients nonnegative and $b_{0,S} > 3 c$ for $|S|=2$. We do not know of any such divisor. In any case, at present it is not clear if our result is sharp.
To clarify this question is, in our opinion, the most interesting point left open by the results of our present chapter.

 The outline of this chapter is as follows. In Section 2 we prove Theorem \ref{one ngn}. In Section 3 we give sufficient conditions for our statement on $\Ngn$ being of general type in terms of an appropriate decomposition of the canonical divisor $K_{\Ngn}$.  
 In Section 4 we introduce most of the divisors needed in the subsequent analysis (i.e. all the standard divisors used in the next section), while in Section 5 we derive most of the table (the standard part of it) of Theorem \ref{two ngn} by  solving  the ensuing system of inequalities. This proof is the least technical. 
It leaves open a number of special cases (precisely 9), each of which requires additional special divisors and a lot of special attention. This is the content of Section 6 and concludes the proof of Theorem \ref{two ngn}.

Finally we remark that, though not being strictly necessary, it  is 
 convenient to employ in Sections 5-6 a standard computer algebra program, e.g. Mathematica or Maple.

\section{Proof of Theorem \ref{one ngn}} \label{p1}
\setcounter{equation}{0}
We shall prove the assertion by representing $\Ngn$ via a base and a general fibre. We recall the Iitaka Conjecture \ref{subadditivity} on the subadditivity of the Kodaira dimension which states that for a fibration $\phi: X \to Y$ with general fibre $F$ the Kodaira dimension satisfies $\k(X) \geq \k(Y)+\k(F)$. In particular, if $F$ and $Y$ are both of general type, then the conjecture is known to be true and implies that $X$ is also of general type.

Now consider the projection 
$$\phi: \Ngn= \Mgnn/G \to\Mg,$$
with general fibre 
$$F  \simeq C^{2n}/G \simeq \Sym ^{n}(\Sym^2(C)) =(C^2/\Z_2)^n/S_n,$$ 
where $ \Sym^m(C):= C^m/S_m$ denotes 
the $m$-fold symmetric product of the curve, and the crepant resolution given by  the Hilbert scheme (see e.g. \cite{hl})
$$\mbox{Hilb}^n(S) \to \Sym^n(\Sym^2(C)), \qquad S=\Sym^2(C).$$
This is of general type if $S$ is of general type. It is well known that $\Sym^d(C)$ is of general type for $d\leq g-1$. Combining these results we get that both the base $\Mg$ and the general fibre $F$ are  of general type, completing the proof.
For the sake of the reader we recall the argument that $\Sym^d(C)$ is of general type for $d\leq g-1$ based on results of \cite{kou}. 

Letting $C$ be an irreducible smooth genus $g$ curve, we consider for $d\leq g-1$ the canonical projection
$$\pi : C^d \to C_d:=\Sym^d(C),$$
which induces a canonical map $D\mapsto \mathcal{L}_D$ from divisors on $C$ to divisors on $C_d$. Denoting by $K$ the canonical divisor on $C$ and by $\Delta /2$ the ramification divisor of $\pi$, Prop. 2.6 of \cite{kou} gives the canonical divisor on $C_d$ as
\begin{equation}\label{canonical class Cd}
K_{C_d}=\mathcal{L}_K - \Delta /2.
\end{equation}
Using the standard interpretation of points in $C_d$ as effective divisors on $C$ of degree $d$ and fixing some effective divisor $D$ of degree $g-1-d$ gives a map
$$\alpha_D: C_d \to C_{g-1}, \quad A \mapsto A+D.$$
Composing $\a_D$ with the Abel-Jacobi map
$$u: C_{g-1}\to J^{g-1}(C),$$
and taking $\theta$ to be the $\theta$-divisor on the Jacobi torus $J^{g-1}(C)$, Prop. 2.3 and Prop. 2.7. of \cite{kou} give
$$u^* \a_D^* \theta =\mathcal{L}_K-\Delta /2 -\mathcal{L}_D.$$
Combining this with \eqref{canonical class Cd} yields
$$ K_{C_d}= u^* \a_D^* \theta +\mathcal{L}_D,$$
representing the canonical divisor of $C_n$ as a sum of an ample divisor (because $\theta$ is ample) and an effective divisor. By Proposition \ref{criterion gentype} this proves that $C_d$ is of general type.

\section{Moduli spaces of nodal curves} \label{sec ngn}
\setcounter{equation}{0}

The aim of this section is to develop a sufficient condition for $\Ngn$ being of general type. This requires a basic understanding of the Picard group $\Pic( \Ngn)$ and an explicit description of the  boundary divisors and tautological classes on $\Ngn$ which we shall always consider as $G$-invariant divisors on $\Mgnn$ (any such divisor descends to a divisor on $\Ngn$).
These divisor classes were introduced in Section \ref{sec pic} for $\Mgn$. Here we must only study which of these classes are already $G$-invariant and which require to be symmetrized.
\\

We recall that the Hodge class $\l$ on $\Mgn$ does not impose conditions on the marked points. Therefore it is $S_n$ and thus $G-$ invariant. This means that it gives the Hodge class $\l$ on $\Ngn$ (where, by the usual abuse of notation, we denote both classes by the same symbol).
 
Let us now consider the action of $G$ on the boundary divisor classes of $\Mgnn.$ Clearly, $\delta_{\mathrm{irr}}$ is $G -$ invariant. Furthermore, $\d_{i,S}$ is mapped into $\d_{j,T}$ by some element of $G$ if and only if $i=j, |S|=|T|$ and $S$ and $T$ contain the same number of pairs. This implies that one gets a $G-$ invariant  divisor class by setting
$$\d_{i;a,b}= \sum_S \d_{i,S},$$
where the sum is taken over all subsets $S$ containing precisely $a$ pairs and $b$ single points.\\

Finally we study the point bundles $\psi_i$. Since $G$ acts transitively on the $2n$ labels any point bundle $\psi_i$ can be mapped to any other point bundle $\psi_j$ by some element of $G$. Thus we get a $G$-invariant class only by taking the sum $\psi=\sum_{i=1}^n\psi_i$.\\

As a first step in the direction of our sufficient criterion we need the following result on the geometry of the moduli space $\Ngn.$

\begin{theorem}  \label{noadcon ngn}
The singularities of $\Ngn$ do not impose adjunction conditions. 
\end{theorem}

The proof follows the lines of the 
the proof of Theorem 1.1 in \cite{fv1}. 
We shall briefly review the argument. A crucial input is Theorem 2 of the seminal paper \cite{hm} which proves that the moduli space $\Mg$ has only canonical singularities. The proof relies on the Reid-Tai criterion: Pluricanonical forms (i.e. sections of $K^{\otimes \ell}$) extend to the resolution of singularities, if for any automorphism $\sigma$ of an object of the moduli space the so-called {\em age} satisfies  $age(\sigma) \geq 1$. The proof in \cite{fv1} then proceeds to verify the Reid-Tai criterion for the quotient of $\Mgn$ by the full symmetric group $S_n $. Here one specifically has to consider those automorphisms of a given curve which act as a permutation of the marked points. For all those automorphisms the proof in \cite{fv1} verifies the   Reid-Tai criterion. Thus, in particular, the criterion is verified for all automorphisms which act on the marked points as an element of some subgroup of $S_n$. Thus, the proof in \cite{fv1} actually establishes the existence of only canonical singularities for {\em any} quotient $\Mgn/G$ where $G$ is a subgroup of $S_n$. Clearly, this covers our case.

Theorem \ref{noadcon ngn} implies that the Kodaira dimension of $\Ngn$ equals the Kodaira-Iitaka dimension of the canonical class $K_{\Ngn}$. In particular, $\Ngn$ is of general type if  $K_{\Ngn}$ is a positive linear combination of an ample and an effective  rational class on $\Ngn$.  It is convenient to slightly reformulate this result. We need

\begin{proposition} \label{psi ngn}
The class $\psi$ on $\Mgnn$ is the pull-back of a divisor class on $\Ngn$ which is big and nef.
\end{proposition}
 
\begin{proof}
 Farkas and Verra have proven in  Proposition 1.2 of \cite{fv2} that the $S_{2n}$-invariant class $\psi$ descends to a big and nef divisor class $N_{g,2n}$ on the quotient space $\Mgnn /S_{2n}$. Consider the sequence of  natural projections
 $\Mgnn \xrightarrow{\pi} \Mgnn/G \xrightarrow{\n} \Mgnn/S_{2n}.$
Then $\n^*(N_{g,2n})$ is a big and nef divisor on $\Ngn=\Mgnn/G$ and $\pi^*(\n^*(N_{g,2n})=\psi$. 
\end{proof}

 Now observe that the ramification divisor (class) of the quotient map $\pi$ is precisely $\d_{0;1,0}$. In fact a quotient map $X\to X/G$ is ramified in a point $x\in X$ exactly if there exists a non-trivial element $g\in G$ such that $g(x)=x$. In our case an element $g\in G$ acts on a class of pointed curves by $g[C,x_1,\ldots, x_{2n}]=[C;x_{g(1)},\ldots, x_{g(2n)}]$. By standard results  the pointed curves $(C,x_1,\ldots, x_n)$ and $(C;x_{\sigma (1)},\ldots, x_{\sigma (2n)})$ are isomorphic for some $\sigma\neq id \in S_n$ , if and only if $C$ has a rational component with exactly two marked points and $\sigma$ is the transposition switching these two points. Since our group $G$ contains exactly those transpositions belonging to the pairs $(x_i,y_i)$, the ramification divisor is $\d_{0;1,0}$.
 
 Furthermore, standard results on the pullback give

\begin{equation}\label{K ngn}
K:= \pi^*(K_{\Ngn})= K_{\Mgnn} - \mbox{ram } \pi= 13 \l + \psi - 2 \d -\d_{1,\emptyset} - \d_{0;1,0}.
\end{equation} 

Here $\d$ is the class of the boundary of $\Mgnn$, i.e.
$$ \d = \dirr + \sum_{0 \leq i \leq \lfloor\frac{g}{2}\rfloor} \sum_{S \subset \{ 1, \ldots ,2n \} }
\d_{i,S}. $$

 We thus obtain the final form of our sufficient condition:
 If  $K$ is a positive multiple of $\psi$ + some effective $G - $ invariant divisor class on $\Mgnn$, then $\Ngn$ is of general type.

\section{Effective $G$-invariant divisors on $\Mgnn$} \label{sec div ngn}
\setcounter{equation}{0}

In this section we construct $G$-invariant effective divisors on $\Mgnn.$ We can find such divisors in different ways. For one we can take divisors already defined on $\Mgnn$ or pull them back from $\Mgn$ via the forgetful map $\phi:\Mgnn \to \Mg$. In our special case , however, we have the additional map $\chi : \Mgnn \to \Mgun$ from the introduction. When pulling effective divisors back along this morphism we must check that the pull-back remains effective according to Proposition \ref{divisor}. In the case of forgetful maps this is trivial because they are surjective.

Let us begin with the Brill-Noether divisors from Lemma \ref{BN divisor}.
If neither $g+1$ nor $g+n+1$ are prime, we pull back the class of $[\overline{\mathfrak{ BN}}]$ as a $\Q$-Divisor (up to multiplication with a positive scalar) and set
\begin{equation} \label{pullback1}
B=\frac{1}{c}\phi^*[\overline{\mathfrak{ BN}}_g],  \qquad D=\frac{1}{c} \chi^*[\overline{\mathfrak{ BN}}_{g+n}].
\end{equation}

By repeatedly applying Lemma \ref{pullback} we get the following decomposition for $B$ and $D$.

\begin{lemma}
With the above notation, one has
\begin{equation}\label{B}
B=b_{\l}\l +0\cdot\psi-  b_0\dirr -\sum_{i= 1}^{\lfloor\frac{g}{2}\rfloor}\sum_S b_i\d_{i,S}
\end{equation}
with $$ b_\l=g+3,\qquad b_0=\frac{g+1}{6},\qquad b_i=i(g-i).$$ 
Furthermore
\begin{equation}\label{D}
D=d_\l\l+d_0(\psi-\dirr-\sum_{i=0}^{\lfloor\frac{g}{2}\rfloor}\sum_{a+b\leq n,b\neq0}\d_{i;a,b})-\sum_{i=0}^{\lfloor\frac{g}{2}\rfloor}\sum_{a=0}^nd_{i+a}\d_{i;a,0}
\end{equation}
with $$ d_\l=g+n+3,\qquad d_0=\frac{g+n+1}{6},\qquad d_i=i(g+n-i).$$

\end{lemma}

In particular, both divisors are $G -$invariant.
The problem is that these divisors will not exist for $g+1$ or $g+n+1$ prime. In that case we have to use the less efficient divisors parametrizing curves that violate the Gieseker-Petri condition from Lemma \ref{GP divisor}

Similar to \eqref{pullback1},  we pull back $[\overline{\mathfrak{GP}}]$ along the forgetful map $\phi$ and the gluing map $\chi$  and set
\begin{equation}
E=\frac{1}{c}\phi^*[\overline{\mathfrak{GP}}_g],\qquad F=\frac{1}{c}\chi^*[\overline{\mathfrak{GP}}_{g+n}].
\end{equation}  
By Lemma \ref{pullback}, we obtain
\begin{equation}\label{E}
E=e_{\l}\l+0\cdot\psi-e_0\dirr -\sum_{i=1}^{\frac{g}{2}}\sum_S e_i\d_{i,S},
\end{equation}
with 
$$e_{\l}=6(\frac{g}{2}+1)^2+\frac{g}{2}-5, \quad
e_0=(\frac{g}{2}+1)\frac{g}{2}, \quad
e_1=(g-1)(\frac{3g}{2}+1), \quad e_{i+1}>b_i,$$ 
and
\begin{equation}\label{F}
F=f_{\l}\l+f_0(\psi-\dirr-\sum_{i=0}^{\lfloor\frac{g}{2}\rfloor}\sum_{a+b\leq n,b\neq0}\d_{i;a,b})-\sum_{i=0}^{\lfloor\frac{g}{2}\rfloor}\sum_{a=0}^n f_{i+a}\d_{i;a,0},
\end{equation}
with 
$$f_{\l}=6(\frac{g+n}{2}+1)^2+\frac{g+n}{2}-5, \quad
f_0=(\frac{g+n}{2}+1)\frac{g+n}{2}, \quad
f_1=(g+n-1)(\frac{3g+3n}{2}+1), $$
$$f_{i+1}>\tilde{b}_i.$$

As we have mentioned, the divisor classes $B$ and $E$ are effective by Proposition \ref{divisor}, because the forgetful map $\phi$ is surjective. The divisor class $F$ is effective because the general nodal curve is Gieseker-Petri general. This is already implicitly contained in \cite{g} (see also \cite{eh1}), in the proof of the Gieseker-Petri Theorem. Since the proof uses deformation to a nodal curve, it actually shows the above statement. 
Effectiveness of $D$ follows in the same way observing that the Gieseker-Petri condition implies the Brill-Noether condition. Alternatively, one could consider the proof of the Brill-Noether Theorem in \cite{gh}, which also uses deformation to a nodal curve.

Next we use the divisors of Weierstra{\ss}-type from Lemma \ref{Logan divisor}. These divisors are not $G$-invariant, but they can be symmetrized by taking appropriate sums as we have done in Section \ref{sec effdiv}. This gives us the following divisor on $\Mgnn$: 
\begin{equation}\label{W}
\begin{split}
W =& -w_\l \l +w_\psi \psi+0\cdot\dirr-\sum_{s\geq 2} w_s \sum_{|S|=s} \d_{0,S}
\\&-\mbox{higher order boundary terms},  
\end{split}
\end{equation}
where the coefficients $w_\l, w_\psi, w_s$ are given by equation \eqref{W small} if $2n\leq g$ and by equation \eqref{W large} if $2n\geq g$.

It turns out that we will be able to improve the upper bound on $n$ given in Proposition
\ref{prop one} by replacing $W$ with the divisors $[\overline{\mathfrak{Mrc}}_{g,k}^r]$ parametrizing curves, that fail the so called {\em Minimal Resolution conjecture}, see Lemma \ref{MinRes div}.

Note that $[\overline{\mathfrak{Mrc}}_{g,k}^r]$ is already $G$-invariant. \\

For $g$ odd and $2n\geq g-1$ we can always find an $r,k$ such that $2n=m=(2r+1)(g-1)-2k$. In fact we will only  consider $2n\geq g+1$ and choose $k\leq g-2$. This will make some computations easier and uniquely determine $r,k$ for given $g,n$.
So for $g$ odd and such $r,k$ we define the divisor $U$ as the class of $\overline{\mathfrak{Mrc}}_{g,k}^r $ as a $\Q$-divisor (up to multiplication with a positive scalar)by

\begin{equation} \label{U}
U:=[\overline{\mathfrak{Mrc}}_{g,k}^r ]= -u_\l+u_\psi \psi +u_{\mathrm{irr}} \dirr -\sum_{i,s} u_{i,s} \sum_{|S|=s} \d_{i,S}
\end{equation}
with $u_\l, u_\psi, u_{\mathrm{irr}}, u_{i,s}$ being $a, c, b_{\mathrm{irr}}, b_{i,s}$ from \eqref{coef MinResConj}.

For $g$ even we set $2n-1=m=(2r+1)(g-1)-2k$ and pull back via all possible forgetful maps $\phi_i:\Mgnn\to\Mgm$, forgetting the $i$th point. Again we only consider $k\leq g-2$ corresponding to $m\geq g+1$. We emphasize that, just as above  for $U$, the integers $r$ and $k$ are uniquely determined by $g$ and $n$. This gives
\begin{equation} \label{V}
V:=\sum_{i=1}^{2n}\phi^*_i [\overline{\mathfrak{Mrc}}_{g,k}^r]= -v_\l+v_\psi \psi +v_{\mathrm{irr}} \dirr -\sum_{i,s} v_{i,s} \sum_{|S|=s} \d_{i,S}
\end{equation}
with 

\begin{equation}
\begin{split}
v_\l	=&\frac{2n}{g-2}\Big( (g-1)(g-2)(6r^2+6r+r)+k(24r+10k+10-10g-12rg)\Big), \\
v_\psi	=& (2n-1)(rg+g-k-r-1), \\
v_{\mathrm{irr}}=& \frac{2n}{g-2}\Big( \binom{r+1}{2} (g-1)(g-2)+k(k+1+2r-rg-g)\big), \\
v_{0,2} =& 2 (rg + g - k - r - 1) + (2 n - 2) (3 g - 3 + 2 rg - 2 r - 2 k) \\
v_{i,s}\geq & v_{0,2}.
\end{split}
\end{equation}

\section{Standard cases in the proof of Theorem \ref{two ngn}: Reduction to a system of inequalities} \label{system}
\setcounter{equation}{0}

As mentioned in Section 3, the moduli space $\Ngn$ is of general type, if and only if its canonical divisor $K_{\Ngn}$ is the sum of an ample divisor and effective divisors. We show this by decomposing $K=\pi^*(K_{\Ngn})$ on $\Mgnn$ as a sum of a positive multiple of $\psi$, of non-negative multiples of the $G$-invariant divisors constructed in the last section and of non-negative multiples of $\l,\dirr$ and the $\d_{i;a,b}$.

We will begin by using the divisor class $W$ to show:
\begin{proposition}   \label{prop one}
$\Ngn$ is of general type in the special cases 
$7 \leq g \leq 23,$  and\\
$n_{\mathrm{min}}(g) \leq n \leq  n_{\mathrm{max}}(g)$ 
with  $n_{\mathrm{min}}(g), n_{\mathrm{max}}(g)$  given in the following table:
$$\begin{array}{c|c|c|c|c|  c|c|c|c|c|  c|c|c|c|c|  c|c|c}
g &7&8&9&10&11&12&13&14&15&16&17&18&19&20&21&22&23\\
\hline
n_{\mathrm{min}}&9&8&8&8&6&7&6&6&6&6&5&6&4&4&3&4&1\\
n_{\mathrm{max}}&10&12&14&16&18&20&22&24&26&28&30&32&34&36&38&40&42
\end{array}$$
\end{proposition} 

Then we will use the divisor classes $U$ and $V$ to improve the upper bound $n_{\mathrm{max}}$ and extend our result to $g=5,6$.

\begin{proposition}   \label{prop two}
$\Ngn$ is of general type in the special cases 
$5 \leq g \leq 23,$  and\\
$n_{\mathrm{min}}(g) \leq n \leq  n_{\mathrm{max}}(g)$ 
with  $n_{\mathrm{min}}(g), n_{\mathrm{max}}(g)$  given in the following table:
$$\begin{array}{c|c|c|c|c|c|c|  c|c|c|c|c|  c|c|c|c|c|  c|c|c}
g &5&6&7&8&9&10&11&12&13&14&15&16&17&18&19&20&21&22&23\\
\hline
n_{\mathrm{min}}&9&9&8&8&8&8&6&7&6&6&6&6&5&6&4&4&3&4&1\\
n_{\mathrm{max}}&10&14&18&21&25&28&32&35&38&42&46&49&52&56&60&63&66&70&74
\end{array}$$
\end{proposition}

Note that the lower bound $n_{\mathrm{min}}$ differs from Theorem \ref{two ngn} only in the cases $g=5,6,7$. For $g=5$ and $g=6$ this is due to the fact that the preliminary upper bound actually lies below the lower bound.\\

\begin{proof} [Proof of Proposition \ref{prop one}] For $g+1$ and $g+n+1$ both composite, we will search for a non-negative linear combination of the divisors $B,D$ and $W$ such that 
$$K-xB-yD-zW$$ 
(possibly up to an arbitrarily small multiple $\epsilon \psi$, see equations
\eqref{1}, \eqref{2} below  and  our discussion of Case II)
becomes an effective combination of the tautological and boundary classes. (Whenever $g+1$ is prime we replace $B$ by $E$ and whenever $g+n+1$ is prime we replace $D$ by $F$.) This will translate into a system of inequalities, one for each tautological or boundary class. However it is easy to check that only the inequalities imposed by $\l,\psi,\dirr,\d_{0;10}$ and $\d_{0;02}$ are relevant. The others are automatically satisfied as soon as these 5 are. 

The table of relevant coefficients is:

\begin{equation}  \label{table}
\begin{array}{c|c|c|c|c|c}
  		& \l			&\psi			&\dirr		&\d_{0;1,0}	 & \d_{0;0,2}		\\
  		\hline 
 B	& b_\l	&0		& -b_0	& 0		& 0 	\\
 D	& d_\l	&d_0 	& -d_0  &-d_1 	& -d_0  \\
 E  & e_\l  &0 		& -e_0 	& 0 	&0 		\\
 F 	& f_\l 	&f_0 	& -f_0 	&-f_1 	& -f_0 	\\
 W	& -w_\l &w_\psi & 0 	&-w_2	&-w_2	\\
 K	&13		&1		&-2		&-3		&-2		\\
\end{array}
\end{equation}
with 
$$b_\l=g+3, \qquad b_0=\frac{g+1}{6},$$

$$d_\l=g+3, \qquad d_0=\frac{g+1}{6},\qquad d_1=g+n-1, $$

$$e_\l=6(\frac{g}{2}+1)^2+\frac{g}{2}-5, \quad
e_0=(\frac{g}{2}+1)\frac{g}{2},$$ 

$$f_\l=6(\frac{g+n}{2}+1)^2+\frac{g+n}{2}-5, \quad
f_0=(\frac{g+n}{2}+1)\frac{g+n}{2}, \quad
f_1=(g+n-1)(\frac{3g+3n}{2}+1), $$
and $w_\l, w_\psi, w_2$ from \eqref{W} and \eqref{W large} or \eqref{W small}.

Since the divisor $W$ is given by two equations we shall treat separately the (easy)  

Case I: $2n < g$ (which we shall split into $2n \leq g-2$ and $2n=g-1$) 

and the (difficult) 

Case II: $g\leq 2n$. 

Looking at the coefficients $w_\psi$ and $w_2$  of this divisor  (see \eqref{W small}and \eqref{W large}) it is easy to show that $w_2>3w_\psi$ for $2n\leq g-2$, $w_2=3w_\psi$ for $2n=g-1$ or $2n=g$ and $w_2<3w_\psi$ for $2n\geq g$. This motivates treating $2n=g-1$ separately.\\

For $2n\leq g-2$ and $g+1$ composite, we have $3w_\psi<w_2$ and therefore the optimal decomposition is $x=\frac{2}{b_0}, \quad y=0, z=\frac{3}{w_2}$. This is exactly the same decomposition Logan uses to prove $\Mgn$ being of general type. It satisfies the inequalities imposed by $\dirr$ and $\d_{0;10}$ as equalities and also satisfies the inequalities imposed by $\psi$ and $\d_{0;0,2}$. Only the inequality imposed by $\l$ remains to be checked:
\begin{equation} \label{lambda}
(\l): \qquad \frac{2b_\l}{b_0}-\frac{3w_\l}{w_2}\leq 13
\end{equation}
Whenever this is true for a given $g$ and $n$, we get the decomposition 
\begin{equation}\label{decomposition 1}
 K=\frac{2}{b_0}B+\frac{3}{w_2}W +(1-\frac{3w_\psi}{w_2})\psi +\mbox{"effective"}.
\end{equation}

For $2n\leq g-2$ and $g+1$ prime (and thus $g\geq 28$), we replace $B$ by $E$ in \eqref{lambda}.  Thus, whenever we have 
\begin{equation} \label{cond}
(\l): \qquad \frac{2e_\l}{e_0}-\frac{3w_\l}{w_2}\leq 13,
\end{equation}
 we actually get the desired decomposition 
\begin{equation}\label{decomposition 1 prime}
 K=\frac{2}{e_0}E+\frac{3}{w_2}W +(1-\frac{3w_\psi}{w_2})\psi +\mbox{"effective"},
\end{equation}  
proving that $\Ngn$ is of general type. Checking the inequality \eqref{cond} by use of a small computer program establishes the table for $2n \leq g-2$.\\

For $2n=g-1$ the above decomposition no longer works, because then $3w_\psi=w_2$, 
giving a vanishing coefficient of $\psi$. We therefore need to use the divisor $D$ for $g+n+1$ composite and $F$ for $g+n+1$ prime. Note that $g+1=2n+2$ is composite.
Whenever 
\begin{equation} \label{lambda1}
(\l): \qquad \frac{2b_\l}{b_0}-\frac{3w_\l}{w_2}< 13
\end{equation}
we can use the decomposition
\begin{equation} \label{1}
K=\frac{2}{b_0}B+\frac{2\epsilon}{d_0}D+\frac{1-3\epsilon}{w_\psi}W +\epsilon\psi +\mbox{"effective"}
\end{equation}
or
\begin{equation} \label{2}
K=\frac{2}{b_0}B+\frac{3\epsilon}{f_0}F+\frac{1-4\epsilon}{w_\psi}W +\epsilon\psi +\mbox{"effective"}
\end{equation}
for some $\epsilon >0 $ sufficiently small.

This proves Proposition
\ref{prop one} in Case I.\\

We shall now treat the  difficult Case II:   $2n\geq g$. Let us begin by assuming both $g+1$ and $g+n+1$ are composite. We start similarly to Case I, by trying to decompose the canonical class as 
$$K=xB+yD+zW+\epsilon\psi+\mbox{"effective"}$$
for some non-negative $x,y,z$ and positive $\epsilon$.

The coefficients of $W$ become especially easy:
$$w_\l=\binom{2n}{g}, \qquad w_\psi=\binom{2n-1}{g-1}, \qquad w_2=2\binom{2n-1}{g-1} +\binom{2n.-2}{g-2}.$$

We shall first show that there is a finite upper bound for $n$, 
i.e.  for fixed $g$ and $n$ sufficiently large, the conditions imposed by $\psi,\d_{0;1,0}$ and $\d_{0;0,2}$ cannot be satisfied simultaneously. 
In fact, let us consider the corresponding system of 3 linear inequalities for the variables $y,z$, read off the table \eqref{table}:
\begin{equation} \label{systemineq}
\begin{split}
(\psi):      	\hspace{1,1cm}	 d_0 y + w_\psi z       &< 1   \\
(\d_{0;1,0}): 	\qquad 		-d_1 y -w_2 z  			&\leq -3  \\
(\d_{0;0,2}): 	\qquad		-d_0 y -w_2 z			&\leq -2
\end{split}
\end{equation}

We apply   Gaussian elimination, allowing only positive multiples of the 
 inequalities in \eqref{systemineq}. Note that (for $2n \geq g$) any solution $y,z$ of 
\eqref{systemineq}  automatically is positive: $(\psi) + (\d_{0;0,2})$ implies $z>0$ and $3 (\psi) + (\d_{0;1,0]})$ then gives $y>0.$
 Then the inequality $(\psi)+ (\d_{0;0,2})$ gives 
\begin{equation}\label{lower bound z}
z > \frac{1}{w_2-w_\psi}\geq \frac{1}{2w_\psi}
\end{equation}
and inequality $d_1(\psi)+d_0(\d_{0;1,0})$ gives 
\begin{equation}
z < \frac{d_1-3d_0}{d_1 w_\psi- d_0 w_2}.
\end{equation}

This is solvable if and only if 
\begin{equation}\label{condition system ineq}
(d_1-3d_0)(w_2-w_\psi)-(d_1 w_\psi -d_0 w_2)>0.
\end{equation}

Consider
\begin{equation}
\begin{split}
p_g(n)&:=(12n-6)\binom{2n-1}{g-1}^{-1} \Big((d_1-3d_0)(w_2-w_\psi)-(d_1 w_\psi -d_0 w_2)\Big)\\
	&=-2n^2+(2g-5)n+g^2-11g+9	
\end{split}
\end{equation}
as a polynomial in $n$. Then $p_g(n)>0$ if and only if
\begin{equation}
\frac{1}{4}(2g-5-\sqrt{97-108g+36g^2}) <n< \frac{1}{4}(2g-5+\sqrt{97-108g+36g^2}).
\end{equation}

As we are only interested in solutions with  $n \in \N,$ we get

\begin{equation}
n \leq \frac{1}{4}(2g-5+\sqrt{81-108g+36g^2})= 2g-4.
\end{equation}
This establishes the existence of an upper bound for $n$, as claimed above.

Now let us assume, that $\frac{g}{2}\leq n\leq 2g-4$ and that $(y,z)$ is a solution of the three inequalities \eqref{systemineq} (preferably with z as large as possible).
 We set 
$$x=\frac{1}{b_0}(2-d_0 y)$$ 
and show that $(x,y,z)$ satisfies also the two remaining inequalities:

\begin{equation} 
\begin{split}
(\l): \qquad  b_\l x+d_\l y-w_\l z &\leq 13 \\
(\dirr): \hspace{1,8cm} -b_0 x -d_0 y   &\leq -2
\end{split}
\end{equation}

In fact $(\dirr)$ will be satisfied as equality, and we are left with checking $(\l)$.

If $g+1$ is prime we replace $B$ by $E$ and for $g+n+1$ prime we replace $D$ by $E$. As there are only finitely many cases to check, we have done this 
by explicitly computing the decomposition with the help of a simple program in Mathematica.
 This proves all the results  tabled in Proposition \ref{prop one}.
\end{proof}

In order to show Proposition \ref{prop two} we replace the divisor $W$ by either $U$ (for $g$ odd) or by $V$ (for $g$ even). 

\begin{proof}[Proof of Proposition \ref{prop two}]
As  a general remark, note that the divisors $U,V$ do not exist if $n$ is too small, depending on $g$, i.e. for $2n < g-4$. In this case, we simply use Proposition \ref{prop one} to generate the table in Proposition \ref{prop two}.

In all other cases we want to decompose $K$ as a 
sum of $xB$ or $xE$, $yD$ or $yF$, $zU$ or $zV$ with non-negative $x,y,z$, some positive multiple $\epsilon \psi$ of the point bundles and an effective combination of the tautological classes and boundary divisors.
So let us consider the table of relevant coefficients
\begin{equation}  \label{table2}
\begin{array}{c|c|c|c|c|c}
  		& \l			&\psi			&\dirr		&\d_{0;1,0}	 & \d_{0;0,2}		\\
  		\hline 
 B	& b_\l	&0		& -b_0			& 0		& 0 	\\
 D	& d_\l	&d_0 	& -d_0  		&-d_1 	& -d_0  \\
 E  & e_\l  &0 		& -e_0 			& 0 	&0 		\\
 F 	& f_\l 	&f_0 	& -f_0 			&-f_1 	& -f_0 	\\
 U	& -u_\l &u_\psi & u_{\mathrm{irr}}	&-u_{0,2}&-u_{0,2}	\\
 V	& -v_\l &v_\psi & v_{\mathrm{irr}}	&-v_{0,2}&-v_{0,2}	\\
 K	&13		&1		&-2		&-3		&-2		\\
\end{array}
\end{equation}
where $b_\l,b_0$ are given in \eqref{B}, $d_\l,d_0,d_1$ in \eqref{D}, $e_\l,e_0$ in \eqref{E}, $f_\l,f_0,f_1$ in \eqref{F}, $u_\l$, $u_\psi$,$u_{\mathrm{irr}}$,$ u_{0,2}$ in \eqref{U} and $v_\l, v_\psi,v_{\mathrm{irr}}, v_{0,2}$ in \eqref{V}.

$\l,\psi,\dirr,\d_{0;1,0},d_{0;0,2}$ each determine a linear inequality that can be read off from the table. 
Analogous to \eqref{systemineq} we begin by considering the system of inequalities 

\begin{equation} \label{systemieq odd composite}
\begin{split}
(\psi):      	\hspace{1,35cm}	 d_0 y + u_\psi z       &< 1   \\
(\d_{0;1,0}): 	\qquad 		-d_1 y -u_{0,2} z  			&\leq -3  \\
(\d_{0;0,2}): 	\qquad		-d_0 y -u_{0,2} z			&\leq -2.
\end{split}
\end{equation}

This corresponds to the case $g+1$ composite, $g+n+1$ composite and $g$ odd. It is one of 6 possible cases and will be the only one which we shall discuss in detail.

By Gaussian elimination (compare \eqref{condition system ineq})  this is solvable, if and only if
\begin{equation}
\begin{split}
0 	&\leq p(g,n,r,k):=6\Big((d_1-3d_0)(u_{0,2}-u_\psi)-(d_1 u_\psi -d_0 u_{0,2})\Big)  \\
	&=9 - 12 g + 3 g^2 + k + g k - 3 n + 3 g n + k n + r - g^2 r + n r - 
 g n r
\end{split}
\end{equation}
This inequality defines the upper cut-off in the table of Proposition \ref{prop two} (in the special case considered here), which improves the table of Proposition \ref{prop one}. We the have to show that for all these values of $g$ and $n$ we actually get a solution of our system of inequalities. This we did by computing, in all these cases, an explicit solution using Mathematica.

All other cases are treated in a similar way: For $g$ even,  the divisor $U$ has to be replaced by $V$, if $g+n+1$ is prime, one has to replace $D$ by $F$, and if $g+1$ is prime, $B$ nedds to by replaced by $E$. In each of these cases  one gets a corresponding system of linear inequalities and a neccessary upper cut-off for $n$. This upper bound defines $n_{\mathrm{max}}(g)$ in our table. 
As above, we have then computed a solution of the ensuing system of inequalities by Mathematica, for all relevant values of  $g$ and $n$.
This gives the improved table in Proposition  \ref{prop two}.

\end{proof}

\section{Proof of Theorem \ref{two ngn}: Special computations} \label{special}
\setcounter{equation}{0}

The aim of this section is to prove Theorem \ref{two ngn}. In view of Propostion \ref{prop two} this boils down to specific calculations for each of the values of $(g,n)$ contained in the table of Theorem \ref{two ngn}, but not in the table of Proposition \ref{prop two}. In each case, we shall need additional divisors, specifically adapted to the case at hand. Our standard references for these divisors are \cite{l}, \cite{f}.\\

\textbf{$\mathcal{N}_{10,6}$ and $\mathcal{N}_{10,7}$} \\
We use the divisor class $\mathcal{Z}_{10,0}$ on $\M_{10}$ from \cite{f},  Theorem 1.4. This is a divisor that violates the slope conjecture, i.e. it has a smaller slope than the Brill-Noether divisor. We have $s(\mathcal{Z}_{10,0})=7$, i.e $\mathcal{Z}_{10,0}=c( 7\l-\dirr -\sum_i a_i \d_i)$ on $\M_{10}$.
Take $Z=\phi^*(\mathcal{Z}_{10,0})$, the pull-back via the forgetful map $\phi:\M_{10,2n}\to\M_{10}$.
Then for $n=6$ we can decompose $K$ as an effective combination of $Z,F,W$,  some positive multiple of $\psi$ and an effective combination of the tautological and boundary classes.
For $n=6$ we can decompose $K$ as an effective combination of $Z,D,W$, some  positive multiple of $\psi$ and an effective combination of the tautological and boundary classes.\\

\textbf{$\mathcal{N}_{21,2}$} \\
We use the divisor class $\mathcal{Z}_{21,0}$ on $\M_{21}$ from \cite{f} Theorem 1.4. This divisor has slope $s(\mathcal{Z}_{10,0})=\frac{2459}{377}$, which is smaller than the slope of the Brill-Noether divisor. We set $Z=\phi^*(\mathcal{Z}_{21,0})$ the pull-back via the forgetful map $\phi:\M_{10,4}\to\M_{10}$.
The canonical divisor will decompose as an 
effective combination of $Z,W$ some positive multiple of $\psi$ and an effective combination of the tautological and boundary classes.\\

\textbf{$\mathcal{N}_{16,5}$} \\
Similarly to the above divisors, we use the divisor class $\mathcal{Z}_{16,1}$ on $\M_{16}$ from \cite{f}, Corollary 1.3. This divisor has slope $s(\mathcal{Z}_{10,0})=\frac{407}{61}$, which once again is smaller than the slope of the Brill-Noether divisor. We set $Z=\phi^*(\mathcal{Z}_{16,1})$, the pull-back via the forgetful map $\phi:\M_{16,10}\to\M_{16}$.
The canonical divisor will decompose as an 
effective combination of $Z,W$, some positive multiple of $\psi$ and an effective combination of the tautological and boundary classes.\\

\textbf{$\mathcal{N}_{12,6}$} \\
Here we use the divisor class $\mathcal{D}_{12}=13245\l-1926\dirr-9867\d_1-\ldots $ on $\M_{12}$ from \cite{fv4} and pull it back along the forgetful map $\phi:\M_{12,12}\to\M_{12}$.
The canonical divisor will decompose as an 
effective combination of $\phi^*(\mathcal{D}),F,W$ some positive multiple of $\psi$ and an effective combination of the tautological and boundary classes.\\

\textbf{$\mathcal{N}_{22,2}$ and $\mathcal{N}_{22,3}$} \\
For $g=22$ there are no Brill-Noether divisors, because $22+1$ is prime. Instead of using the divisor class $E$ we take a Brill-Noether divisor $\mathcal{B}_{23}=c( 26\l-4\dirr-22\d_1-\ldots) $ on $\M_{23}$. We pull this divisor back along $\chi:\M_{22,2}\to \M_{23}$ and then along all possible forgetful maps $\phi_S: \M_{22,4}\to \M_{22,2}$ or $\M_{22,6}\to\M_{22,2}$.

On $\M_{22,4}$, adding all these pull-backs, gives us the divisor class $L_{22,4}=12c(13\l+\psi-2\dirr-\frac{10}{3}\sum_{|S|=2}\d_{0,S}-\ldots)$. This divisor class alone proves the canonical divisor $K$ to be effective. By using a linear combination with $\epsilon_1 E, \epsilon_2 W$ for some suitable $\epsilon_1,\epsilon_2>0$ we can show $K$ to be big.

On $\M_{22,6}$, adding all these pull-backs, gives us the divisor class 
$$L_{22,6}=30c(13\l+\frac{2}{3}\psi-2\dirr-\frac{56}{30}\sum_{|S|=2}\d_{0,S}-\ldots).$$
Using $L_{22,6}$ and $W$ we can show $K$ to be big.\\

\textbf{$\mathcal{N}_{14,5}$} \\
We can show that $\mathcal{N}_{14,5}$ is of general type by using the divisor classes $B$ and $\mathfrak{Nfold}^1_{14,12} = c(35\l+54\psi-10\dirr-173\sum_{|S|=2}\d_{0,S}-\ldots)$ from \cite{f}, Theorem 4.9.\\

\textbf{$\mathcal{N}_{18,5}$} \\
We  take the divisor class $\mathfrak{Lin}^8_{24}=c(290\l+24\psi-45\dirr -82\sum_{|S|=2}\d_{0,2}-\ldots)$ on $\M_{18,9}$ from \cite{f}, Theorem 4.6.
We pull this divisor class back along all forgetful maps $\phi_i: \M_{18,10} \to \M_{18,9}$. Using the sum of these pullbacks $D$ and $W$, we can show that  $\mathcal{N}_{18,5}$ is of general type.\\

\chapter{On quotients of $\Mgn$ by certain subgroups of $S_n$}

 We show that certain quotients of the compactified moduli space  of 
$n-$pointed genus $g$ curves, $\MG:= \Mgn / G$, are of general type, for a fairly broad class of subgroups $G$ of the symmetric group  $S_n$ which act by permuting the $n$ marked points. The values of $(g,n)$ which we specify in our theorems are near optimal in the sense that, at least in the cases that G is the full symmetric group $S_n$ or a product $S_{n_1}\times \ldots \times S_{n_m}$, there is a relatively narrow transitional zone in which $\MG$ changes its behaviour from being of general type to its opposite, e.g. being uniruled or even unirational. 
As an application we consider the universal difference variety $\Mgnn /S_n \times S_n$.

\section{Introduction}

In this chapter we shall consider a class of quotients of $\Mgn$, the compactified moduli
space of $n$-pointed genus g complex curves, by certain subgroups of the symmetric group $S_n$ which act 
by permuting the marked points. Our aim is to analyse under which conditions such quotients are of general type or, in a complementary case, uniruled or even unirational. As usual, we do this by using the Kodaira dimension.

We were led to considering the questions addressed in the present chapter
by analysing  an analogous problem for the compactified moduli space $\ngn$
of $n-$nodal genus $g$ curves in \cite{sch2}. Here $\ngn = \Mgnn/G$ where the special group $G$ is also a subgroup of $S_{2n}$, namely the semidirect product
$G:=(\Z_2)^n\rtimes S_n$. In view of the great importance of $n-$nodal curves,
e.g. in the deformation type arguments used in the proof of the Brill-Noether theorem, 
this problem was directly motivated by geometry.
 
Our  proof in \cite{sch2}, however, let us realize that  there are some related results for {\em general} quotients of $\Mgn$ which in some aspect are {\em different} from the special case of $\ngn$. In particular, it is important  that $G:=(\Z_2)^n\rtimes S_n$ is a semidirect product and
{\em not} a product of subgroups. The main point of this chapter is to prove first results in this direction for a class of general quotients.

For  $G \subset S_n$,  we denote the quotient by this action as $ \Mgn^G := \Mgn / G$ and we suppress the subscript $(g,n)$ if we feel it unnecessary  within  a given context.  Then the natural quotients induce the chain of morphisms of schemes
\begin{equation}  \label{quotient}
\Mgn \to \Mgn^G  \to \Mgn^{S_n}  
\end{equation} 
which by subadditivity of the Kodaira dimension for base and fibre
\footnote{Technically subadditivity is only proven or conjectured for connected fibres. However it still holds in this case as can be seen directly from Corollary \ref{corr} below.}, see Conjecture \ref{subadditivity}, gives the following ordering for the Kodaira dimension
\begin{equation}
\k(\Mgn) \geq \k(\Mgn^G)  \geq \k(\Mgn^{S_n}).
\end{equation}
 Since all algebraic varieties in \eqref{quotient} have the same dimension, one gets 
 
 \begin{equation}
\Mgn^{S_n}  \quad \mbox{of general type    } \Rightarrow \Mgn^G \quad \mbox{of general type   }\Rightarrow \Mgn \quad \mbox{ of general type}
 \end{equation}
By the same argument, one gets
\begin{lemma}  \label{sub}
For any subgroup $H$ of $G$ one has
\begin{equation}
\MG \quad \mbox{of general type    } \Rightarrow \MH \quad \mbox{of general type    }
\end{equation}
\end{lemma}

We shall need the ramification divisor $R$ (see equation \eqref{R} below) of the quotient map $ \pi: \Mgn \to \MG$. Denoting by $(i \  j)$ the transposition in $S_n$ interchanging $i$ and $j$ and checking where sheets will come together,  one readily finds
\begin{equation}
R = \sum_{(i \  j) \in G} \d_{0, \{i,j\}}
\end{equation}
where the (standard) definition of the boundary divisor $\d_{0, \{i,j\}}$ is given in Section \ref{prem} below, which in particular introduces all divisors needed in this chapter.
Then the well known explicit formula for the canonical divisor $K_{\Mgn}$ gives

\begin{corollary} \label{corr}
The pullback $K_G:= \pi^*(K_{\MG})$ to $\Mgn$ is given by
\begin{equation}    \label{kg}
K_G = K_{\Mgn} - R = 13 \l + \psi - 2 \d -\d_{1,\emptyset} - \sum_{(i \ j) \in G} \d_{0, \{i,j\}}
\end{equation}
\end{corollary}

Corollary \ref{corr} clearly shows that for a subgroup $H$ of $G$ we get $K_H\geq K_G$ and thus $\k(\MH)\geq \k(\MG)$. This proves the subadditivity of the Kodaira dimension for such quotients and thus implies Lemma \ref{sub}.\\
\ 

For the sake of the reader, we shall now recall   
 that the moduli space $\Mgn$ is of general type for $g\geq 24$ or for $n\geq n_{\mathrm{min}}(g)$ given in the following table 
\footnote{We once again caution the reader that Chapter 3 is not updated to include the result of \cite{f4} that $\M_{22}$ and $\M_{23}$ are of general type.} 
(see Proposition \ref{proposit}) :
\begin{table}[h!]
$$\begin{array}{c|c|c|c|c|c|c|c|  c|c|c|c|c|  c|c|c|c|c|  c|c|c}
g &4&5&6&7&8&9&10&11&12&13&14&15&16&17&18&19&20&21&22&23\\
\hline
n_{\mathrm{min}}&16&15&16&15&14&13&11&12&11&11&10&10&9&9&9&7&6&4&4&1
\end{array}$$
\caption{ } 
\end{table}

As a corollary, one then obtains 
\begin{theorem} \label{gen}
In all cases of Proposition \ref{proposit}, if $G$ does not contain any transposition, then $\MG$ also is of general type.
\end{theorem}

 \begin{proof}
By Corollary \ref{corr}, the divisor classes $K_G$ and $K_{\Mgn}$ coincide. The moduli space $\Mgn$ is of general type if and only if  $K_{\Mgn}$ is the sum of an ample and an effective divisor and, similarly, 
$ \MG $ is of general type if and only if $K_G$ is the sum of an ample and an effective divisor which are both also $G-$ invariant. But in all cases of Table 3.1 only
$S_n-$invariant divisors have been used to show that $\Mgn$ is of general type, see
\cite{l, f, fv4}. This proves our claim.
 \end{proof}
 
In particular, this covers the case where $G$ is cyclic (and different from $\Z_2$) or the  cardinality $|G|$ is odd. It 
also covers the largest non-trivial subgroup of $S_n$, the alternating group $A_n=\mbox{kernel} \mbox{  signum} $ (in fact G contains no transpositions, if and only if it is a subgroup of $A_n$).  This has an obvious geometric interpretation: The set of $n- $ tuples of $n$ fixed different points on a genus g curve carries 
a notion of {\em orientation}: Two $n-$tuples have the same orientation, if they are mapped one to another by an even permutation. Taking the quotient by $S_n$ corresponds to passing from $n-$pointed curves to $n-$marked curves, while taking the quotient by $\cA_n$ means passing to curves marked in $n$ points with orientation. Under the first action the property of being of general type might change, but it is invariant under the second.

We remark that it is very natural to use $S_n-$invariant divisors. We expect that it is possible to do so in {\em any} case in which $\Mgn$ is of general type. 
Thus we conjecture:  If $\Mgn$ is of general type, then $\MG$ is of general type for any subgroup  $G$ of $A_n$.

If the subgroup $G$ does contain transpositions, the situation is more complicated. For large $g$, the following theorem contains an (easy) general result, while for small $g$   we shall need explicit computations with well chosen divisors.   
Collecting  from \cite{fv2}  and supplementing this by an analog proof in the cases $g=22,23$
(using the divisors of \cite{fv2})  we obtain

\begin{proposition} \label{one quotients}
The space $\MS$ (and thus, by Lemma \ref{sub}, the space $\MG$ for any subgroup $G$ of $S_n$) is of general type if
\begin{enumerate}
\item[(i)]  $g \geq 24, \quad n<g, \quad$     or
\item[(ii)]  $13 \leq g \leq 23,$   and   $ n_{\mathrm{min}}(g) \leq n \leq g-1,$ where $n_{\mathrm{min}}(g)$ is given in the following table
\end{enumerate}

\begin{table}[h!]
$$\begin{array}{c|c|c|c|c|  c|c|c|c|c|  c|c|c}
g &12&13&14&15&16&17&18&19&20&21&22&23\\
\hline
n_{\mathrm{min}}&10&11&10&10&9&9&10&7&6&4&7&1
\end{array}$$
\caption{ } \label{table mS}
\end{table}
Furthermore, 
the domain of values $(g,n)$ for which $\MS$ is of general type is {\em near optimal} since it is  known that $\M^{S_n}$ is
\begin{enumerate}
\item uniruled, if $n>g$ (for any $g$) or $g \in \{10,11\}$ with $n \neq g$
\item  unirational, if $ g<10, n \leq g$
\item  for $g \geq 12$ the Kodaira dimension $\k(\Mgg)=3g -3$ is intermediary.
\end{enumerate}
\end{proposition}

Here the result of (i) follows from weak additivity of the Kodaira dimension, while (ii) is proven in \cite{fv2} by explicit computation. The first assertion in (1) follows from Riemann-Roch, while the second is proved in \cite{fv1} which also contains (2) and (3). We briefly recall the Riemann-Roch argument:
Observe that the fibre of $\Mgn/S_{n} \to \Mg$ over a smooth curve $[C]\in \Mg$ is birational to the symmetric product $C_{n}:=\Sym^n(C):= C^n/S_n$. This can be interpreted as the space of effective divisors of degree $n$ on the curve $C$. Since the Riemann-Roch theorem implies that any effective divisor of degree $d> g$ lies in some $\mathfrak{g}^1_d$, the quotient  $\Mgn/S_{n}$ is trivially uniruled for $n> g$. 

Thus, outside the domain of values specified in Proposition \ref{one quotients}, there is just a narrow transitional band in which $\MS$ changes from general type to its opposite.

It is, however, left open by Proposition \ref{one quotients} if $\MG$ might be of general type for some values of $n>g$ if the subgroup $G$ of $S^n$ is chosen judiciously. In the following theorem we shall display a large class of groups for which this is the case.

\begin{theorem} \label{two quotients}
Fix a partition $n=n_1 + \ldots + n_m$ and let $G=S_{n_1} \times \ldots \times S_{n_{m}}$.  Then $\MGgn$ (and thus, by Lemma \ref{sub}, $\Mgn^H$ for any subgroup $H$ of $G$) is of general type if 
\begin{enumerate}
\item[(i)] $g \geq 24, \quad \max \{n_1, \ldots n_m \} \leq g-1$ or
\item[(ii)] $g \leq 23, \quad  \max \{n_1, \ldots n_m \} \leq g-2$ and
$f_m(g;n_1, \ldots n_m) \leq 13$, where $f_m$ is the function defined in equation \eqref{fm} of Section \ref{sec proof2 q} below.
\item[(iii)]  $g \leq 23, \quad  \max \{n_1, \ldots n_m \} \leq g-1$ and
$f_m(g;n_1, \ldots n_m, L_1, \ldots ,L_m) \leq 13$, where $f_m$ is the function defined in equation \eqref{fml} of Section \ref{sec proof2 q} below (depending on a choice of divisor classes 
$ L_1, \ldots ,L_m$  as described at the end of Section \ref{sec proof2 q}.
\end{enumerate}
Furthermore,  $\MG$ still has non-negative Kodaira dimension if $\max \{n_1, \ldots n_m \} \leq g$ and $f_m(g;n_1, \ldots n_m) \leq 13$.

\end{theorem} 

A geometric interpretation of this result is similar to the interpretation for $\cA_n$: The partition $P:  n=n_1 + \ldots + n_m$  induces the group 
$G=G_P=S_{n_1} \times \ldots \times S_{n_{m}}$, and the action of $G_P$ maps an $n$-pointed  genus g curve to a curve with markings in $n_1, \ldots, n_m$ (considered as an ordered $m-$ tuple) which we may call a $P -$ marked curve. Thus Theorem \ref{two quotients} states that the moduli space of $P-$ marked genus $g$ curves is of general type (if the conditions in Theorem \ref{two quotients} are satisfied).

Since there is no upper bound on the number of summands $m$ in the partition of $n$, an inspection of the defining equation for $f_m$ shows that the values of $n$ may tend to infinity, provided the subgroup $G$ is chosen appropriately. 
As in the above case for $ m=2, $ Riemann-Roch establishes a (small) transitional band beyond which $\MGgn$ becomes uniruled. We emphasize that the
existence of  this transitional band (for any {\em fixed} subgroup $G$) is different from the result for $\ngn$ proved in \cite{sch2}: 
$\ngn=\Mgnn/(\Z_2)^n\rtimes S_n$ is of general type for all values of $n$, if $g \geq 24$. This is perfectly compatible with Theorem \ref{two quotients}, since $G:=(\Z_2)^n\rtimes S_n$ with its action on the $2n$ marked points is {\em not}
given by a direct product subgroup of $S_{2n},$  as required in Theorem \ref{two quotients}. To understand this from a more conceptual point of view, however, is wide open at present. We emphasize that it is not merely the size of the group which is relevant: The alternating group $G=\cA_n$ might be taken arbitrarily large and still $\MG$ will preserve general type, while taking the quotient by much smaller groups, e.g. $G=S_{g+1}$ will turn $\MG$ to being uniruled.

As an application we consider $m=2$ and the special case $G=S_n \times S_n$. This quotient has a geometric interpretation as the universal difference variety, i.e. the fibre of the map $\MG \to \Mg$ over a smooth curve $C$ is birational to the image of the difference map $C_n \times C_n \to J^0(C), (D,E)\mapsto D-E$,  see e.g. \cite{acgh}.

Then $\M^G_{g,2n}$ is uniruled for $n>g$ (by the Riemann-Roch argument given above, combined with the fact that a product is uniruled if one factor is), while Theorem \ref{two quotients} gives the following.

\begin{proposition}\label{difvar}
The universal difference variety $\Mgnn/S_n\times S_n$ is of general type for $g \geq 24, \quad n \leq g-1$, or, in the low-genus case, if $10 \leq g \leq 23$ and
$n_{\mathrm{min}}(g) \leq n \leq g-2$ where $n_{\mathrm{min}}(g) $ is specified in the following table

\begin{table}[h!]
$$\begin{array}{c|c|c|c|c| c|c|c|c|c| c|c|c|c|c}
g &10&11&12&13&14&15&16&17&18&19&20&21&22&23\\
\hline
n_{\mathrm{min}}&7&8&8&7&7&7&6&6&7&5&4&3&5&2
\end{array}$$
\caption{} \label{table dif}
\end{table}

\end{proposition}

This corollary amplifies the results of \cite{fv3}, which considers the universal difference variety in the special case $n=\lceil \dfrac{g}{2} \rceil$. We emphasize that the result in Table 3 for $g=13$ and $n=7$ is taken from \cite{fv3}; in view of the sharp coupling between $g$ and $n$ they are able to use in this special case an additional divisor, which is not applicable in the other cases and which is not contained in our Section \ref{sec div q}. All other cases in Table 3 follow from our Theorem \ref{two quotients}.

The outline of the chapter is as follows. In Section 2 we introduce notation and some preliminary results, in Section 3 we introduce the class of divisors used in our proof. Here we basically recall, for the sake of the reader, some material from \cite{sch2}.  
In Section 4 we prove Theorem \ref{two quotients}. 
The use of a small program in computer algebra is appropriate to check our calculations.

\section{Preliminaries and Notation} \label{prem}

The aim of this section is to develop a sufficient condition for $\MG$ being of general type. This requires a basic understanding of the Picard group $\Pic( \MG)$ and an explicit description of the  boundary divisors and tautological classes on $\MG$ which we shall always consider as $G$-invariant divisors on $\Mgn$ (any such divisor descends to a divisor on $\MG$). This is almost identical to Section \ref{sec ngn}. We also refer to Section \ref{sec pic} for the Picard group of $\Mgn$.

We recall that the Hodge class $\l$ on $\Mgn$ has does not impose conditions on the marked points. Therefore it is $S_n$-invariant. This means that it gives the Hodge class $\l$ on $\MG$ (where, by the usual abuse of notation, we denote both classes by the same symbol).
 
Let us now consider the action of $G$ (for a subgroup $G$ of $S_n$) on the boundary divisor classes of $\Mgn.$ Clearly, $\delta_{\mathrm{irr}}$ is $G-$invariant. We write $\d$ for the sum of all boundary divisors and set $\d_{i,s}=\sum_{|S|=s}\d_{i,S}$. We remark that a single $\d_{i,S} $ is not $G-$invariant , but the divisor $\sum_{g \in G} \d_{i,g(S)}$, averaged by the action of $G$, obviously is. In particular $\d$ and $\d_{i,s}$ are always $G$-invariant. We shall use such an averaging in the proof of Theorem \ref{two quotients}. 

Finally we study the point bundles $\psi_i$. As for the boundary divisors any single point bundle is not $G$-invariant, but it can be symmetrized by taking $\sum_{g \in G} \psi_{g(i)}$. In particular $\psi=\sum_{i=1}^n\psi_i$ is always $G$-invariant.\\

As a first step in the direction of our sufficient criterion we need the following result on the geometry of the moduli space $\MG.$

\begin{theorem}  \label{noadcon quotients}
For any subgroup $G$ of $S_n$, the singularities of $\MG$ do not impose adjunction conditions. 
\end{theorem}

The proof 
follows the lines of the 
the proof of Theorem 1.1 in \cite{fv1}. 
We shall briefly review the argument. A crucial input is Theorem 2 of the seminal paper \cite{hm} which proves that the moduli space $\Mg$ has only canonical singularities. The proof relies on the Reid-Tai criterion: Pluricanonical forms (i.e. sections of $K^{\otimes \ell}$) extend to the resolution of singularities, if for any automorphism $\sigma$ of an object of the moduli space the so-called {\em age} satisfies  $age(\sigma) \geq 1$. The proof in \cite{fv1} then proceeds to verify the Reid-Tai criterion for the quotient of $\Mgn$ by the full symmetric group $S_n $. Here one specifically has to consider those automorphisms of a given curve which act as a permutation of the marked points. For all those automorphisms the proof in \cite{fv1} verifies the   Reid-Tai criterion. Thus, in particular, the criterion is verified for all automorphisms which act on the marked points as an element of some subgroup of $S_n$. Thus, the proof in \cite{fv1} actually establishes the existence of only canonical singularities for {\em any} quotient $\Mgn/G$ where $G$ is a subgroup of $S_n$. Clearly, this is our theorem.

By Proposition \ref{criterion gentype} Theorem \ref{noadcon quotients} implies that the Kodaira dimension of $\MG$ equals the Kodaira-Iitaka dimension of the canonical class $K_{\MG}$. In particular, $\MG$ is of general type if  $K_{\MG}$ is a positive linear combination of an ample and an effective  rational class on $\MG$.  It is convenient to slightly reformulate this result. We need

\begin{proposition} \label{psi q}
The class $\psi$ on $\Mgn$ is the pull-back of a divisor class on $\MG$ which is big and nef.
\end{proposition}
 
\begin{proof}
 Farkas and Verra have proven in  Proposition 1.2 of \cite{fv2} that the $S_{n}$-invariant class $\psi$ descends to a big and nef divisor class $N_{g,n}$ on the quotient space $\Mgn /S_{n}$. Consider the sequence of  natural projections
 $\Mgn \xrightarrow{\pi} \MG \xrightarrow{\n} \MS.$
Then $\n^*(N_{g,n})$ is a big and nef divisor on $\MG=\Mgn/G$ and $\pi^*(\n^*(N_{g,n})=\psi$. 
\end{proof}

 Now observe that the ramification divisor (class) of the quotient map $\pi:\Mgn \to \MG$ is precisely 
\begin{equation}\label{R}
R = \sum_{(i\ j) \in G} \d_{0, \{i,j\}}.
\end{equation} 
In fact ramification requires existence of a non-trivial automorphism belonging to $G$, and by standard results this only occurs in the presence of the projective line with 2 marked points that can be swapped. The non-trivial automorphism is then the transposition (of the labels) of these two marked points. 
 Furthermore, the Hurwitz formula for the quotient map $\pi$ gives
\begin{equation}\label{K q}
K:= \pi^*(K_{\MG})= K_{\Mgn} - R= 13 \l + \psi - 2 \d -\d_{1,\emptyset} -\sum_{(i\ j) \in G} \d_{0, \{i,j\}} .
\end{equation}

 We thus obtain the final form of our sufficient condition:
 If  $K$ is a positive multiple of $\psi$ + some effective $G - $ invariant divisor class on $\Mgn$, then $\MG$ is of general type.

\section{Divisors}   \label{sec div q}

In this section we introduce the relevant $S_n -$ invariant effective divisors on $\Mgn.$ First we recall the following standard result.

We shall need invariant divisors on $\Mgn$. Rather than exhibiting them directly by explicit definitions, we shall simply recall from the literature the existence of special divisors with small  {\em slope}:  If $g+1$ is not prime, then there is an effective $S_n-$invariant divisor class $D$ on $\Mg$ (of Brill-Noether type) of slope
\begin{equation}  \label{sl1}
 s(D)=6 + \frac{12}{g+1},
\end{equation} 
while for $g+1$ odd (which trivially includes the case $g+1$ being prime) there is an effective $S_n-$invariant divisor class $D$ on $\Mg$ (of Giesecker-Petri type) of slope
\begin{equation}  \label{sl2}
 s(D)=6 + \frac{14 g+4}{g^2+2g},
\end{equation} 
see \cite{eh}. For a few cases ($g=10, 12, 16, 21$) it has been shown in   \cite{fv4} (for g=12)  and \cite{f} (for the other 3 cases) that there exist special effective invariant divisors $D=D_g$  with even smaller slope, i.e.
\begin{equation} \label{sl3}
s(D_g) = 
\begin{cases}
7 & g=10 \\
6 + \frac{563}{642}  \qquad \qquad  & g=12 \\
6 + \frac{41}{61}  \qquad \qquad  & g=16 \\
6 + \frac{197}{377}  \qquad \qquad  & g=21 .
\end{cases}
\end{equation}
We shall need them in the proof of Theorem \ref{one quotients}. 

Again, we need the divisors of Weierstrass-type from Lemma \ref{Logan divisor}.
Since these are not $S_n$-invariant we use the symmetrized divisors from equation \eqref{W small} for $n\leq g$ or \eqref{W large} for  $n\geq g$ respectively. For the proof of Theorem \ref{two quotients} it will be convenient to renormalized these divisors in such a way that the coefficient of $\psi$ is equal to 1. We thus introduce the divisors $W_{g,n}$ on $\Mgn$ and calculate from equations \eqref{W small} and \eqref{W large}: 
\begin{equation}
\label{wgn}
W_{g,n}  
= a(g,n) \l  +  \psi+ 0\cdot\dirr- \sum_{s\geq 2} b_s \d_{0,s}
-\mbox{higher order boundary terms},  
\end{equation}
where  {\em higher order} denotes a positive linear combination of boundary divisors $\d_{i,S}$ with     $i \geq 1$,
\begin{equation}
a(g,n)=
\begin{cases}  

\frac{2n}{(k+1)(g+r)} & g=kn+r, r<n \\
\frac{n}{g} &  n>g    
\end{cases}
\end{equation}
$b_2=b(g,n)$  with
\begin{equation}
b(g,n)= \begin{cases}
2+ \frac{2}{n-1}\frac{r(r-1)(k+1)^2 + 2r(n-r)k(k-r)+(n-r)(n-r-1)k^2}{r(k+1)(k+2)+(n-r)k(k+1)} \qquad \qquad & g=kn+r,r<n   \\
2 + \frac{g-1}{n-1} & n>g
\end{cases}
\end{equation}
and  $b_s > b_2$ for all $s>2$.

In addition, we shall use the anti ramification divisor classes from \cite{fv2}, Section 2, to obtain (by straightforward though somewhat lenghty algebraic computation) the existence of effective divisor classes $T_g$ on $\M_{g,g-1}$ satisfying
\begin{equation}  \label{tg}
T_g = - \frac{g-7}{g-2} \lambda + \psi - \frac{1}{2g-4} \dirr - \left( 3 + \frac{1}{2g-4} \right) \delta_{0,2} + h.t.
\end{equation}
 where the higher order terms $h.t.$ denote a linear combination of all other boundary divisors with coefficients $\leq -2$.
 
Furthermore, normalizing the divisor classes in \cite{fv2}, Theorem 3.1, one obtains for $g \geq       1$ and any $1 \leq m \leq g/2$ effective divisor classes $F_{g,m}$ on $\Mgn$ (with $n=g-2m$)  satisfying
\begin{equation}   \label{fgm}
F_{g,m} = a \lambda + \psi - b_{irr} \dirr -   b_{0,2} \delta_{0,2} + h.t.,
\end{equation}
where, as above,  the higher order terms $h.t.$ denote a linear combination of all other boundary divisors with coefficients $\leq -2$ and
\begin{equation}
a= \frac{n}{n-1}\left( \frac{10m}{g-2}+\frac{1-g}{g-m} \right), \quad b_{0,2}=3+\frac{(g-n)(n+1)}{(g+n)(n-1)}, \quad b_{irr}= \frac{nm}{(g-2)(n-1)}
\end{equation}
Finally, to cover the case where $g$ and $n$ have different parity, we set $n=g-2m+1$ and pull back $F_{g,m}$ given in eqution \ref{fgm} in all possible ways to $\Mgn$ (via a forgetful map forgetting one of the marked points). Summing all these divisor classes and then normalizing gives an effective divisor class $\tilde{F}_{g,m}$ on $\Mgn$ satisfying
\begin{equation}   \label{tfgm}
\tilde{F}_{g,m} = a \lambda + \psi - b_{irr} \dirr -   b_{0,2} \delta_{0,2} + h.t.,
\end{equation}
where, as above,  the higher order terms $h.t.$ denote a linear combination of all other boundary divisors with coefficients $\leq -2$ and
\begin{equation}
a= \frac{n}{n-2}\left( \frac{10m}{g-2}+\frac{1-g}{g-m} \right), \quad b_{0,2}=3+\frac{g-n-1}{g+n-1}, \quad b_{irr}= \frac{nm}{(g-2)(n-2)}.
\end{equation}

\section{Proof of Theorem \ref{two quotients}}  \label{sec proof2 q}

For  assertion (i), recall that $\Mg $ is of general type for $g \geq 24.$ Furthermore, the generic fibre of the  canonical projection $ \MG \to \Mg$ is
$$ C^n /G  \simeq    (C^{n_1}/S_{n_1}) \times \ldots \times (C^{n_m}/S_{n_m}).$$
For $\max\{n_1, \ldots ,n_m\} \leq g-1$ each factor $C^{n_i}/S_{n_i}$ is of general type, and thus is the product $ C^n /G.$ Therefore the assertion follows from weak additivity of the Kodaira dimension.

Assertion (ii) is more complicated, involving divisors. We take  Weierstrass divisors $W_k=W_{g, n_k}$ on each $\M_{g,n_k}$ and $W=W_{g,n}$ on $\Mgn$, with coefficients $a(g,n_k),b(g,n_k)$ and $a(g,n),b(g,n)$ respectively, see Section \ref{sec div q} \eqref{wgn}.  Let $S_k$ be the set of points corresponding to the summand $n_k$ in the partition $n=\sum_{1 \leq k \leq m} n_k$, and denote by $\pi_k: \Mgn \to \M_{g,n_k}$ the forgetful map forgetting all points  except those in $S_k$. In order to calculate $\pi_k^* W_k$ we introduce some notation.

For any sets $S \subset \{1, \ldots,n\}$ we define
\begin{equation}
\d_{i,s}^{S,\ell} := \sum_{|T \cap S|=\ell, |T|=s} \d_{i,T}
\end{equation}
and denote by $\pi_S: \Mgn \to \M_{g,|S|}$ the natural forgetful map. With this notation, by the usual abuse of notation amplified in Section \ref{prem}, we have 

\begin{proposition}  \label{forgetful}

The pull-back divisors are 
$\pi^*_S(\l)=\l, \quad \pi^*_S(\dirr) =\dirr$ and
\begin{equation}
\pi^*_S(\psi) = \sum_{i \in S} \psi_i - \sum_{s=2}^{n-n_k +1} \d_{0,s}^{S,1}, 
\qquad \qquad  \pi^*_S(\d_{i,s}) = \sum_{\ell \geq 0} \d_{i,s+\ell}^{S,s}.
\end{equation}
\end{proposition}

Furthermore, observe that the labels $i,j$ belong to different components $S_k,S_\ell$ if and only if the transposition $(i\ j)$ is not in $G$. This gives:  the divisor $L:= \sum_{1 \leq k \leq m} \pi^*_k W_k$ has the decomposition
\begin{equation}  \label{L}
L= - \sum_{1 \leq k \leq m} a(g,n_k) \l + \psi -2 \sum_{(i\ j)\notin G} \d_{0,\{i,j\}} + 0 \dirr - \sum_{1 \leq k \leq m} b(g,n_k) \sum_{i,j \in S_k} \d_{0,\{i,j\}} + h.t,
\end{equation}
where $h.t.$ denotes a (higher order) sum of boundary divisors, each multiplied
with coefficients $<-2$.
In addition we consider
\begin{equation}  \label{WD}
W= -a(g,n) \l + \psi - b(g,n) \sum_{i,j} \d_{0,\{i,j\}} + h.t, \qquad D=s\l -\dirr + h.t.,
\end{equation}
where $D=D_g$ is chosen with minimal slope $s=s(g)$ (see the list of divisors with small slope in \eqref{sl1}-\eqref{sl3}) and set 
\begin{equation} \label{e}
\epsilon := \min \{ b(g,n_k) -3| \quad k \in \{1, \ldots, m \}  \mbox{  with  } n_k \geq 2 \}.
\end{equation}

Clearly, $ \epsilon >0$ if and only if $\max \{n_1, \ldots n_m\} \leq g-2$. 
Combining equations \eqref{L},  \eqref{WD}, \eqref{e} (see also \ref{wgn}) one obtains the decomposition
\begin{equation} \label{decompos}
K_G  \geq 2 D + \frac{1}{1+\epsilon} L + \frac{2 \epsilon}{b(g,n)(1+\epsilon)}W + \eta \psi, \qquad \eta:= \frac{\epsilon}{1+\epsilon}(1-\dfrac{2}{b(g,n)})>0,
\end{equation}
provided one has the inequality
\begin{equation}  \label{fm}
f_m(g,n_1, \ldots,n_m):= 2 s(g) - \frac{1}{1+\epsilon} \sum_{1 \leq k \leq m} a(g,n_k) - \frac{2 \epsilon}{b(g,n) (1+ \epsilon)} a(g,n)  \leq 13.
\end{equation}

Since by Proposition \ref{psi q} the divisor  class $\psi$ is big and nef and 
all divisors in equation \eqref{decompos} are effective and $S_n -$invariant, the proof boils down to checking the inequality \eqref{fm}. 

Note that for $\max \{n_1, \ldots, n_m\}\in \{g-1,g\}$ we get $\epsilon=\eta=0$, which proves that $K_G$ is at least effective and thus gives non-negative Kodaira dimension.

To treat the additional case $\max \{n_1, \ldots n_m \} = g-1$ in case (iii) we need more general divisors $L_1, \ldots ,L_m$. The function $f_m$ in equation \eqref{fm} will then depend on these divisors, destroying the explicit form of $f_m$ given  in equation \eqref{fm}. 

Instead of the family of (generalized) Weierstrass divisores $W_k$ on $\M_{g,n_k}$, for $1 \leq k \leq m,$ we use divisors $L_k$  on $\M_{g,n_k}$, for $1 \leq k \leq m,$ having a decomposition
\begin{equation}  \label{lk}
L_k=a_k \lambda+\psi-b_{k,irr} \dirr -b_k \delta_{0,2} + h.t.
\end{equation}
where $b_k >3$ and  the higher order terms $h.t.$ denote a linear combination of all other boundary divisors with coefficients $\leq -2$. Setting (analog to the above)  
$L:= \sum_{1 \leq k \leq m} \pi^*_k L_k$ we obtain
\begin{equation}  \label{Lnew}
L= - \sum_{1 \leq k \leq m} a_k \l + \psi - 2 \sum_{(i\ j)\notin G} \d_{0,\{i,j\}} -  
\sum_{1 \leq k \leq m} b_{k,irr} \dirr - \sum_{1 \leq k \leq m} b_k \sum_{i,j \in S_k} \d_{0,\{i,j\}} + h.t,
\end{equation}
with $h.t.$ as above. This is analog to \eqref{L}.

In this notation, we already have for shortness's sake suppressed dependence on $g,n$. Using the same convention in equation \eqref{WD} - thus simply writing $a,b$ in the decomposition of $W$ - and introducing
\begin{equation} \label{enew}
\epsilon := \min \{ b_k -3| \quad k \in \{1, \ldots, m \}  \mbox{  with  } n_k \geq 2 \},
\end{equation}
which is \eqref{e} with $b(g,n_k)$ replaced by $b_k$ and writing $\a_+:=\max\{\a,0\}$, we obtain the decomposition
\begin{equation} \label{decomposen}
K_G  \geq (2- \frac{1}{1+\epsilon}\sum_{k} b_{k,irr})_+  D + \frac{1}{1+\epsilon} L + \frac{2 \epsilon}{b(1+\epsilon)}W + \eta \psi, \quad \mbox{where} \quad \eta:= \frac{\epsilon}{1+\epsilon}(1-\frac{2}{b})>0,
\end{equation}
provided one has the inequality

\begin{equation}  \label{fml}
f_m(g,n_1, \ldots,n_m,L_1, \ldots ,L_m):= (2- \frac{1}{1+\epsilon}\sum_{k} b_{k,irr})_+ s + \frac{1}{1+\epsilon} \sum_{1 \leq k \leq m} a_k - \frac{2 \epsilon}{b (1- \epsilon)} a  \leq 13.
\end{equation}

This finishes the proof.

\chapter[The moduli space of hyperelliptic  curves with  marked points]{On the Kodaira dimension of the moduli space of hyperelliptic  curves with  marked points}

It is known  that the moduli space $\Hgn$ of genus g stable hyperelliptic curves with $n$ marked points is uniruled for   $n \leq 4g+5$.  In this chapter we consider the complementary case and show that $\Hgn$ has non-negative Kodaira dimension for $n = 4g+6$ and is of general type for $n \geq 4g+7$. Important parts of our proof are the calculation of the canonical divisor  and establishing that the singularities of 
$\Hgn$ do not impose adjunction conditions.

\section{Introduction}
 The birational geometry of the moduli spaces $\mg$ and $\mgn$ of genus $g$ curves 
and of genus $g$ curves with $n$ marked points has been studied for a long time, together with their Deligne-Mumford compactifications $\Mg$ and $\Mgn$.  A prominent question is when these spaces are unirational, i.e. explicitly describable (at least generically) by finitely many complex parameters, or on the contrary of general type, i.e. of maximal Kodaira dimension. Such investigations go back at least to the fundamental papers \cite{hm} (for $\Mg$) and \cite{l} (for $\Mgn$), establishing that these spaces are of general type if $g$ is sufficiently large or $n$ is sufficiently large, depending on $g$. Subsequently, the results of these papers have been refined by various authors, see e.g.\cite{f,f2,fv4}.
 
 On the other side, \cite{be} contains a recent summary of results (including the improvements made in that paper) for wich values of $g,n$ the moduli space $\Mgn$ is uniruled or even unirational.
 
 It is of great geometric interest to understand in a similar way the birational geometry of subvarieties of $\Mg$ and $\Mgn$. Here the known results are much less complete. We recall, however, our own result in \cite{sch2} on the moduli space $\ngn$ of $n$-nodal curves of geometric genus $g$ (which might be considered either as a subvariety of $\M_{g+n}$ or as a quotient of $\Mgnn$ by a subgroup of the symmetric group $S_{2n}$ acting on the $2n$ marked points). Probably the most  classical space in this direction, however, is the locus $\hg$ (and $\hgn$) in $\mg$
 (or $\mgn$) of hyperelliptic genus $g$ curves (with $n$ marked points). We refer to \cite{acgh, acg} for background on hyperelliptic curves and their moduli space $\hg$ (as well as the associated moduli stack).
 
 Clearly, $\hg$ always is unirational, being explicitly parametrized by 
equations $y^2= f(x)$, where $f$ is a polynomial of degree $2g+1$ or $2g+2$ with simple zeroes.  Adding marked points, this becomes less explicit. We recall, however,  that it is proved in  \cite{be}  that $\Hgn$ is uniruled for all $n \leq 4g+4$, by applying the methods of that paper  also to the subvarieties $\Hgn$ of $\Mgn$. Furthermore, we have been informed by D. Agostini and I. Barros that using additional arguments this result can actually be extended to the case $4g+5$, \cite{ab}.  In the present chapter we shall study the complementary case. Our main result is
 
 \begin{theorem}    \label{t1}
The moduli space $\Hgn$ is of general type for $n\geq 4g+7$.
For $n=4g+6$ the Kodaira dimension of $\Hgn$ is non-negative.
\end{theorem}
 
 This  
gives a sharp transition zone for passing from uniruled to general type by increasing the number of marked points for precisely the value $4g+6$.
 
 A main step in our proof is the computation of the canonical divisor of $\Hgn$. In Section 3 we shall prove

\begin{theorem}\label{canonical Hgn}
The canonical class of the stack $\stHgn$ is
\begin{equation}\label{canonical Hgn eq stack}
\begin{split}
K_{\stHgn} &= \sum_{i=1}^n \psi_i -(\frac{1}{2}+\frac{1}{2g+1})\e_0+ 
\sum_S \sum_{i=1}^{\lfloor \frac{g-1}{2}\rfloor} (\frac{(2i+2)2(g-i)}{2g+1}-2)\e_{i,S} \\ &+\sum_S\sum_{i=1}^{\lfloor \frac{g}{2}\rfloor} \frac{(2i+1)(2g-2i+1)}{4g+2} \d_{i,S}
-2\sum_{|S|\geq 2} \d_{0,S},
\end{split}
\end{equation}
and for $n\geq 2$ the canonical class of the coarse moduli space $\Hgn$ is
\begin{equation}\label{canonical Hgn eq coarse}
\begin{split}
K_{\Hgn} &= \sum_{i=1}^n \psi_i -(\frac{1}{2}+\frac{1}{2g+1})\e_0+ 
\sum_S \sum_{i=1}^{\lfloor \frac{g-1}{2}\rfloor} (\frac{(2i+2)2(g-i)}{2g+1}-2)\e_{i,S} \\ &+\sum_S\sum_{i=1}^{\lfloor \frac{g}{2}\rfloor} \frac{(2i+1)(2g-2i+1)}{4g+2} \d_{i,S}
-2\sum_{|S|\geq 2} \d_{0,S}-\sum_{i=1}^g \d_{i,\emptyset} ,
\end{split}
\end{equation}
where the sum is taken over all subsets $S\subset \{1,\ldots,n\}$
and we use $\d_{i;\emptyset}=\d_{g-i, \{1,\ldots,n\}}$.
\end{theorem}
 
Here $\psi_i$ denote the point bundles (or tautological classes) on $\Hgn$ and $\e_0, \e_{i,S}, \d_{i,S}, \d_{0,S}$ are the boundary divisors. All these divisors are introduced in Section 2 or 3.
Given Theorem \ref{canonical Hgn}, the proof of Theorem \ref{t1} then proceeds similarly to our previous paper \cite{sch2}. However, establishing that the singularities of $\Hgn$ do not impose adjunction conditions, requires a substantial amount of additional work, adapting the original arguments in the seminal paper \cite{hm} to our framework.

We emphasize that this part of our work only refers to the coarse moduli space $\Hgn$, the associated stack $\stHgn$ being smooth (see \cite{acg}, p. 388). Our complete proof, however, needs the canonical divisor on $\Hgn$, and here it seemed natural and helpful to compute on the moduli stack $\stHgn$ as an intermediary step. Technically, this allows to use universal families which are applicable to stacks or fine moduli spaces, but not to coarse moduli spaces.
 
 The outline of the chapter is as follows. In Section 2 we collect notation and preliminaries on $\Hg$ and $\Mgn$. In Section 3 we collect first results on $\Hgn$, and in Section 4 we show that its singularities do not impose adjunction conditions, see Theorem \ref{noadcon}. In Section 5 we introduce effective divisors on $\Hgn$, following \cite{l}. With these preparations, the actual proof of Theorem \ref{t1} in Section 6 becomes short.
 
\section{Preliminaries}

In this section we want to recall some well known facts about the hyperelliptic locus $\hg \subset \mg$, its compactification $\Hg\subset \Mg$ and its rational Picard group. We will use the isomorphism of coarse moduli spaces $\Hg\simeq \Mrat /S_{2g+2}$ (see \cite{al}) to compute its canonical divisor.

We begin by recalling some basic facts about hyperelliptic curves.
A hyperelliptic curve of genus $g$ is a smooth curve of genus $g$ admitting a degree 2 morphism to $\mathbb{P}^1$ which by the Hurwitz formula will be ramified in exactly $2g+2$ points. The map induces an involution called the hyperelliptic involution. The $2g+2$ ramification points, i.e. the fixed points of the hyperelliptic involution, are called Weierstra\ss \ points. A stable hyperelliptic curve is a stable curve admitting a degree 2 morphism to a stable rational curve. The induced involution is also called hyperelliptic involution. In both cases the degree 2 morphism
is unique and the (stable) hyperelliptic curve can be recovered from its $2g+2$ branch points.

We define $\hg\subset \mg$ as the locus of all (classes of) smooth hyperelliptic curves of genus $g$. We define $\Hg$ as the closure of $\hg$ in $\Mg$ which turns out to be the locus of stable hyperelliptic curves. 

Before we study the boundary of $\Hg$ we recall the boundary and tautological classes on $\Mgn$ from Section \ref{sec pic}.
We also refer to the book \cite{acg}. 
We emphasize that \cite{acg} mainly works on the moduli stacks. However, all the basic divisors which we shall soon introduce exist both on the stack and its associated coarse moduli space. When it becomes necessary we shall always indicate in notation where we are working.
 All Picard groups are taken with rational coefficients and, in particular, we identify the the Picard group on the moduli stack with that of the corresponding coarse moduli space.
We caution the reader that by a standard abuse of notation we will consistently use the same symbol for classes on different moduli spaces. \\

Now let us turn to the moduli space $\Hg$ and recall some well known facts (see \cite{acg}, Chapter XIII §8). 

The locus $\hg\subset \mg$ is a subspace of dimension $2g-1$. It is irreducible and closed in $\mg$. Therefore the boundary is $\dhg:= \Hg\setminus \hg =\Hg\cap \dmg$. We can look at the intersection of $\Hg$ with each (irreducible) component $\D_i$ of $\dmg$ independently. The components of the boundary $\dhg$ are the components of these intersections. 

For $i\geq 2$ a general curve in $\Hg\cap \D_i$ is obtained from smooth hyperelliptic curves $C_1$ and $C_2$ of genera $i$ and $g-i$ by
identifying a Weierstra\ss \ point on $C_1$ with a Weierstra\ss \ point on $C_2$.
When $i = 1$, we must take as $C_1$ a curve in $\M_{1,1}$ and attach it to $C_2$ at the marked point. When in addition $g=2$, $C_2$ must also be a curve in $\M_{1,1}$, attached to $C_1$ at the marked point. By the usual abuse of notation we denote $\Hg\cap \D_i$ as $\D_i$ and its class as $\d_i$.

The case $\Delta_{0}\cap \Hg$ is more complicated, because there are different types of nonseperating nodes, see Definition \ref{nodes}. Each of these types of nodes corresponds to a different component of $\dhg$. We will denote the boundary component of curves with a node of type $\e_i$ by $\cE_i$ and its class by $\e_i$.

For $g>2$, a general curve in $\cE_0$ is obtained from a smooth
hyperelliptic curve $C$ of genus $g-1$ by identifying two points which are
conjugate under the hyperelliptic involution. For $g=2$ we must instead take $C\in \M_{1,2}$ and identify the two marked points.
A general curve in $\cE_i$ with $i>0$ is obtained from a smooth hyperelliptic curve $C_1$ of genus $i$, a smooth hyperelliptic curve $C_2$ of genus $g-i-1$, a pair $(p_1, q_1)$ of points on $C_1$, conjugate under the hyperelliptic involution of $C_1$, and a pair $(p_2, q_2)$ of points on $C_2$, conjugate under the hyperelliptic involution of $C_2$, by identifying $p_1$ with $p_2$ and $q_1$ with $q_2$. We leave the case $i=1$ or $g-i-1=1$ to the reader.

These are the irreducible components of $\dhg$. The moduli space $\Hg$ intersects each of the Divisors $\D_i\subset \Mg $ transversally for $i\geq 1$. The Divisor $\D_0\subset \Mg$ intersects $\Hg$ transversally in $\cE_0$, but with multiplicity  2 in $\cE_i$ for $i>0$. This is due to the fact that nodes of type $\e_i$ come in pairs. We therefore get the decomposition on $\Hg$
\begin{equation}\label{dirr Hg}
\dirr= \e_0+ 2\sum_{i\geq 1} \e_i
\end{equation}

With these preparations,
we recall Theorem 8.4 from \cite{acg}, Chapter XIII:
\begin{theorem}
The rational Picard group $\mbox{Pic}(\Hg)\otimes \Q$ is freely generated by the classes $\d_i$ and $\e_i$. For the Hodge class $\l$ we have the relation
\begin{equation} \label{l Hg}
(8g+4)\l= g\e_0 +2\sum_{i=1}^{\lfloor\frac{g-1}{2}\rfloor} (i+1)(g-i)\e_i +4\sum_{i=1}^{\lfloor\frac{g}{2}\rfloor} i(g-i)\d_i.
\end{equation}
\end{theorem}

We recall that in \cite{acg} equation \eqref{l Hg} is proved on the level of stacks, but it is also valid on the level of coarse moduli spaces where we shall use it. In contradistinction, to calculate the canonical divisors, we shall carefully distinguish 
between stack and coarse moduli space.


As stated above, a smooth hyperelliptic curve $C$ admits a unique double cover $C\to\mathbb{P}^1$, the quotient by the hyperelliptic involution, with $2g+2$ simple branch points. In fact we can construct $C$ from these branch points. In other words, there is a canonical isomorphism between $\hg$, the moduli space of smooth hyperelliptic curves of genus $g$, and $\mrat/S_{2g+2}$, the moduli space of rational $(2g+2)$-pointed curves modulo the symmetric group $S_{2g+2}$. We call a curve in $\mrat/S_{2g+2}$ $(2g+2)$-marked. This isomorphism can be extended to an isomorphism $\Hg\simeq \Mrat/S_{2g+2}$ (see \cite{al} Corollary 2.5). We will use this isomorphism to study the Picard group of $\Hg$ and calculate its canonical divisor (class).

Let us look at the boundaries of both moduli spaces: the boundary (class) $\dhg$ consist of the $g$ irreducible components $\e_i$ for $i=0,\ldots, \lfloor \frac{g-1}{2}\rfloor$ and $\d_i$ for $i=1, \ldots, \lfloor\frac{g}{2}\rfloor$. On $\Mrat$ two boundary components $\d_{0,S}$ and $\d_{0,T}$ will be identified by the action of the symmetric group $S_{2g+2}$ if and only if $S$ and $T$ have the same cardinality. Therefore (by the usual abuse of notation) we denote the boundary components on $\Mrat/S_{2g+2}$ - corresponding to boundary divisors of $\Hg$ -  as $\D_{0,s}$ and their classes by $\d_{0,s}$ where $ s=2, \ldots, g+1$.

\begin{proposition}\label{isomorphism}
Under the canonical isomorphism $\phi: \Hg  \to  \Mrat/S_{2g+2}$ the boundary components $\d_i$ on $\Hg$ will correspond to $\d_{0,2i+1}$ on $\Mrat/S_{2g+2}$ and $\e_i$ will correspond to $\d_{0,2i+2}$, for all $i$, more precisely
\begin{equation}  \label{iso}
\phi^*(\d_{0,2i+2})= \e_i, \qquad  \phi^*(\d_{0,2i+1})= \frac{1}{2}\d_i
\end{equation}
\end{proposition}
\begin{proof}
We shall only sketch the proof, skipping a precise calculation of intersection multiplicities. For this we refer to \cite{hm}, Chapter 6C. However, we caution the reader that there are small mistakes, resp. misprints, in the final formula on p. 303
\footnote{In particular, the special multiplicity of $\e_0$ compared to $\e_i$ for $i >0$, which is also apparent in equation \eqref{dirr Hg}, seems to have been neglected.}. Therefore we carefully restate the result.

 For a general curve in each $\D_{0,s}$ we will construct a double cover, simply ramified in the marked points, possibly ramified in the nodes and unramified everywhere else.
 
 Let us begin with a general curve $C$ in $\D_{0,2i+1}$. The curve $C$ consists of a general $2i+1$-marked rational curve $C_1$ intersecting a general $2(g-i)+1$-marked curve rational $C_2$ in a general point. A double cover of a smooth rational curve must always be ramified in an even number of points. Therefore the cover of $C$ must consist of a double cover of $C_1$, ramified in the $2i+1$ marked points and the node, and a double cover of $C_2$ ramified in the $2(g-i)+1$ marked points and the node. This means that the preimage of $C$ under the double cover must consist of general hyperelliptic (or possibly elliptic) curves of genera $i$ and $g-i$ intersecting in Weierstra\ss \ points. Clearly this is an element of $\D_i$. One then generalizes these constructions to families of curves and, following \cite{hm}, one computes the intersection multiplicity (see \cite{hm}, equation (6.18) ) as $ \D_i . \phi^*(\D_{0,2i+1})= \frac{1}{2},$ for $i \geq 0$. This proves the second statement in \eqref{iso}.
 
 A general curve $C$ of $\D_{0,2i+2}$ consists of a general $2i+2$-marked rational curve $C_1$ intersecting a general $2(g-i)$-marked rational curve $C_2$ in a general point. A double cover of $C$ consists of  double covers of $C_1$ and $C_2$ ramified in the marked points, but not in the node, which will be general hyperelliptic (or possibly elliptic) curves of genera $i$ and $g-i-1$. These curves will intersect twice in the two conjugate points lying above the node of $C$. For $i>0$, this shows the correspondence with $\cE_i$. For $i=0$, the curve $C_1$ is rational and $C_2$ has genus $g-1$. We have to remember that a rational component meeting the rest of the curve in exactly two nodes violates stability and must be contracted. This causes a self intersection on the irreducible component of genus $g-1$. Thus $\D_{0,2}$ corresponds to $\cE_0$.
Similar to the above, one then calculates the intersection multiplicities (see \cite{hm}, equation (6.17) ) as
$$ \D_0 .  \phi^*(\D_{0,2})= 1, \qquad \D_0 . \phi^*(\D_{0,2i+2})= 2  \quad (i \geq 1).$$
In view of equation \eqref{dirr Hg} this gives the first statement in \eqref{iso} and completes the proof. 
\end{proof}

We can now calculate the canonical divisor of $\Hg$ and $\stHg$. 
 \begin{theorem} \label{canonical Hg}
The canonical divisor of the coarse moduli space $\Hg$ is 
\begin{equation}  \label{candivcoarse}
K_{\Hg}=-(\frac{1}{2}+\frac{1}{2g+1})\e_0+ \sum_{i=1}^{\lfloor \frac{g-1}{2}\rfloor} 
(\frac{(2i+2)2(g-i)}{2g+1}-2)\e_i + \sum_{i=1}^{\lfloor \frac{g}{2}\rfloor} 
(\frac{(2i+1)(2g-2i+1)}{4g+2}-1)\d_i,
\end{equation}
while 
the canonical divisor of the stack $\stHg$ is given by
\begin{equation}  \label{candivstack}
K_{\stHg}=-(\frac{1}{2}+\frac{1}{2g+1})\e_0+ \sum_{i=1}^{\lfloor \frac{g-1}{2}\rfloor} 
(\frac{(2i+2)2(g-i)}{2g+1}-2)\e_i + \sum_{i=1}^{\lfloor \frac{g}{2}\rfloor} 
\frac{(2i+1)(2g-2i+1)}{4g+2}\d_i.
\end{equation}
\end{theorem}

\begin{proof}
We start with  the canonical divisor on the coarse moduli space $\Hg$,
using the canonical isomorphism of Proposition \ref{isomorphism}. This isomorphism does not exist on the level of stacks.
We recall from \cite{km}, Lemma 3.5,  that the canonical divisor on the coarse
moduli space $\Mrat/S_{2g+2}$ is given by
\begin{equation}  \label{candivrational}
K_{\Mrat/S_{2g+2}}= -(\frac{1}{2}+\frac{1}{2g+1})\d_{0,2}+ \sum_{s=3}^{g+1} (\frac{s(2g+2-s)}{2g+1}-2)\d_{0,s}
\end{equation}  
Thus equation \eqref{candivcoarse} 
follows from Proposition \ref{isomorphism} by pullback. For computing the canonical divisor of the stack, we need the ramification divisor $R$ for the map
$$ \epsilon: \stHg \to \Hg.$$
The divisor $R$ can be read off the appropriate automorphism groups. First note that a generic element of $\Hg$ carries only the hyperelliptic involution as an automorphism. Likewise, a generic element of the boundary divisor $\e_i$ only carries the hyperelliptic involution, while the generic elements of $\d_i$ have automorphism group 
$\Z_2 \times \Z_2$ (the hyperelliptic involution acts independently on both components, since the node is a Weierstra{\ss}   point). It follows that the map $\epsilon$ is simply ramified over the boundary components $\d_i$, giving
$R=\sum \d_i.$ 
Thus equation \eqref{candivstack} follows from
$$ K_{\stHg}= \epsilon^*(K_{\Hg}) + R .$$
\end{proof}

\section{The locus of pointed hyperelliptic curves}

In this section we shall study the moduli space $\Hgn$ of n-pointed stable hyperelliptic curves of genus $g$ and calculate its canonical class.

We define $\hgn$ (and $\Hgn$) as the moduli spaces of (stable) hyperelliptic curves of genus $g$ together with $n$ distinct marked points (in the stable case nodes can not be marked.) Denoting the canonical projection by $\pi: \Mgn \to \Mg$, we get $\Hgn=\pi^{-1}(\Hg)$ and $\hgn=\pi^{-1}(\hg)\cap \mgn$. Both $\hgn$ and $\Hgn$ are irreducible of dimension $2g-1+n$.

The boundary of $\Hgn$ consist of the following irreducible components:
$\cE_{i,S}$ for $0\leq i\leq \lfloor\frac{g-1}{2}\rfloor$ and $ S\subset \{1,\ldots,n\} $, consisting of those curves in $\cE_i$ such that exactly the marked points labelled by $S$ are on the component of genus $i$;
$\D_{i,S}$ for $ 1\leq i\leq \lfloor\frac{g}{2}\rfloor$ and $S\subset \{1,\ldots,n\}$ consisting of curves in $\D_i$ such that exactly the marked points labelled by $S$ are on the component of genus $i$ and 
$\D_{0,S}:=\D_{0,S}\cap \Hgn$ for  $|S|\geq 2$, where by the usual abuse of notation we use $\D_{0,S}$ for both the divisor on $\Hgn$ and on $\Mgn$.

We denote the classes of these divisors by $\e_{i,S}$ and $\d_{i,S}$.
If $\i: \Hgn\to\Mgn$ is the inclusion, we denote the $\psi$-classes as $\psi_i:=\i^* \psi_i$.

It is known (see \cite{s}) that

\begin{theorem}
The rational Picard group of $\Hgn$ is freely generated by the $\psi$-classes and all the boundary classes.
\end{theorem}


We can now calculate the canonical classes of both the coarse moduli space $\Hgn$ and its assotiated stack $\stHgn$, i.e. prove Theorem \ref{canonical Hgn}

\begin{proof}
We begin on the level of stacks and consider the commutative diagram
\begin{center}
\begin{tikzpicture}
  \matrix (m) [matrix of math nodes,row sep=2em,column sep=4em,minimum width=2em]
  {\stHgn &  \stMgn   \\
   \stHgnm & \stMgnm        \\
   \vdots & \vdots  \\
   \stHg  &  \stMg \\         };

  \path[-stealth]
    (m-1-1) edge node [above] {$\i_n$} (m-1-2)
            edge node [right]  {$\hat{\pi}_n$} (m-2-1)
    (m-1-2) edge node [left] {$\pi_n$}(m-2-2)
    (m-2-1) edge node [above] {$\i_{n-1}$} (m-2-2)
            edge node [right]  {$\hat{\pi}_{n-1}$} (m-3-1)
    (m-2-2) edge node [left] {$\pi_{n-1}$}(m-3-2)
    (m-3-1) edge node [right]  {$\hat{\pi}_1$} (m-4-1)
    (m-3-2) edge node [left] {$\pi_1$}(m-4-2)
    (m-4-1) edge node [above] {$\i$} (m-4-2)
    (m-1-1) edge [bend right=45] node [left] {$\hat{\pi}$} (m-4-1)
    (m-1-2) edge [bend left=45] node [right] {$\pi$} (m-4-2);   
\end{tikzpicture}
\end{center}

First note that each of the squares in this diagram is Cartesian (i.e. a fibre product). This follows immediately from the fact that $\stHgn =\pi^{-1}( \stHg$) and a simple diagram chase.
Next we show that, for all $n\geq 1$, the forgetful map $\hat{\pi}_n$ is a universal family. Recall that the  universal family of a fine moduli space $\cM$ is a morphism $\cC\to \cM$ such that any family $\cX \to S$ in $\cM$ induces an isomorphism $\cX\simeq S\times_{\cM} \cC$, see e.g. \cite{hmo}. In \cite{acg} an analogous property is introduced for the stack of $\stMgn$ which has the properties of a fine moduli space. Recall further that the universal family of $\stMgnm$ is the forgetful map $\pi_n: \stMgn\to\stMgnm$. Now any family $\cX\to S$ in $\stHgnm$ is in particular a family in $\stMgnm$. Therefore we get
\begin{equation}
\begin{split}
\cX &\simeq S\times_{\stMgnm} \stMgn \simeq S\times_{\stMgnm} \stMgn \times_{\stHgnm} \stHgnm \\
&\simeq S\times_{\stHgnm} \stMgn \times_{\stMgnm} \stHgnm \simeq S\times_{\stHgnm} \stHgn.
\end{split}
\end{equation}

Next recall that in a universal family $\phi: \cC\to \cM$ the canonical divisor (class) is given as
\begin{equation}\label{rel dual}
K_{\cC}=\phi^* K_{\cM} +\omega_\phi,
\end{equation}
where $\omega_\phi$ is the relative dualizing sheaf of $\phi$ (and in our particular case it is the sheaf of relative K\"ahler differentials $\Omega_\phi$.) By \cite{h}, chapter II, Proposition 8.10 on the relative K\"ahler differentials of a fibre product we can calculate the relative dualizing sheave of the map $\hat{\pi}_n$ in the diagram above as
\begin{equation}\label{fib rel dual}
\omega_{\hat{\pi}_n}=\i_n^* \omega_{\pi_n}.
\end{equation}

In \cite{h} this identity is shown for schemes, and here we use it in its version for stacks.
Using the equations \eqref{fib rel dual} and \eqref{rel dual} for both $\stMgn$ and $\stHgn$ we can show by induction over $n$ that
\begin{equation}
\begin{split}
K_{\stHgn}=\hat{\pi}^* K_{\stHg}+ \i_n^* (K_{\stMgn}-\pi^*K_{\stMg})
		=\hat{\pi}^* K_{\stHg}+ \sum_{i=1}^n \psi_i -2\sum_{|S|\geq 2} \d_{0,S}.
\end{split}
\end{equation}

Equation \eqref{canonical Hgn eq stack} in Theorem \ref{canonical Hgn} now follows from Theorem \ref{canonical Hg} (giving the sum over $i$ on the right hand side of \eqref{canonical Hgn eq stack}) and repeated applications of Lemma \ref{pullback} (giving the sum over $S$). 

In order to compute the canonical divisor on the coarse moduli space, we consider
the map
$$ \epsilon: \stHgn \to \Hgn$$
and note that
\begin{equation} \label{coarse}
\epsilon^*K_{\Hgn}= K_{\stHgn} -R ,
\end{equation} 
where $R$ is the ramification divisor of $\epsilon$. In order to compute $R$ (for $n \geq 1$) we consider the locus $\Sigma \subset \Hgn$ of pointed curves with a non-trivial automorphism. Then the codimension 1 components of $\Sigma$ are 
\begin{itemize}
\item $\{ (C,x) \in \overline{\cH}_{g,1}| x \mbox{  a Weierstra{\ss}  point} \}, \quad n=1$
\item  $\d_{i,\emptyset}, \qquad (i=1, \ldots, g  \mbox{  and  } n\geq1, g \geq 2).$
\end{itemize}

In each case a general element has automorphism group $\Z_2$:  In the first case this group is generated by the hyperelliptic involution of the curve $C$ which acts as an automorphism of the pointed curve $(C,x)$. In the second case, for $i>1$, the non-trivial automorphism is the hyperelliptic involution on the component of genus $i$, while for $i=1$ it is the involution with respect to the node on the elliptic tail.

Ignoring the first case (which is irrelevant  for our theorem) we find 
$R= \sum_{i=1}^g \d_{i,\emptyset}.$
Thus equation  \eqref{canonical Hgn eq coarse}  follows from equations \eqref{canonical Hgn eq stack} and \eqref{coarse}, completing the proof of the theorem.
\end{proof}

\section{Singularities of $\Hgn$}

In this section we study the singularities (see Section \ref{sec sing}) of the moduli space $\Hgn$ and prove the following theorem:

\begin{theorem}  \label{noadcon}
The singularities of $\Hgn$ do not impose  adjunction conditions. 
\end{theorem}

This is well known for $g=2$ and $g=3$ (see e.g. \cite{hmo} for $g=3$). For $g\geq 4$ we will prove this theorem by showing the absence of adjunction conditions, first for $\hg$, then for $\Hg$ and finally for $\Hgn$. 

We will also use the Reid-Tai criterion \ref{Raid-Tai} and the Kodaira-Spencer theory introduced in Section \ref{sec sing}

As a first step in the proof of Theorem \ref{noadcon} we show the following proposition.

\begin{proposition} \label{sing smooth}
If $C$ is a smooth hyperelliptic curve and $\a$ an automorphism of $C$, then the action of $\a$ on $T_C \hg$ is either senior or the identity or a quasireflection of order 2.
Furthermore the last case can only occur for $g=2$, $C$ a double cover of an elliptic curve and $\a$ the associated involution.
\end{proposition}

\begin{proof}
Let $C$ be the hyperelliptic curve defined by the equation $y^2=f(x)$ and $\pi:C\to \P^1$ the quotient by the hyperelliptic involution. Then any automorphism of $C$ defines an automorphism of $\P^1$ which permutes the $2g+2$ branch points of $\pi$. In particular, this implies that the order of the automorphism can not be greater than $2g+2$. On the other hand any automorphism of $\P^1$ permuting the branch points can be lifted to two different automorphisms of $C$.
It follows from Hurwitz's formula that any finite automorphism of $\P^1$ has at least two fixed points (and if it has at least three, then it is the identity). In suitable coordinates we can always choose these two fixed points as $0$ and $\infty$ and write an automorphism of order $m$ as $x\mapsto \z x$ for some primitive $m$-th root of unity. Lifting this automorphism to $C$ gives us 
\begin{equation} \label{action}
\a: C\to C, \quad x\mapsto \z x, \quad y\mapsto \pm y. 
\end{equation} 

It follows from \eqref{serredual} that  the cotangent space $(T_C \hg)^V$ is 
given by the space of (co)invariants 
$H^0(C, \o^2_C)^\g$, where $\g$ denotes the hyperelliptic involution.
Any invariant quadratic differential on $C$ is the pullback via the quotient map $\pi:C \to \P^1$ modulo the hyperelliptic involution of a quadratic differential on
$\P^1$ with simple poles along the branch locus $D$ of $\pi$. Therefore the space $H^0(C, \o^2_C)^\g$ can also be identified with $H^0(\P^1, \o^2_{\P^1}(D))$ and with $H^0(C, \o^2_C(-R))$, where $R$ denotes the ramification divisor of $\pi$.
 Thus a basis of this space of quadratic differentials on $C$ can be read off the well known basis of holomorphic differentials on the hyperelliptic curve $C$ in the form
$$ \{ \dfrac{dx}{y}, x\dfrac{dx}{y}, \ldots, x^{g-1}\dfrac{dx}{y} \},$$
see e.g. \cite{gh2} p. 255, giving  a basis of $(T_C \hg)^V$ as
\footnote{ For completeness sake, we recall that the remaining $g-2$ non-invariant quadratic differentials on $C$ can eplicitly be written down as $x^j y^{-1} (dx)^2$ 
for $0 \leq j  \leq g-3 $, see e.g. \cite{fk}.}
\begin{equation}  \label{basis}
\{ (\dfrac{dx}{y})^2, x(\dfrac{dx}{y})^2, \ldots, x^{2g-2}(\dfrac{dx}{y})^2 \}.
\end{equation}


Therefore, using \eqref{action},  $\a$ acts on the tangent space $T_C \hg$, in the basis dual 
to \eqref{basis},  as
$$ \a=\begin{pmatrix}\z^{2} & 0 & \cdots & 0 \\0 & \z^{3} & 0 & 0 \\\vdots & 0 & \ddots & \vdots \\0 & 0 & \cdots & \z^{2g} \\\end{pmatrix}.$$
This gives 
$$\age(\a)=\dfrac{1}{m}\sum_{k=2}^{2g} \bar{k}, $$ with $0\leq\bar{k}\leq m-1$ and $\bar{k}= k \quad (\mbox{mod  }  m).$
We will calculate the age separately for different orders $m$ of $\a$. 

First, if $\a$ is the hyperelliptic involution, then it acts on the tangent space as the identity. This is the reason why $\hg\simeq \mrat/S_{2g+2}$ despite the extra automorphisms of hyperelliptic curves.

Second, if $m=2$ but $\a$ is not the hyperelliptic involution, then we get $\age(\a)=\dfrac{g-1}{2}$. Thus $\a$ is senior unless $g=2$, but in this case $\a$ is a quasireflection. Note also that such a junior $\a$ exists only when $C$ is a double cover of an elliptic curve and $\a$ is the associated involution.

Third, if $2<m\leq 2g$, then $\a$ has the eigenvalues $\z^2$ and $\z^{m-1}$. Thus we get 
$\age(\a)\geq \dfrac{2}{m}+\dfrac{m-1}{m}> 1$. 

Last, if $m>2g$, then we must have $m\in \{2g+1, 2g+2\}$ and $\a$ has the eigenvalues $\z^2$ and $\z^{2g}$. Thus we get $\age(\a)\geq \dfrac{2}{m}+\dfrac{2g}{m}\geq 1$.

\end{proof}

In particular, this gives
\begin{corollary}
The moduli space $\hg$ has (only) canonical singularities.
\end{corollary}

As a second step we study where the closure $\Hg$ may have non-canonical singularities.

For this purpose we need the infinitesimal deformation space of $\Hg$ as an explicit homology class. We recall from section \ref{sec sing} that the deformation space of $\Mg$ is just $\Ext^1_{\cO_C}(\O^1_C, \cO_C)$ (see also equation \eqref{serredual}).
Concerning the deformation space for $\Hgn$ and $\Hg$, we simply note that the hyperelliptic involution $\g$ still exists for stable hyperelliptic curves and acts in a canonical way on all of the spaces in the exact sequence \eqref{sesgext} and \eqref{exseqcomp}. Thus the infinitesimal deformation space at $[C]$ associated with $\Hgn$ is given by the $\g$-invariant elements in $\Ext^1_{\cO_C}(\O^1_C, \cO_C)$.  
It can be computed from the $\g$-invariant version of \eqref{exseqcomp},
see e.g. the closely related discussion \cite{acg},  Chapter XI, in the proof of Lemma 6.15.

\begin{proposition}\label{sing closure 1}
Let $C$ be a stable hyperelliptic curve of arithmetic genus $g\geq 4$ and let $\a$ be an automorphism of $C$ of order $n$. Then the action of $\a$ on $T_C \Hg$ is either trivial, senior or $\a$ is an elliptic tail automorphism, i.e. $C=C_1\cup C_2$ were $C_1\cap C_2=\{p\}$, genus of $C_2$ is $1$ and $\a$ acts trivially on 
the deformation space of $C_1$. 

Furthermore, in the last case we have either 
\begin{enumerate}
\item $C_2$ is elliptic or rational with one node and $\a$ is the inverse with respect to $p$,
\item $j(C_2)=0$ and $\a|_{C_2}$ is one of the two automorphisms of order $6$ that fix $p$ or
\item $j(C_2)=12^3$ and $\a|_{C_2}$ is one of the two automorphisms of order $4$ that fix $p$.
\end{enumerate}
 
\end{proposition}

\begin{proof}
This is the analogue for $\Hg$ instead of $\Mg$ of Theorem 2 in \cite{hm}.  We will not try to completely recreate the proof. Instead, we will briefly summarize the main idea of the proof and explain, why it still works in our case.

As in \cite{hm}, our proof proceeds by induction on the number of double points. Instead of the exact sequence \eqref{exseqcomp} we shall use its $\g$-invariant version, with $\g$ the hyperelliptic involution,  to compute $\age(\a)$ where $\a$ is considered as a map on $T_C \Hg$.

With this modification, we follow the arguments in \cite{hm}, pp. 33 - 36, line by line and find: One is reduced to the case where $\a$ fixes every component of $C$. Then every normalized component $C_{a}$, with nodes in the points $p_b$, contributes to $\age(\a)$ the eigenvalues of $\a$
on the $\g$-invariant sections
$$  H^0(\cO_{C_{a}}(2K_{a} + \sum_{b} p_{b}))^{\g}.$$

But, by Proposition \ref{sing smooth}, for any hyperelliptic normalized component  $C_a$, on which $\a$ does not act as the identity, the eigenvalues of $H^0(\cO_{C_{a}}(2K_{a} ))^{\g}$ suffice to make $\a$ senior, unless $C_a$ is a double cover of an elliptic curve and $\a$ is the associated involution.

Only this case and the case of elliptic components needs further consideration. However $\Hg$ does not impose any additional conditions on components of genus $g_a \leq 2$ compared with $\Mg$. Therefore the computations in \cite{hm} pp. 36-40 prove our Proposition \ref{sing closure 1}.

\end{proof}

As a third step we now use the proof of Theorem 1 in \cite{hm} to conclude 
\begin{proposition}\label{sing closure 2}
The singularities of $\Hg$ do not impose adjunction conditions. 
\end{proposition}

\begin{proof}
This is the analogue for $\Hg$ instead of $\Mg$ of Theorem 1 in \cite{hm}. The proof uses Theorem 2 of \cite{hm}. In the remaining cases Harris and Mumford individually construct the extensions of the pluricanonical forms.

In our case, Proposition \ref{sing closure 1} shows that we have the same exceptions as Harris and Mumford do. Furthermore, inspecting the boundary divisors of $\Hg$, it follows that $\Hg$ does not impose any additional conditions on elliptic tails compared to $\Mg$. Therefore the calculations in \cite{hm}, pp. 40-44, restricted to $\g-$invariant pluricanonical forms, work for our case as well and finish the proof. %
\end{proof}

As the final step in the proof of Theorem \ref{noadcon} we consider pointed curves in $\Hgn$. This is essentially the analogue of Theorem 2.5. in \cite{l}. We shall quickly summarize the proof. Let $(C;x_1,\ldots,x_n)$ be a pointed hyperelliptic curve and $\a$ an automorphism of this curve. Than $\a$ is just an automorphism of $C$ that fixes all the marked points. Now the deformation space of $C$ can be embedded into the deformation space of the pointed curve. 

Observe also that contracting a $\P^1$-component (in the case of stable reduction after removing marked points) does not alter the action of $\a$. Therefore the age of $\a$ as an automorphism of the pointed curve is at least as great as the age of $\a$ as an automorphism of $C$. By the Reid-Tai criterion and Proposition \ref{sing closure 1} there are only canonical singularities unless $C$ has an elliptic tail as in Proposition \ref{sing closure 1}. 
Let us assume we are in this exceptional case. If at least one of the marked points lies on the elliptic tail, then the action  of $\a$ on this component is the same as if another component were attached at this marked point. The computations in \cite{hm} on p. 36 show that the singularity will be canonical. If none of the marked points lies on the elliptic tail, then the pluricanonical forms can be extended by the calculations in \cite{hm} pp. 40-44. This finishes the proof of Theorem \ref{noadcon}.

\section{Effective divisors}

In this section we construct effective divisors on $\Hgn$ via pullback from $\Mgn$. We recall from Proposition \ref{divisor} that an effective divisor $D$ on $\Mgn$ defines an effective divisor on $\Hgn$ if and only if $\Hgn$ is not contained in $D$.

Furthermore, we recall from Lemma \ref{Logan divisor} the effective divisors $\mathfrak{D}(g;a_1,\ldots, a_n)$ on $\mgn$ defined as the set of all pointed curves $(C,x_i,\ldots, x_n)$ carrying a $\mathfrak{g}^1_g$ through the divisor $\sum_{i=1}^n a_ix_i$. In particular $\mathfrak{D}(g;g)$ is the well known Weierstra\ss \ divisor. We shall now show that these divisors impose a condition on the marked points, not on the curves, and therefore define an effective divisor on $\hgn$.

\begin{proposition}
The hyperelliptic locus $\hgn$ is not contained in any $\mathfrak{D}(g;a_1,\ldots, a_n)$.
\end{proposition}
\begin{proof}
Take any pointed hyperelliptic curve $(C,x_1,\ldots, x_n)$ and set $D:=\sum a_i x_i$. Then $C$ will carry a $\mathfrak{g}^1_g$ through $D$ if and only if $h^0(D)\geq 2$. By Riemann-Roch this is equivalent to $h^0(K-D)\geq 1$. In other words, there must be an effective divisor $D'$ of degree $g-2$ such that $D+D'\sim K$. However, any effective canonical divisor on a hyperelliptic curve consists of $(g-1)$ pairs of points conjugate under the hyperelliptic involution. This means that there must be either indices $i\neq j$ with $x_i$ and $x_j$ conjugate or some $i$ with $a_i\geq 2$ and $x_i$ a Weierstra\ss \  point. Thus only curves with such special choice of marked points are contained in $\mathfrak{D}(g;a_1,\ldots, a_n)$ and not all of $\hgn$.
\end{proof}

Since we are only interested in the cases $n>g$ we take the symmetrized pull back of $\mathfrak{D}(g;1,\ldots, 1)$ to $\Mgn$ from equation \eqref{W large}. In order to avoid further calculations with binomial coefficients we renormalize the divisor in such a way that the coefficient of $\psi$ becomes $1$, i.e. we divide equation \eqref{W large} by $\binom{n-1}{g-1}$. We get

\begin{equation}\label{W norm}
W=-\frac{n}{g}\l +\sum_{i=1}^n  \psi_i -0\cdot\dirr -\sum_{s\geq 2}b_{0,s}\d_{0,s}
-\mbox{higher order terms},
\end{equation}
with $b_{0,2}=2+\frac{g-1}{n-1}$, $b_{0,n}=\frac{n(g+1)}{2}$ and $b_{0,s}\geq b_{0,2}$ for all $s>2$. 

By abuse of notation we write $W=\iota^*_n W$, where $\iota_n$ is the inclusion $\iota_n: \Hgn\to \Mgn$.

\section{Proof of Theorem \ref{t1}}
We shall prove the theorem by using the Proposition \ref{criterion gentype} relating positivity properties of the canonical divisor with the Kodaira dimension:
If the canonical divisor of $\Hgn$ is effective, then $\Hgn$ has  non-negative Kodaira dimension. If the canonical divisor is big, i.e. the sum of an ample and an effective divisor, then $\Hgn$ is of general type. 
Validity of these criteria is based on the fact that the singularities of $\Hgn$ do not impose adjunction condition as stated in Theorem \ref{noadcon}.

Note that the divisor class $\psi=\sum_{i=1}^n\psi_i$ is ample on $\Hgn$ because it is ample on $\Mgn$. Thus it suffices to decompose $K=K_{\Hgn}$ as the sum a positive multiple of $\psi$ and some effective divisor. We will use the divisor $W$ introduced in Section 5 and show that 
$$K=\epsilon \psi +aW +E$$ for some $a,\epsilon > 0$ and $E$ effective.

We begin by taking the decomposition \eqref{l Hg} and pulling it back to $\Hgn$. We get
\begin{equation}\label{l Hgn}
(8g+4)\l=g\e_0+2\sum_S\sum_{i\geq 1} (i+1)(g-i)\e_i +\sum_S \sum_{i\geq 1} i(g-i)\d_{i,S}.
\end{equation}

Likewise we pull back \eqref{dirr Hg} to get
\begin{equation} \label{dirr Hgn} 
\dirr= \e_0 +2\sum_S \sum_{i\geq 1} \e_i.
\end{equation}

Now we set $a:= (1-\epsilon)$ and combine equations \eqref{l Hgn} and \eqref{dirr Hgn} with \eqref{canonical Hgn eq coarse} and \eqref{W norm} to show that $E:=K-(1-\epsilon)W-\epsilon \psi$ becomes effective for some sufficiently small $\epsilon$. We do this by decomposing $E$ as a linear combination of the $\psi_i$, $\e_0$, $\e_{i,S}$ with $i\geq 1$ and $\d_{i,S}$ with $i\geq 0$ and looking at each coefficient individually.

Clearly, in the decomposition of $E$, the coefficients of $\e_{i,S}$ and $\d_{i,S}$ with $i\geq 1$ are all positive.
In fact, a short computation shows that the coefficients in the decomposition of $K$ - given in \eqref{canonical Hgn eq coarse} - are all positive and the coefficients of $W$ - given in \eqref{W norm} - are all negative. Thus   we are subtracting a negative number from a positive one.

The coefficient of $\d_{0,S}$, in the decomposition of $E$, with $|S|=s\leq n-1 $ is 
$$-2+(1-\epsilon) \cdot b_{0,s}\geq -2+(1-\epsilon)\cdot b_{0,2}= -2+(1-\epsilon)(2+\frac{g-1}{n-1})$$
which is positive for $\epsilon$ sufficiently small.

We consider separately the coefficient  of $\d_{0,\{1, \ldots, n\}}=\d_{g,\emptyset}$ which is
$$ -3 + (1- \epsilon) b_{0,n} = -3 + (1- \epsilon) \frac{n(g+1)}{2} > 0, $$
for $g \geq 2$ and $n>g$.

The problematic case is $\e_0$ which only appears as part of $\l$ in $W$. Its coefficient in the decomposition of E is
$$ -(\dfrac{1}{2}+\dfrac{1}{2g+1})+(1-\epsilon)\frac{n}{g}\cdot \frac{g}{8g+4} $$
which vanishes for $n=4g+6$ and $\epsilon=0$. Thus it is positive for $n\geq 4g+7$ and $\epsilon$ sufficiently small.

We have thus shown that $K$ is big for $n\geq 4g+7$ and still effective for $n=4g+6$. This proves the theorem in view of the criteria stated above.


\newpage
{\Large Acknowledgements}\\

Most importantly, I want to thank my advisor G. Farkas both for suggesting the interesting subjects treated in this thesis
and his efficient support.
I also want to thank Will Sevin for helpful comments on the paper \cite{sch2} which constitutes an earlier version of Chapter 2 of this thesis
and my colleagues  Ignacio Barros and Daniele Agostini for treating the case $n=4g+5$ for the moduli space $\Hgn$ closing a gap between my result and the known results on uniruledness.
Last but not least I thank my family for their support, in particular 
my sister Iris for checking my English and my father for continuous encouragement to extend the Introduction to this thesis a bit beyond the bare minimum.\\

\vspace{5cm}

{\Large Selbstst\"andigkeitserkl\"arung}\\
\bigskip\\
Ich erkl\"are, dass ich die vorliegende Arbeit selbstst\"andig und nur unter Verwendung der angegebenen Quellen und Hilfsmittel angefertigt habe und zum ersten Mal eine Dissertation in diesem Studiengang einreiche.
\bigskip

Berlin, den 29.04.2020

\end{document}